\documentclass[custombib,a4paper,custommargin,times,numbered,print,index]{UCUPThesis}

\input{preamble.tex}

\title{Some aspects of descent theory and applications}

\author{Rui Rodrigues de Abreu Fernandes Prezado}

\crest{\begin{tabular}[t]{cp{8mm}c}
\parbox[c]{7.5cm}{\includegraphics[width=0.48\textwidth]{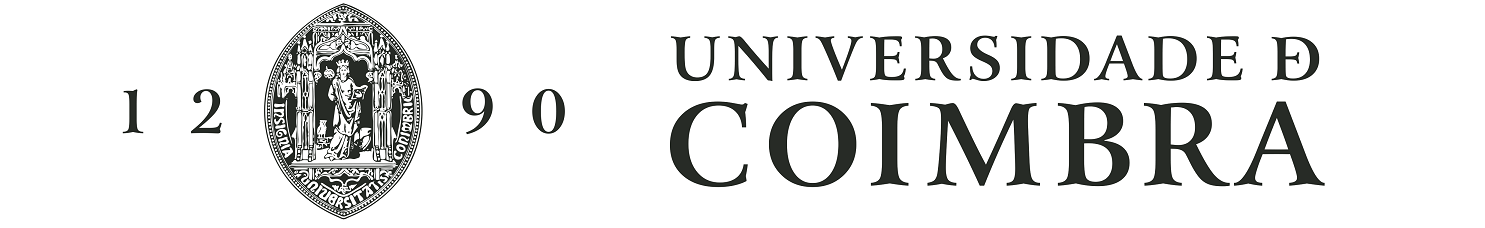}} && \parbox[c]{7cm}{\includegraphics[width=0.32\textwidth]{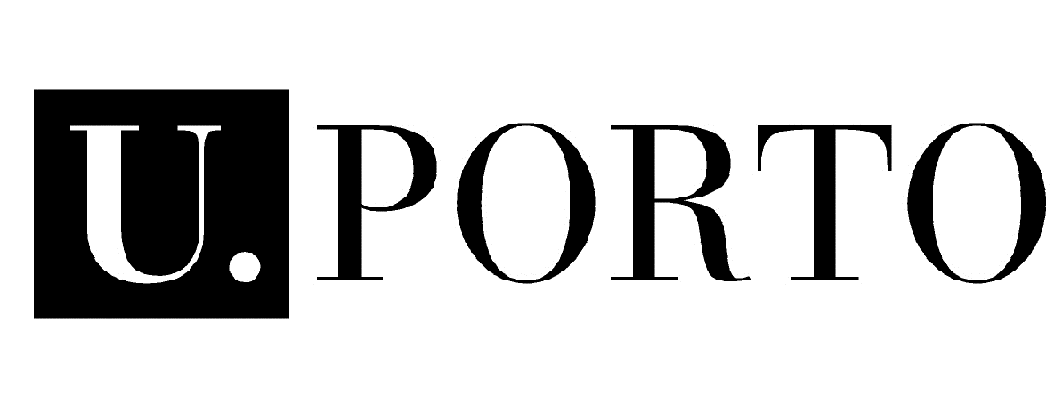}}
\end{tabular}\\[8mm]}

\degree{UC|UP Joint PhD Program in Mathematics \\[3mm] Programa
Interuniversitário de Doutoramento em Matemática \\[15mm]PhD Thesis |
Tese de Doutoramento}

\degreedate{September 2023}

\ifdefineAbstract
\fi

\ifdefineChapter
\fi

\begin{document}


\frontmatter

\begin{titlepage}
  \vspace*{30mm}
  \maketitle
\end{titlepage}


\begin{acknowledgements}

  I am forever indebted to Maria Manuel Clementino and Fernando Lucatelli
  Nunes. Not only for your guidance and your teachings, but also for all the
  opportunities you have provided for me through this experience, as well as
  providing encouragement and support.

  I would like to extend this gratitude to the jury members of my Ph.~D.
  defense: Jorge Picado, Marcelo Fiore, Peter Gothen, Dirk Hofmann, Diana
  Rodelo, Manuela Sobral. Thank you for reading the thesis with care, and for
  the fascinating discussion that followed. 

  I am also very thankful for the opportunity to co-author with Lurdes Sousa,
  and I hope this bears future occasions for collaboration.

  I would like to thank the people at the Centre of Mathematics of the
  University of Coimbra (CMUC), particularly the Algebra, Logic and Topology
  research group which taught me many things about research in category
  theory.

  For being kindly patient and very helpful, I would like to thank Rute
  Andrade. 

  For warmly hosting me during my stay at UCLouvain, a special thank you to
  Marino Gran and Tim Van der Linden.

  To André, Carlos and Eli, I thank each of you for the valuable friendships,
  and for all the conservations we have had about mathematics and other
  general nonsense.

  Thanks to Mercedes and Silvana, I was supplied plenty of cold beers and hot
  coffee. I am grateful for all the humourous moments at the ``Bar de
  Matemática'' we had, and I hope more are to come.

  Finally, I would like to thank my family for always being supportive and
  present. In particular: my parents, Maria de Fátima and Jaime, thank you
  for nurturing my passion for mathematics; my partner, Liliana, thank you
  for your endless love and patience, which I will always treasure; my
  companions, Boris and Maria.

  The financial supports from Fundação para a Ciência e Tecnologia (FCT)
  through the grant PD/BD/150461/2019, and from Centro de Matemática da
  Universidade de Coimbra (CMUC) funded by the Portuguese Government through
  FCT/MCTES are gratefully acknowledged. 

  The support from the Oberwolfach Foundation Fellowship, funded by the
  Mathematisches Forschunginstitut Oberwolfach (MFO), as well as the
  opportunity for the research stay at MFO, are gratefully acknowledged as
  well.

\end{acknowledgements}

\begin{abstract}
  This thesis is an exposition of the author's contribution on  effective
  descent morphisms in various categories of generalized categorical
  structures. It consists of: Chapter 1, where an elementary description of
  descent theory and the content of each remaining chapter is provided,
  supplemented with references; Chapter 2, consisting of various descent
  theoretical definitions and results employed in the remainder of this work;
  four chapters, each corresponding to an article written by the author during
  the period of his PhD studies.

  In Chapter 3, we describe conditions for which a $ \mathcal V $-functor 
  is an effective descent morphism in the category $ \mathcal{V} \text{-}
  \mathsf{Cat} $ of $ \mathcal V $-categories, where $ \mathcal V $ is a
  cartesian monoidal category with finite limits. Since these conditions rely
  on understanding (effective) descent morphisms in the free coproduct
  completion $ \mathsf{Fam}(\mathcal V) $ of the category $ \mathcal V $, we
  also carried out a study of such morphisms. We show how these results may be
  applied to describe the effective descent $ \mathcal V $-functors for the
  categories $ \mathcal V = \mathsf{CHaus} $ of compact Hausdorff spaces and $
  \mathcal V = \mathsf{Stn} $ of Stone spaces. The main reference of this
  chapter is the single-authored article \textit{On effective descent
  $\mathcal V$-functors and familial descent morphisms}, published in the
  \textit{Journal of Pure and Applied Algebra, vol. 228, n. 5, 2024}.

  We study effective descent morphisms for generalized multicategories
  internal to a category $ \mathcal V $ with finite limits in Chapter 4,
  proposing two approaches to obtain their description. The first approach
  relies on depicting the category $ \mathsf{Cat}(T,\mathcal V) $ of $ T
  $-categories internal to $ \mathcal V $ as a 2-dimensional limit, which
  provides a method of studying their effective descent morphisms. The second
  approach extends Ivan Le Creurer's techniques on internal categories to the
  setting of generalized internal multicategories. As a consequence of this
  work, we provide conditions for functors between internal multicategories to
  be of effective descent, as well as for functors between internal graded
  categories (by an internal monoid), internal operadic multicategories and
  ``enhanced'' multicategories. The main reference for this chapter is the
  article \textit{Descent for internal multicategory functors}, published in
  \textit{Applied Categorical Structures, vol. 31, n. 11, 2023}, with Fernando
  Lucatelli Nunes.

  Furnished with the results for effective descent morphisms in internal
  generalized multicategories, Chapter 5 aims to extend these results to the
  setting of \textit{enriched} generalized multicategories -- the so-called $
  (T, \mathcal V)$-categories. This is accomplished by extending the embedding
  of ``enriched $ \to $ internal'' categories to the setting of generalized
  multicategories, via a broad notion of \textit{change-of-base} for
  generalized categorical structures, which we specialize to our setting. We
  discuss the conditions under which the embedding $ (\overline{T}, \mathcal
  V)\text{-}\mathsf{Cat} \to \mathsf{Cat}(T,\mathcal V) $ exists and whether
  it reflects effective descent morphisms. Finally, we show these results can
  be applied to the enriched counterparts of the multicategories considered in
  Chapter 4.  More precisely, we obtain descriptions of the effective descent
  functors between enriched multicategories, enriched graded categories,
  enriched operadic multicategories, and the discrete counterparts to the
  ``enhanced'' multicategories. The main reference for this chapter is the
  article \textit{Generalized multicategories: change-of-base, embedding and
  descent, arXiv:2309.08084, DMUC preprints 23-29}, under review, with
  Fernando Lucatelli Nunes.

  Chapter 6 considers the techniques used by Sobral to study effective descent
  functors with respect to the fibration of discrete opfibrations under a new
  perspective. More specifically, we first highlight the relationship between
  the Cauchy completion of $ \mathcal V $-enriched categories and the $
  \mathcal V $-fully faithful lax epimorphisms: the latter are precisely those
  $\mathcal V$-functors that induce an equivalence on the Cauchy completions.
  Second, we show that the study of effective descent functors with respect to
  a suitable pseudofunctor $ \mathsf{Cat}^\mathsf{op} \to \mathsf{CAT} $ can
  be simplified via formal methods. Combining these two ideas, we confirm that
  Sobral's characterization can be extended, showing the same conditions
  also characterize the effective descent morphisms with respect to the
  fibration of \textit{split opfibrations}. The main reference for this
  chapter is the article \textit{Cauchy completeness, lax epimorphisms and
  effective descent for split fibrations}, published in \textit{Bulletin of
  the Belgian Mathematical Society -- Simon Stevin, vol. 30, n. 1, 2023}, with
  Fernando Lucatelli Nunes and Lurdes Sousa.

  \cleardoublepage
  \setsinglecolumn
  \chapter*{\centering \Large Resumo}
  \thispagestyle{empty}

  Esta tese é uma exposição das contribuições do autor sobre morfismos de
  descida efetiva em várias categorias de estruturas categoriais
  generalizadas. É consistido por: Capítulo 1, onde é fornecida uma descrição
  elementar sobre a teoria de descida e o conteúdo dos demais capítulos,
  complementado com referências bibliográficas; Capítulo 2, composto por
  várias definições e vários resultados da teoria de descida empregues na
  restante obra; quatro Capítulos, correspondendo a cada artigo escrito pelo
  autor durante o período dos seus estudos doutorais.

  No Capítulo 3, descrevemos condições para que um $ \mathcal V $-functor seja
  de um morfismo de descida efetiva na categoria $ \mathcal{V} \text{-}
  \mathsf{Cat} $ de $ \mathcal V $-categorias, onde $ \mathcal V $ é uma
  categoria monoidal cartesiana com limites finitos. Como estas condições
  dependem do conhecimento dos morfismos de descida (efetiva) no completamento
  livre $ \mathsf{Fam}(\mathcal V) $ da categoria $ \mathcal V $ para
  coprodutos, também se desempenhou um estudo de tais morfismos. Demostramos
  como estes resultados podem ser aplicados para descrever os $ \mathcal V
  $-functores de descida efetiva para as categorias $ \mathcal V =
  \mathsf{CHaus} $ de espaços de Hausdorff compactos e $ \mathcal V =
  \mathsf{Stn} $ de espaços de Stone. A referência principal deste capítulo é
  o artigo de autoria única \textit{On effective descent $\mathcal V$-functors
  and familial descent morphisms}, publicado no \textit{Journal of Pure and
  Applied Algebra, vol. 228, n. 5, 2024}.

  Estudamos morfismos de descida efetiva para multicategorias generalizadas
  internas a uma categoria $ \mathcal V $ com limites finitos no Capítulo 4,
  propondo duas abordagens para obter a sua descrição. A primeira abordagem
  recorre a uma descrição da categoria $ \mathsf{Cat}(T,\mathcal V) $ de $ T
  $-categorias internas a $ \mathcal V $ como um limite de dimensão 2, que
  proporciona um método para o estudo dos seus morfismos de descida efetiva.
  A segunda abordagem extende as técnicas de Ivan Le Creurer para a descrição
  de morfismos de descida efetiva em categorias internas para o contexto das
  multicategorias internas generalizadas. Como consequência destas descrições,
  fornecemos condições para que functores entre multicategorias internas sejam
  de descida efetiva, tal como functores entre categorias graduadas internas
  (por um monóide interno), multicategorias operádicas internas, e
  multicategorias ``aprimoradas''. A referência principal para este capítulo é
  o artigo \textit{Descent for internal multicategory functors}, publicado em
  \textit{Applied Categorical Structures, vol. 31, nº 11, 2023}, com Fernando
  Lucatelli Nunes.

  Munido com os resultados sobre morfismos de descida efetiva em
  multicategorias generalizadas internas, o Capítulo 5 pretende extender estes
  resultados para o contexto das multicategorias generalizadas
  \textit{enriquecidas} -- as ditas $(T,\mathcal V)$-categorias. Isto foi
  concretizado através de uma extensão da imersão de categorias ``enriquecidas
  $ \to $ internas'' para o ambiente das multicategorias generalizadas,
  através de uma noção abrangente de \textit{mudança de base} para estruturas
  categoriais generalizadas, que especializamos para o nosso contexto. São
  também discutidas as condições sob as quais é possível levar a cabo uma tal
  extensão, e quando é, se $ (\overline{T}, \mathcal V)\text{-}\mathsf{Cat}
  \to \mathsf{Cat}(T,\mathcal V) $ reflete morfismos de descida efetiva.
  Finalmente, demostra-se que estes resultados podem ser aplicados às
  multicategorias enriquecidas associadas às consideradas no Capítulo 4. Mais
  precisamente, obtemos descrições para os morfismos de descida efetiva entre
  multicategorias enriquecidas, categorias graduadas enriquecidas,
  multicategorias operádicas enriquecidas, e os análogos discretos das
  multicategorias ``aprimoradas''.  A referência principal para este capítulo
  é o artigo \textit{Generalized multicategories: change-of-base, embedding
  and descent, arXiv:2309.08084, pré-publicações DMUC 23-29}, sob revisão, com
  Fernando Lucatelli Nunes.

  O Capítulo 6 considera as técnicas utilizadas por Sobral no estudo de
  morfismos de descida efetiva em relação ao fibrado dos opfibrados discretos
  sob uma nova perspetiva. Mais especificamente, realçamos, em primeiro lugar,
  a relação entre o completamento de Cauchy para categorias enriquecidas em $
  \mathcal V $ e os epimorfismos lassos $ \mathcal V $-plenamente fiéis: estes
  últimos são precisamente os $\mathcal V$-functores que induzem uma
  equivalência nos completamentos de Cauchy. Em segundo lugar, mostramos que o
  estudo de morfismos de descida efetiva em relação a um pseudofunctor 
  $ \mathsf{Cat}^\mathsf{op} \to \mathsf{CAT} $ pode ser simplificado através
  de métodos formais. Combinando estas duas ideias, confirmamos que a
  caracterização de Sobral pode ser extendida, mostrando que as mesmas
  condições também caracterizam os morfismos de descida efetiva em relação ao
  pseudofunctor de \textit{opfibrados cindidos}. A referência principal para
  este capítulo é o artigo \textit{Cauchy completeness, lax epimorphisms and
  effective descent for split fibrations}, publicado em \textit{Bulletin of
  the Belgian Mathematical Society -- Simon Stevin, vol. 30, nº 1, 2023}, com
  Fernando Lucatelli Nunes e Lurdes Sousa.

\end{abstract}


\tableofcontents



%
%






\mainmatter

\chapter{Introduction} 
\label{chap:intro}

\textit{Descent theory} was first established in \cite{Gro60, Gir64, Dem70,
Gro71} in the context of algebraic geometry, aiming to generalize the solution
of the following problem: describe the commutative ring homomorphisms \( R \to
S \) for which the extension-of-scalars functor \( \RMod \to \SMod \) is
well-behaved. A more recent account of this problem, studied in a broader
context, can be found in \cite{JT04}.

Descent theory has since found various applications and connections with
other areas of mathematics, namely:
\begin{itemize}[label=--,noitemsep]
  \item
    the theory of monads \cite{BR70}, \cite{Luc21}, \cite{Luc22},
  \item
    two-dimensional limits and coherence, \cite{Lac02}, \cite{Luc16},
    \cite{Luc18},
  \item
    algebraic topology \cite{BJ97}, \cite{CH12},
  \item
    Janelidze-Galois theory \cite{Jan90}, \cite{BJ01},
  \item
    topology \cite{RT94}, \cite{CH02}.
\end{itemize}
It is often useful to depict descent theory as a higher dimensional analogue
of \textit{sheaf theory}, as in \cite[Introduction]{JT94}. The \textit{gluing
condition}, described in terms of an equalizer of a parallel pair of functions
on sets (\textit{sheaf condition}), is replaced by \textit{descent data},
described in terms of a descent object of a suitable diagram of categories
(\textit{descent condition}).

\section{Descent theory with respect to the basic bifibration}

The fundamental setting begins with a category \( \cat C \) with pullbacks, a
morphism \( p \colon e \to b \), and considers the following
\textit{change-of-base} adjunction
\begin{equation*}
  \begin{tikzcd}
    \cat C \comma b \ar[r,bend right,"p^*"{name=A,below},swap]
    & \cat C \comma e, \ar[l,bend right,"p_!"{name=B,above},swap]
    \ar[from=A,to=B,phantom,"\adj" {anchor=center, rotate=-90}]
  \end{tikzcd}
\end{equation*}
where \( \cat C \comma x \) is the comma category whose objects are the
morphisms in \( \cat C \) with codomain \(x\), also called \textit{bundles}
over \(x\). The descent problem, in this setting, can be stated as follows:
describe the morphisms \( p \colon e \to b \) where bundles over \( b \) admit
a presentation as bundles over \( e \) plus some structure, specified by \(
p^* \), satisfying coherence conditions -- this structure is the so-called
\textit{descent data}.

The \textit{category} \( \Desc(p) \) \textit{of descent data of \( p \)} may
be presented as the category \( T^p\dash \Alg \) of \mbox{\(T^p\)-algebras},
where \(T^p\) is the monad induced by the change-of-base adjunction, by the
Bénabou-Roubaud theorem~\cite{BR70}. Hence, we may consider the
\textit{Eilenberg-Moore factorization} of the pullback functor \( p^* \) as
follows:
\begin{equation*}
  \begin{tikzcd}
    \cat C \comma b \ar[rd,"\mathcal K^p",swap] \ar[rr,"p^*"] 
      && \cat C \comma e \\
    & \Desc(p) \ar[ur]
  \end{tikzcd}
\end{equation*}

Therefore, the descent problem is reduced to the question of whether the
comparison functor \( \mathcal K^p \colon \cat C \comma b \to \Desc(p) \) is
an equivalence. When this is the case, we say that \(p\) is an
\textit{effective descent morphism}: these morphisms, the main object of study
of this work, are precisely the solutions to the descent problem. Therefore it
is informative to obtain descriptions of such morphisms. In a pursuit of such
descriptions, it is useful to consider the notions of \textit{descent} and
\textit{almost descent} morphism. We say that \(p\) is
\begin{itemize}[noitemsep,label=--]
  \item
    \( p \) is a descent morphism if \( \mathcal K^p \) is fully faithful, 
  \item
    \( p \) is an almost descent morphism in \( \mathcal K^p \) is faithful.
\end{itemize}
These refinements will be useful in our description of effective descent
morphisms in categorical structures.

If \( \cat C \) has all finite limits, then
\begin{itemize}[noitemsep,label=--]
  \item
    the descent morphisms are precisely the pullback-stable regular
    epimorphisms, and
  \item
    the almost descent morphisms are precisely the pullback-stable
    epimorphisms,
\end{itemize}
but, in general\footnote{If $\cat C$ is either Barr-exact \cite{Bar71} or
locally cartesian closed, the effective descent morphisms are precisely the
regular epimorphisms.}, the effective descent morphisms seldom have an
elementary description. In fact, our story begins with the classical example
of this phenomenon: the characterization of effective descent morphisms in the
category \( \Top \) of topological spaces, first given in \cite{RT94}, shows
how involved such a description can get.

In \cite{Cre99}, we find another example of the prominently challenging
problem of studying effective descent morphisms. This work studies these
morphisms in categories of essentially algebraic theories internal to a
category \( \cat B \) with finite limits. In particular, given a functor \( p
\colon x \to y \) of categories internal to \( \cat B \), Le Creurer shows
that \(p\) is effective for descent if\footnote{Le Creurer required the
component of $p$ on objects, $p_0 \colon x_0 \to y_0$ to be effective for
descent, but this was shown to be redundant in \cite[Lemma A.3]{PL23}.}.
\begin{enumerate}[noitemsep,label=(\Roman*)]
  \item
    \label{enum:singles.eff.desc}
    \( p_1 \colon x_1 \to y_1 \) is an effective descent morphism in \( \cat B
    \),
  \item
    \label{enum:pairs.desc}
    \( p_2 \colon x_2 \to y_2 \) is a descent morphism in \( \cat B \),
  \item
    \label{enum:triples.almost.desc}
    \( p_3 \colon x_3 \to y_3 \) is an almost descent morphism in \( \cat B \),
\end{enumerate}
where \( p_n \colon x_n \to y_n \) is the component of \(p\) on the object of
the \(n\)-tuples of composable morphisms (or \textit{\(n\)-chains}).
Moreover, when \( \cat B \) is lextensive and has a (regular epi,
mono)-factorization system, it was also verified that these criteria are
necessary.

The effective descent morphisms in the category of finite ordered sets
(equivalent to the category of finite topological spaces) were studied in
\cite{JS02b}, proving that a morphism \( p \colon x \to y \) between finite
ordered sets is effective for descent if and only if for all \( a,b,c \in y \)
with \( a \leq b \leq c \) there exist \( a',b',c' \in x \) with \( a' \leq b'
\leq c' \) such that \( a = pa' \), \( b = pb' \) and \( c = pc' \). We point
out the similarity of this condition with \ref{enum:pairs.desc}, as any
ordered set is a category. This insight of ``chain-surjectivity'' led
\cite{CH02} to show that a suitable restatement of these conditions provided a
neat perspective on the characterization of the effective descent morphisms in
\( \Top \), which arise naturally once topological spaces are incarnated as a
\textit{generalized categorical structure}. Inspired by this perspective,
\cite{CH12, CH17, CH04, CJ11} studied the description problem of effective
descent morphisms in various notions of \textit{spaces}; when \( \cat V \) is
a suitable quantale, these are the so-called \( (T, \cat V)
\)\textit{-categories}, which were introduced in a more general setting in
\cite{CT03}. These categories of \( (T,\cat V) \)-categories are a notion of
\textit{enriched} generalized categorical structure, of which the category \(
\Top \) of topological spaces is an example. This corroborates the perspective
of Lawvere, given in \cite{Law73}, that \textit{fundamental structures}, such
as topological spaces, are \textit{categorical} in nature.

The insight of \cite{Cre99, JS02b, CH02} supports the intuition that effective
descent morphisms of categorical structures have a natural description in
terms of the ``chain-surjectivity'' conditions \ref{enum:singles.eff.desc},
\ref{enum:pairs.desc} and \ref{enum:triples.almost.desc}. As an example, via
his study on the commutativity of bilimits, Lucatelli Nunes obtained
\cite[Theorem 9.10]{Luc18}, where it was shown that the embedding of the
category \( \VCat \) of \( \cat V \)-categories into the category \( \CatV \)
of categories internal to \( \cat V \) 
\begin{equation}
  \label{eq:enr.int}
  \VCat \to \CatV 
\end{equation}
reflects effective descent functors, when \( \cat V \) is a suitable
lextensive, cartesian monoidal category. Together with the
``chain-surjectivity'' criteria of Le Creurer, we recover a list of criteria
for such enriched \( \cat V \)-functors to be effective for descent.

This is where the work of the author comes in. It was shown in \cite{Pre23}
that 
\begin{equation*}
  \VCat \to \FamVCat 
\end{equation*}
reflects effective descent morphisms, when  \( \cat V \) is a category with
finite limits (\cite[Lemma 3.1]{Pre23}), where \( \FamV \) is the category of
\textit{families} of objects of \( \cat V \), also known as the \textit{free
coproduct completion} of \( \cat V \). The category \( \FamV \) is a suitable
lextensive category, so that \cite[Theorem 9.10]{Luc18} can be applied to the
embedding \( \FamVCat \to \CatFamV \), confirming that it reflects effective
descent morphisms. Thus, by reflecting along the composite
\begin{equation*}
  \begin{tikzcd}
    \VCat \ar[r] & \FamVCat \ar[r] & \CatFamV,
  \end{tikzcd}
\end{equation*}
the ``chain--surjectivity'' criteria of \cite{Cre99} allow us to describe the
effective descent morphisms in \( \VCat \) in terms of morphisms in \( \FamV
\) (\cite[Theorem 3.3]{Pre23}), extending \cite[Theorem 9.10]{Luc18} to all
categories \( \cat V \) with finite limits. 

The conditions stated in \cite[Theorem 3.3]{Pre23} for effective descent
morphisms in \( \VCat \) are stated in terms of conditions on morphisms in \(
\FamV \). This naturally prompts the study of effective descent morphisms in
free coproduct completions, which was also carried out in \cite{Pre23}. Via
these results, we can show that, when \( \cat V \) is a frame, we obtain one
implication of the main results of \cite{CH04} (\cite[Theorem~4.7]{Pre23}),
effectively confirming that the criterion set forth by \cite{JS02b} are
implied by the criteria of \cite{Cre99}, confirming that both approaches to
seemingly unrelated descent problems have the same underlying ideas.

The work of \cite{CH02, CH12, CH17, CJ11} regards effective descent morphisms
of \textit{\( (T,\cat V) \)-categories}, when \( \cat V \) is a quantale. The
conditions described therein can also be described via similar
``chain-surjectivity'' conditions, evidencing that this perspective on
effective descent morphisms goes beyond plain categorical structures.

\textit{Multicategories} are the most fundamental example of a generalized
categorical structure. An illustrative example of which is the multicategory
\( \Vect \) of vector spaces and \textit{multilinear} maps, that is, functions
\( f \colon V_1 \times \ldots \times V_n \to W \) from a finite list of
vector spaces \( V_1, \ldots, V_n \) to a vector space \( W \), which are
linear in each component:
\begin{equation*}
  f(v_1, \ldots, v_i + \lambda w_i, \ldots, v_n)
    = f(v_1,\ldots,v_i,\ldots,v_n)
        + \lambda f(v_1,\ldots,w_i,\ldots,v_n),
\end{equation*}
where \( v_j \) is a vector in \( V_j \) for each \( j = 1, \ldots, n \), \(
w_i \) is a vector in \(V_i \), and \( \lambda \) is a scalar. This definition
includes \( n = 0 \); in which case \(f\) consists of a vector in \(W\).

Thus, multicategories generalize categories in the sense that the domain of a
morphism consists of a finite string of objects, with an adequate notion of
composition of morphisms, as well as identity morphisms, satisfying suitable
associativity and unity laws. A more thorough introduction to these objects
can be found in Chapter \ref{chap:internal-multi}, along with references for
further study.

More general notions of ``multicategory'' can be obtained by varying the
``shape'' of the domain of a morphism. In the case of categories, the
``shape'' is just an object, while in the multicategory case, the ``shape'' is
a finite string of objects. As we have mentioned, topological spaces are
generalized categorical structures. In this case, for a topological space
\(X\), the domains of the morphisms are \textit{ultrafilters} on its
underlying set of objects (points), and a morphism \( \mathfrak x \to x \) is
the \textit{assertion} that the ultrafilter \( \mathfrak x \) converges to the
point \(x\).\footnote{This is analogous to the notion that a ordered set $X$
can be viewed as a category, whose morphisms $x \to y$ are the assertions that
``$x$ is related to $y$''.}  This perspective on topological spaces can be
traced back to \cite{Bar70}.

As in the case of plain categories, we have a notion of \textit{internal}
generalized multicategories, first considered in \cite{Bur71}, and more
recently in \cite{Her00}, as well as a notion of \textit{enriched} generalized
multicategories, which are the \( (T,\cat V) \)-categories of \cite{CT03}. Our
goal is to obtain an uniform description of the effective descent morphisms on
both accounts of generalized categorical structure.

Towards such a description, the work carried out in \cite{PL23} is our first
step, where we studied the effective descent morphisms of
\textit{\(T\)-categories}\footnote{Here, $T$ is a suitable monad on a $\cat V$
with finite limits which models the shape of the domains of the morphisms.}
\textit{internal to $\cat V$}. Therein, we showed that, if \( \cat V \) is a
category with finite limits, then any functor of \(T\)-categories internal to
\( \cat V \) satisfying a suitable notion of ``chain-surjectivity'' conditions
is an effective descent morphism (\cite[Theorem 5.3]{PL23}), extending the
result of \cite{Cre99} to the setting of internal generalized categorical
structures. In particular, our results provide insight into the effective
descent functors of (ordinary) multicategories internal to any category with
finite limits, such as \( \Set \) or \( \Top \).

Based on \cite{PL23}, and inspired by the techniques of \cite{Luc18}, the work
developed in \cite{PL23b} considers the problem of reflecting effective
descent morphisms along a suitable embedding ``enriched \( \to \) internal''
in the setting of generalized multicategories, analogous to
\eqref{eq:enr.int}. Therein, a notion of \textit{change-of-base} for
generalized multicategories was developed \cite[Theorem 5.2]{PL23b}, and it
was shown that, under suitable conditions (for instance, when \( \cat V \) is
lextensive), the natural generalization of \eqref{eq:enr.int} to the setting
of generalized multicategories
\begin{equation}
  \label{eq:multi.enr.int}
  \TtVCat \to \CatTV
\end{equation}
is an embedding (Theorem 9.2) and reflects effective descent morphisms. Thus,
by applying the results of \cite{PL23}, we obtain a description of effective
descent morphisms for enriched generalized multicategories in terms of
``chain-surjectivity'' conditions (Theorem 10.5).

The topic of obtaining the relationship between the work of \cite{CH02, CH12,
CH17, CJ11} and \cite{PL23b} regarding effective descent morphisms of enriched
generalized multicategories is still the subject of on-going work.

\section{Descent theory with respect to a pseudofunctor}

The descent problem in \cite{Gro60} was stated in a more general setting. 
For each object \(x\) in \( \cat C \), we replace the category of bundles \(
\cat C \comma x \) by a category \( Fx \) of ``structures'' over \(x\), and
for each morphism \( p \colon e \to b \), we have a change-of-base functor \(
p^* \colon Fb \to Fe \), replacing the pullback functor. Together with
coherent isomorphisms \( \id^*_x \iso \id_{Fx} \) and \( (r \circ p)^* \iso
p^* \circ r^* \) for a morphism \( r \colon b \to c \), \(F\) defines a
\textit{pseudofunctor} \( \cat C^\op \to \CAT \).

As was the case for the basic bifibration, we can also define a category \(
\Desc_F(p) \) of \(F\)-descent data for \( p \) for any pseudofunctor \(F\),
for which we obtain a factorization of \( p^* \) -- the \textit{\(F\)-descent
factorization}:
\begin{equation}
  \label{eq:intro.desc.fact}
  \begin{tikzcd}
    Fb \ar[rd,"\mathcal K^p_F",swap] \ar[rr,"p^*"] 
      && Fe \\
    & \Desc_F(p) \ar[ur]
  \end{tikzcd}
\end{equation}
We say that \(p\) is an \textit{effective $F$-descent} ($F$-descent)
morphism if \(\mathcal K^p_F\) is an equivalence (fully faithful).

The main objects of study of \cite{Sob04} are the (effective) descent
morphisms with respect to the pseudofunctor 
\begin{equation*}
  F = \CAT(-,\Set) \colon \Cat^\op \to \CAT 
\end{equation*}
of discrete opfibrations.  Therein, Sobral has shown that any functor \( p
\colon e \to b \) between small categories has a factorization
\begin{equation}
  \label{eq:cat.fact}
  \begin{tikzcd}
    b && e \ar[ll,"p",swap] \ar[ld] \\
    & k_p \ar[lu,"\phi"]
  \end{tikzcd}
\end{equation}
whose image via \( \CAT(-,\Set) \) is equivalent to the \( F \)-descent
factorization~\eqref{eq:intro.desc.fact} with \( F = \CAT(-,\Set) \).
Consequently, we have an equivalence \( \theta \colon \Desc_F(p) \eqv
\CAT(k_p, \Set) \) such that \( \cat K_F^p = \CAT(\phi,\Set) \circ \theta \). 

In this way, the relevance of the notion of \textit{lax epimorphism} becomes
transparent. We say that a functor \( f \colon c \to d \) between small
categories is a \textit{lax epimorphism} \cite{ABSV01} if \( \CAT(f,\Set) \)
is fully faithful. Thus, it follows that \( p \) is a \( \CAT(-,\Set)
\)-descent morphism if and only if \( \phi \) is a lax epimorphism
\cite[Theorem 1]{Sob04}. Moreover \( p \) is an effective \( \CAT(-,\Set)
\)-descent morphism if and only if \( \CAT(\phi,\Set) \) is an equivalence. It
was shown in \cite[Theorem 2]{Sob04} that this is the case if and only if \(
\phi \) is a \textit{fully faithful} lax epimorphism.

The work developed in \cite{LPS23} aims to give a systematic view on the
observations of \cite{Sob04}, aiming to apply them in other contexts. More
specifically, we confirm that the characterization given in \cite{Sob04} can
be plainly extended to further characterize the (effective) descent morphisms
with respect to the pseudofunctor \( \CAT(-, \Cat) \colon \Cat^\op \to \CAT \)
of \textit{split opfibrations}, see \cite[Theorem 3.2]{LPS23}. 

We begin by considering a factorization as in \eqref{eq:cat.fact} where \( k_p
\) is the \textit{lax codescent category} for the kernel pair of \(p\). When
a pseudofunctor \( \cat F \colon \Cat^\op \to \CAT \) preserves lax descent
categories, it follows that \( \cat K^F_p \) is equivalent to \( \cat F\phi
\), reducing the study of whether \(p\) is an (effective) \(F\)-descent
morphism to the study of \( \phi \).

The relationship of (fully faithful) lax epimorphisms with copresheaf
categories influenced the study of their relationship with the \textit{Cauchy
completion} of a category. This work was carried out in the \( \cat V
\)-enriched context, as it is suitable for future considerations, and to do
so, we consider the notion of \( \cat V \)-fully faithful functors, and \(
\cat V \)-lax epimorphisms, as studied in \cite{LS21}.  We have shown that the
following are equivalent, for a \( \cat V \)-functor \( p \colon e \to b \)
between small \( \cat V \)-categories:
\begin{itemize}[label=--,noitemsep]
  \item
    \( p \) is a \( \cat V \)-fully faithful lax epimorphism,
  \item
    the functor \(p^* \colon \Cauchy b \to \Cauchy b\) on the Cauchy
    completions induced by \(p\) is an equivalence,
  \item
    the change-of-base functor \( p^* \colon \VCAT(b,\cat V) \to \VCAT(e,\cat
    V) \) is an equivalence,
\end{itemize}
provided \( \cat V \) is a suitable monoidal category.

We obtain the main result of \cite{LPS23}, Theorem 3.2, by applying our
characterization of \( \cat V \)-fully faithful lax epimorphisms when \( \cat
V = \Set, \Cat \) to the formal considerations pertaining to the factorization
\eqref{eq:cat.fact}. 

\section*{Outline}

In Chapter \ref{chap:descent} we provide a concise introduction to descent
theory. We begin by recalling the 2-dimensional limit known as
\textit{lax descent category} \cite{Luc21}, and stating its 2-dimensional
universal property in Section \ref{sect:lax.desc}. Afterwards, we proceed to
establish the fundamental notion of this thesis: that of \textit{effective
descent morphism} with respect to a pseudofunctor \( \cat C^\op \to \CAT \),
in Section \ref{sect:eff.desc}, where we also give some remarks about the
\textit{Beck-Chevalley condition}. In Section \ref{sect:basic.bifib}, we focus
on the \textit{basic bifibration}, fixing several pieces of notation and
describing the fundamental descent-theoretical results present in
\cite{PL23}, \cite{Pre23} and \cite{PL23b}.

Chapter \ref{chap:enriched}, which covers the work done in \cite{Pre23}, aims to
study effective descent morphisms in \( \VCat \) for a cartesian monoidal
category \( \cat V \) with finite limits. Our first goal is to establish that
\begin{theorem}
  \label{thm:refl1}
  \( \VCat \to \FamVCat \) reflects effective descent morphisms.
\end{theorem}
Theorem \ref{thm:refl1} is obtained via a series of observations on
pseudopullbacks and the fact that the enrichment 2-functor preserves these
2-dimensional limits. Moreover, since \( \FamV \) is a suitable lextensive
category, we conclude, by \cite[Theorem 9.11]{Luc18}, that \( \FamVCat \to
\CatFamV \) reflects effective descent morphisms. Thus, if \(F\) is a \( \cat
V \)-functor, we can verify whether it is effective for descent in terms of
(effective, almost) descent morphisms in \( \FamV \) (Theorem
\ref{thm:desc.vcat}). In turn, this motivated us to study the (effective)
descent morphisms in the free coproduct completions of categories with finite
limits. We apply these results when \( \cat V \) is 
\begin{itemize}[label=--,noitemsep]
  \item
    a (co)complete Heyting lattice, establishing the connection between the
    ``chain-surjectivity'' ideas from \cite{Cre99} and \cite{JS02b},
  \item
    a regular category, such as the categories \( \CHaus \) of compact
    Hausdorff spaces and \( \Stn \) of Stone spaces,
\end{itemize}
This aforementioned connection between \cite{Cre99} and \cite{JS02b} helps
solidify our intuition and general understanding of the problem of effective
descent in categorical structures. More precisely, we recover one implication
of \cite[Theorem 2.5]{CH04} for Heyting lattices \( \cat V \), which, when
taking \( \cat V = 2 \), also recovers the ``chain surjectivity'' of
\cite{JS02b}, confirming the link with the approach of \cite{Cre99}.

In Chapter \ref{chap:internal-multi}, covering the work done in \cite{PL23},
we begin with an overview of the notion of generalized internal
multicategories, studied in \cite{Bur71} and \cite{Her00}. After illustrating
the approach carried out in Section \ref{sect:cat.tv} for the simpler setting
of \textit{reflexive $T$-graphs}, we provide a description of the category \(
\CatTV \) of \textit{internal $T$-categories} via a 2-dimensional limit, via
which we describe the effective descent morphisms, Theorem
\ref{thm:desc.cattv}. This is one of the approaches; we give a second approach
to the study of effective descent morphisms in \( \CatTV \) via direct
calculation, closely following the ideas of \cite[Chapter 3]{Cre99}, Theorem
\ref{thm:int.eff.desc}. We finish the chapter by applying our results to
various sorts of generalized multicategory, among them are the graded,
operadic and enhanced multicategories.

In Chapter \ref{chap:enriched-multi}, covering the relevant definitions and
the descent theoretical results of \cite{PL23b}, we revisit (a slight
generalization of) the notion of \( (T, \cat V) \)-categories defined in
\cite{CT03}, under the terminology \textit{enriched $(T,\cat V)$-categories}.
Afterwards, we give a few concise remarks regarding the change-of-base functor
``enriched \( \to \) internal'' in the context of generalized multicategories,
mentioning only the most fundamental definitions to fix the notation, and
ideas for the results. Having established the embedding under suitable
conditions, we then study the problem of reflection of effective descent
morphisms, and, via Theorem \ref{thm:int.eff.desc}, we obtain our main result,
Theorem \ref{thm:desc.tvcat}, providing a description of the effective descent
morphisms in the enriched generalized multicategory setting. The chapter ends
with some brief comments on the scope of our results, and we list some
examples. 

In Chapter \ref{chap:fib-descent}, covering the work of \cite{LPS23}, the goal
is to study the effective descent morphisms with respect to the bifibration of
\textit{split opfibrations}. The main observation is that the results of
\cite{Sob04} on effective descent morphisms for the bifibration of
\textit{discrete opfibrations} carry over exactly to our setting (Theorem
\ref{thm:desc.fib.opfib}). Indeed, we verify that a functor is of effective \(
\CAT(-,\Cat) \)-descent if and only if it is of effective \( \CAT(-,\Set)
\)-descent. These results are obtained via a study of the relationship between
the Cauchy completion of a category and the fully faithful, lax epimorphisms;
this study was carried out in the enriched setting, alluding to future work.

\section*{List of publications}

The present thesis is based on the work of the following four papers:

\begin{itemize}[noitemsep]
  \item[\cite{LPS23}]
    F. Lucatelli Nunes, R. Prezado and L. Sousa. Cauchy completeness, lax
    epimorphisms and effective descent for split fibrations. \textit{Bull.
    Belg.  Math. Soc. Simon Stevin}, 30(1):130--139, (2023). 
  \item[\cite{Pre23}]
    R. Prezado. On effective descent $\cat V$-functors and familial descent
    morphisms. \textit{J. Pure Appl. Algebra}, 228(5) (2024).
  \item[\cite{PL23}]
    R. Prezado and F. Lucatelli Nunes. Descent for internal multicategory
    functors. \textit{Appl. Categor. Struct.}, 31(11) (2023).
  \item[\cite{PL23b}]
    R. Prezado, F. Lucatelli Nunes. Generalized multicategories:
    change-of-base, embedding and descent. \textit{DMUC preprints} 23--29,
    under review.
\end{itemize}

\chapter{A primer on descent theory}
\label{chap:descent}

This chapter aims to give a concise introduction to classical descent theory,
under a categorical point of view, as well as to uniformize the notation and
gather the preliminary results from \cite{PL23}, \cite{LPS23}, \cite{Pre23}
and \cite{PL23b} pertaining to descent theory.

We begin by reviewing the 2-dimensional limit known as \textit{lax descent
category} \cite[p. 177]{Str76}, \cite{Luc18b, Luc21, Luc22} in Section
\ref{sect:lax.desc}, which is the fundamental notion encompassing the idea of
coherence, and we state its universal property. Part of the work developed in
Chapter \ref{chap:enriched} is done directly under the perspective of the lax
descent category, particularly in Section \ref{sect:fam.desc}.

The fundamental notion to our work, that of \textit{effective descent
morphism} with respect to a pseudofunctor \( F \colon \cat C^\op \to \CAT \),
is presented in Section \ref{sect:eff.desc}. We also provide a few remarks on
the \textit{Beck-Chevalley} condition (introduced in \cite{BR70}), and its
importance in the relationship between monadicity and descent theory, as
evidenced by the \textit{Bénabou-Roubaud theorem}.

In Section \ref{sect:basic.bifib}, we begin our study of descent theory with
respect to the \textit{basic bifibration} 
\begin{equation*}
  \cat C \comma - \colon \cat C^\op \to \CAT 
\end{equation*}
associated to a category \( \cat C \) with pullbacks. We establish the
preliminary descent theoretical tools which are employed in Chapters
\ref{chap:internal-multi} and \ref{chap:enriched-multi}, as well as most of
Chapter \ref{chap:enriched}, namely, classical results regarding the
\textit{reflection} of effective descent morphisms along an embedding \( \cat
C \to \cat D \) (Propositions \ref{prop:full.subcat.desc} and Corollary
\ref{cor:eff.desc.iso}), descent-theoretical results via bilimits
(Propositions \ref{prop:pspb.descent} and \ref{prop:pseq.descent}), as well as
an original result, \cite[Lemma 2.5]{Pre23}, regarding preservation of descent
morphisms (Lemma \ref{lem:preserve.desc}). We also recall results from
\cite{Cre99} and \cite{Luc18} regarding effective descent morphisms of
categorical structures as our starting point (Theorems \ref{thm:int.lec} and
\ref{thm:fln.desc}).

We finish this preliminary chapter with Section \ref{sect:fib.split.opfib},
where we give some remarks about the descent theory with respect to the
\textit{bifibration of split opfibrations} \( \CAT(-,\Cat) \colon \Cat^\op
\to \CAT \). Chapter \ref{chap:fib-descent} studies descent theory with
respect to this pseudofunctor, which is an important example of a bifibration
that does not satisfy the Beck-Chevalley condition.

\section{Lax descent category}
\label{sect:lax.desc}

Throughout this work, we let \( \CAT \) be the 2-category of (large)
categories, functors and natural transformations, and we let \( i \colon \Cat
\to \CAT \) be the full sub-2-category of small categories.

The present definition of lax descent category follows the approach of
\cite[Section 1]{Luc21} of considering the 2-dimensional limit of a
pseudofunctor \( \Delta_3 \to \CAT \), as opposed to the approach via a
2-functor \( \Delta_\str \to \CAT \) carried out in \cite{Luc22}, where \(
\Delta_\str \) is a strict replacement of \( \Delta_3 \). 

To fix the notation, we briefly recall the definition of the truncated
cosimplicial diagram \( \Delta_3 \) and of a pseudofunctor. We define the
category \( \Delta_3 \) to be generated by the following diagram
\begin{equation*}
  \begin{tikzcd}[column sep=large]
    1 \ar[r,shift left=3mm,"d_1"]
      \ar[r,shift right=3mm,"d_0",swap]
    & 2 \ar[r,shift left=3mm,"d_2"]
        \ar[r,"d_1" description]
        \ar[l,"s_0" description]
        \ar[r,shift right=3mm,"d_0",swap]
    & 3
  \end{tikzcd}
\end{equation*}
with relations
\begin{align*}
  s_0 \circ d_1 &= \id_1, & d_2 \circ d_1 &= d_1 \circ d_1, \\
  s_0 \circ d_0 &= \id_1, & d_0 \circ d_0 &= d_1 \circ d_0, \\
  && d_2 \circ d_0 &= d_0 \circ d_1.
\end{align*}

Let \( \cat C \) be a category. A \textit{pseudofunctor} \( F \colon \cat C
\to \CAT \) consists of
\begin{itemize}[noitemsep,label=--]
  \item
    a function \( \ob F \colon \ob \cat C \to \ob \CAT \),
  \item
    a function \( F_{x,y} \colon \cat C(x,y) \to \CAT(Fx,Fy) \), for each
    pair of objects \( x, y \) in \( \cat C \),
  \item
    a natural isomorphism \( \e{F}_x \colon \id_{Fx} \to F(\id_x) \) for each
    object \( x \) in \( \cat C \),
  \item
    a natural isomorphism \( \m{F}_{f,g} \colon Fg \circ Ff \to F(g \circ f)
    \) for each pair of morphisms \( f \colon x \to y \), \( g \colon y \to z
    \) in \( \cat C \),
\end{itemize}
such that the following diagrams commute
\begin{equation*}
  \begin{tikzcd}
    Ff \ar[r,"\e{F}_y \cdot Ff"] \ar[rd,equal]
    & F(\id_y) \circ Ff \ar[d,"\m{F}_{f,\id_y}"] \\
    & Ff
  \end{tikzcd}
  \quad
  \begin{tikzcd}
    Ff \ar[r,"Ff \cdot \e{F}_x"] \ar[rd,equal]
    & Ff \circ F(\id_x) \ar[d,"\m{F}_{\id_x,f}"] \\
    & Ff   
  \end{tikzcd}
  \quad
  \begin{tikzcd}
    Fh \circ Fg \circ Ff \ar[r,"\m{F}_{g,h} \cdot Ff"]  
                         \ar[d,"Fh \cdot \m{F}_{f,g}",swap]
    & F(h \circ g) \circ Ff \ar[d,"\m{F}_{f,h \circ g}"] \\
    Fh \circ F(g \circ f) \ar[r,"\m{F}_{g \circ f, h}",swap]
    & F(h \circ g \circ f)
  \end{tikzcd}
\end{equation*}
for all morphisms \( f \colon x \to y \), \( g \colon y \to z \) and \( h
\colon z \to w \).

For a pseudofunctor \( F \colon \Delta_3 \to \CAT \), we write the underlying
diagram as
\begin{equation*}
  \begin{tikzcd}[column sep=large]
    F1 \ar[r,shift left=3mm,"d^F_1"]
      \ar[r,shift right=3mm,"d^F_0",swap]
    & F2 \ar[r,shift left=3mm,"d^F_2"]
        \ar[r,"d^F_1" description]
        \ar[l,"s^F_0" description]
        \ar[r,shift right=3mm,"d^F_0",swap]
    & F3
  \end{tikzcd}
\end{equation*}
and we define the following natural isomorphisms:
\begin{align*}
  \upsilon^F_1 
    &= {\m{F}_{d_1,s_0}}^{-1} \circ \e{F}_1 \colon \id \to s^F_0d^F_1,
    & \theta^F_{01} &= {\m{F}_{d_1,d_1}}^{-1} \circ \m{F}_{d_1,d_2} 
                       \colon d^F_2 d^F_1 \to d^F_1 d^F_1, \\
  \upsilon^F_0 
    &= {\m{F}_{d_0,s_0}}^{-1} \circ \e{F}_1 \colon \id \to s^F_0 d^F_0,
    & \theta^F_{02} &= {\m{F}_{d_0,d_2}}^{-1} \circ \m{F}_{d_1,d_0} 
                       \colon d^F_2d^F_0 \to d^F_0d^F_1, \\
    && \theta^F_{12} &= {\m{F}_{d_0,d_1}}^{-1} \circ \m{F}_{d_0,d_0} 
                       \colon d^F_1d^F_0 \to d^F_0d^F_0.
\end{align*}

The \textit{lax descent category} of a pseudofunctor \( F \colon \Delta_3 \to
\CAT \) is a category \( \Desc(F) \) whose objects, called \textit{lax
$F$-descent data}, are pairs \( (x,\phi) \) where \( x \) is an object in \(
F1 \) and \( \phi \colon d^F_1(x) \to d^F_0(x) \) is a morphism in \( F2
\) satisfying the following \textit{reflexivity} condition
\begin{equation}
  \label{eq:refl.desc}
  \begin{tikzcd}
    & x \ar[ld,"\upsilon^F_1",swap]  
        \ar[rd,"\upsilon^F_0"] \\
    s^F_0d^F_1(x) \ar[rr,"s^F_0(\phi)",swap]
      && s^F_0d^F_0(x) 
  \end{tikzcd}
\end{equation}
and \textit{transitivity} condition
\begin{equation}
  \label{eq:trans.desc}
  \begin{tikzcd}
    & d^F_1d^F_1(x) \ar[rr,"d^F_1(\phi)"]
    && d^F_1d^F_0(x) \ar[rd,"\theta^F_{12}"] \\
    d^F_2d^F_1(x) \ar[ru,"\theta^F_{01}"] 
                  \ar[rd,"d^F_2(\phi)",swap]
    &&&& d^F_0d^F_0(x) \\
    & d^F_2d^F_0(x) \ar[rr,"\theta^F_{02}",swap]
    && d^F_0d^F_1(x) \ar[ru,"d^F_0(\phi)",swap]
  \end{tikzcd}
\end{equation}

Given lax \(F\)-descent data \( (x,\phi) \), \( (y,\psi) \), a morphism \(
\omega \colon (x,\phi) \to (y,\psi) \) in \( \Desc(F) \) of lax \(F\)-descent
data consists of a morphism \( \omega \colon x \to y \) such that the
following diagram commutes:
\begin{equation*}
  \begin{tikzcd}
    d_1^F(x) \ar[d,"d_1^F(\omega)",swap]
             \ar[r,"\phi"] 
    & d_0^F(x) \ar[d,"d_0^F(\omega)"] \\
    d_1^F(y) \ar[r,"\psi",swap]
    & d_0^F(y)
  \end{tikzcd}
\end{equation*}

\subsection{Universal property} 

Associated to the lax descent category, we have a forgetful functor \(
X \colon \Desc(F) \to F1 \) given on objects by \( (x,\phi) \mapsto x \),
and a natural transformation \( \Phi \colon d_1^F X \to d_0^F X \) given by \(
\Phi_{(x,\phi)} = \phi \). This pair \( (X,\Phi) \) makes the following
diagrams commute
\begin{equation*}
  \begin{tikzcd}
    & X \ar[ld,"\upsilon^F_1 \cdot X",swap]  
        \ar[rd,"\upsilon^F_0 \cdot X"] \\
    s^F_0d^F_1 X \ar[rr,"s^F_0 \cdot \Phi",swap]
      && s^F_0d^F_0 X
  \end{tikzcd}
\end{equation*}
\begin{equation*}
  \begin{tikzcd}
    & d^F_1d^F_1 X \ar[rr,"d^F_1 \cdot \Phi"]
    && d^F_1d^F_0 X \ar[rd,"\theta^F_{12} \cdot X"] \\
    d^F_2d^F_1 X \ar[ru,"\theta^F_{01} \cdot X"] 
                  \ar[rd,"d^F_2 \cdot \Phi",swap]
    &&&& d^F_0d^F_0 X \\
    & d^F_2d^F_0 X \ar[rr,"\theta^F_{02} \cdot X",swap]
    && d^F_0d^F_1 X \ar[ru,"d^F_0 \cdot \Phi",swap]
  \end{tikzcd}
\end{equation*}
which is just a restatement of the conditions \eqref{eq:refl.desc} and
\eqref{eq:trans.desc}.

If \( (Y,\Psi) \) is a pair where \( Y \colon \cat A \to F1 \) is a functor
and \( \Psi \colon d_1^F Y \to d_0^F Y \) is a natural transformation
satisfying
\begin{equation}
  \label{eq:glob.refl.desc}
  \begin{tikzcd}
    & Y \ar[ld,"\upsilon^F_1 \cdot Y",swap]  
        \ar[rd,"\upsilon^F_0 \cdot Y"] \\
    s^F_0d^F_1 Y \ar[rr,"s^F_0 \cdot Y",swap]
      && s^F_0d^F_0 Y
  \end{tikzcd}
\end{equation}
\begin{equation}
  \label{eq:glob.trans.desc}
  \begin{tikzcd}
    & d^F_1d^F_1 Y \ar[rr,"d^F_1 \cdot \Psi"]
    && d^F_1d^F_0 Y \ar[rd,"\theta^F_{12} \cdot Y"] \\
    d^F_2d^F_1 Y \ar[ru,"\theta^F_{01} \cdot Y"] 
                  \ar[rd,"d^F_2 \cdot \Psi",swap]
    &&&& d^F_0d^F_0 Y \\
    & d^F_2d^F_0 Y \ar[rr,"\theta^F_{02} \cdot Y",swap]
    && d^F_0d^F_1 Y \ar[ru,"d^F_0 \cdot \Psi",swap]
  \end{tikzcd}
\end{equation}
then there is a unique \( G \colon \cat A \to \Desc(F) \) such that 
\begin{equation}
  \label{eq:pre.desc.fact}
  \begin{tikzcd}
    \cat A \ar[rr,"Y"] \ar[rd,"G",swap] && F1 \\
    & \Desc(F) \ar[ru,"X",swap]
  \end{tikzcd}
\end{equation}
commutes, and \( \Psi = \Phi \cdot G \).

Let \( (Z,\Xi) \) be another pair where \( Z \colon \cat A \to F1 \) is a
functor and \( \Xi \colon d_1^F Z \to d_0^F Z \) is a natural transformation
satisfying \eqref{eq:glob.refl.desc} and \eqref{eq:glob.trans.desc}, and let
\( H \colon \cat A \to \Desc(F) \) be the unique functor such that \( Z = XH
\) and \( \Xi = \Phi \cdot H \). For any natural transformation \( \Gamma
\colon Y \to Z \) such that
\begin{equation*}
  \begin{tikzcd}
    d_1^F Y \ar[d,"d_1^F \cdot \Gamma",swap]
            \ar[r,"\Psi"] 
    & d_0^F Y \ar[d,"d_0^F \cdot \Gamma"] \\
    d_1^F Z \ar[r,"\Xi",swap] 
    & d_0^F Z
  \end{tikzcd}
\end{equation*}
commutes, there is a unique \( \hat \Gamma \colon G \to H \) such that \( X
\cdot \hat \Gamma = \Gamma \).

\section{Effective descent morphisms}
\label{sect:eff.desc}

Let \( \cat C \) be a category with pullbacks, and \( p \colon e \to b \) be a
morphism in \( \cat C \). The \textit{kernel pair} of \( p \), given by the
pullback
\begin{equation}
  \label{eq:ker.pair}
  \begin{tikzcd}
    p \times_b p \ar[r,"d_1"] \ar[d,"d_0",swap]
                 \ar[rd,"\ulcorner",phantom,very near start]
    & e \ar[d,"p"] \\
    e \ar[r,"p",swap] & b
  \end{tikzcd}
\end{equation}
induces an \textit{equivalence relation} internal to \( \cat C \), given by
the following diagram 
\begin{equation}
  \label{eq:eqp}
  \begin{tikzcd}
    p \times_b p \times_b p
      \ar[r,shift left=2mm] \ar[r,shift right=2mm] \ar[r]
    & p \times_b p \ar[r,shift left=2mm,"d_1"] 
                   \ar[r,shift right=2mm,"d_0",swap]
    & e, \ar[l]
  \end{tikzcd}
\end{equation}
which we denote by \( \Ker(p) \colon \Delta_3^\op \to \cat C \). 

For any pseudofunctor \( F \colon \cat C^\op \to \CAT \), we let \( f^* =
Ff \) for every morphism \(f\) in \( \cat C \), and \( \Ff_p = F
\circ \Ker(p)^\op \). We write
\begin{equation}
  \label{eq:cat.descent}
  \Desc_F(p) = \Desc(\Ff_p)
\end{equation}
for the category of lax descent data of \( \Ff_p \). Moreover, we observe that
the kernel pair \eqref{eq:ker.pair} induces a pair \( (p^* \colon Fb \to
Fe,\,\, \Omega \colon d^*_1 p^* \to d^*_0 p^*) \) satisfying
\eqref{eq:glob.refl.desc} and \eqref{eq:glob.trans.desc}, where \( \Omega =
{\m{F}_{p,d_0}}^{-1} \circ \m{F}_{p,d_1} \). Thus, by the universal property
\eqref{eq:pre.desc.fact}, we obtain a unique functor \( \mathcal K^p_F \colon
Fb \to \Desc_F(p) \) such that the following diagram commutes
\begin{equation}
  \label{eq:desc.fact}
  \begin{tikzcd}
    Fb \ar[rr,"p^*"] \ar[rd,"\mathcal K^p_F",swap] && Fe \\
    & \Desc_F(p) \ar[ru,"\mathcal U^p_F",swap]
  \end{tikzcd}
\end{equation}
and \( \Omega = \Phi \cdot p^* \), where \( (\cat U^p_F,\Phi) \) is the pair
associated to the lax descent category of \( \Ff_p \). We call
\eqref{eq:desc.fact} the \textit{$F$-descent factorization} of \( p \).

We now reach the most important definition in this work. A morphism \( p \) in
a category \( \cat C \) with pullbacks is said to be 
\begin{itemize}[label=--,noitemsep]
  \item
    an \textit{almost $F$-descent morphism} if \(\mathcal K^p_F\) is faithful,
  \item
    an \textit{$F$-descent morphism} if \(\mathcal K^p_F\) is fully faithful,
  \item
    an \textit{effective $F$-descent morphism} if \(\mathcal K^p_F\) is an
    equivalence.
\end{itemize}

\subsection{Beck-Chevalley condition}

The \textit{Beck-Chevalley condition}, introduced in \cite{BR70}, underlies
the relationship between monadicity and descent theory. We will give a brief
remark on this topic; we recall that a pseudofunctor \( F \colon \cat C^\op
\to \CAT \) is said to be a \textit{bifibration} if, for every morphism \( f
\) in \( \cat C \), the functor \( f^* \) has a left adjoint, which we denote
by \( f_! \adj f^* \), and whose unit and counit are denoted by \( \eta^f \)
and \( \epsilon^f \).

In this context, a bifibration is said to satisfy the \textit{Beck-Chevalley
condition} if, for all pullback squares in \( \cat C \)
\begin{equation*}
  \label{eq:com.sq}
  \begin{tikzcd}
    w \ar[r,"h"] \ar[d,"k",swap]
      \ar[rd,"\ulcorner",phantom,very near start]
      & x \ar[d,"f"] \\
    y \ar[r,"g",swap] & z
  \end{tikzcd}
\end{equation*}
the following natural transformation 
\begin{equation*}
  \begin{tikzcd}
    h_!k^* \ar[r,"h_!k^*\eta^g"]
    & h_!k^*g^*g_! \ar[r,"h_! \theta g_!"]
    & h_!h^*f^*g_! \ar[r,"\epsilon^h f^*g_!"]
    & f^*g_!
  \end{tikzcd}
\end{equation*}
is a natural isomorphism, where \( \theta \colon k^*g^* \to h^*f^* \) is the
induced isomorphism by the commutative square \eqref{eq:com.sq}.

\begin{theorem}[Bénabou-Roubaud \cite{BR70}]
  \label{thm:br}
  Let \( F \colon \cat C^\op \to \CAT \) be a bifibration satisfying the
  Beck-Chevalley condition. For a morphism \( p \colon e \to b \) in \( \cat
  C \), we write \( T^p \) for the monad induced by the adjunction \( p_! \adj
  p^* \). 

  The \(F\)-descent factorization \eqref{eq:desc.fact} of \(p\) is equivalent
  to the Eilenberg-Moore factorization of \( p^* \):
  \begin{equation}
    \label{eq:em.fact}
    \begin{tikzcd}
      Fb \ar[rr,"p^*"] \ar[rd,"\mathcal K^p",swap] && Fe \\
      & T^p\dash \Alg \ar[ru,"\mathcal U^p",swap]
    \end{tikzcd}
  \end{equation}
  so, in particular, we have an equivalence of categories 
  \begin{equation*}
    \Desc_F(p) \eqv T^p\dash \Alg,
  \end{equation*}
  and the following are equivalent:
  \begin{enumerate}[label=(\roman*),noitemsep]
    \item
      \label{enum:p.eff.desc}
      \( p \) is an effective \(F\)-descent morphism (resp. \(F\)-descent
      morphism).
    \item
      \label{enum:pstar.monadic}
      \( p^* \) is monadic (resp. premonadic).
  \end{enumerate}
\end{theorem}

This result was generalized in \cite[Theorem 7.4]{Luc18}, via the study of
commutativity of bilimits, and the main result of \cite{Luc21} (Theorem 4.7)
confirms that \ref{enum:p.eff.desc} \( \implies \) \ref{enum:pstar.monadic},
for \textit{all} bifibrations (that is, not necessarily satisfying the
Beck-Chevalley condition).

\section{Basic bifibration}
\label{sect:basic.bifib}

Let \( \cat C \) be a category with pullbacks. For each morphism \( f \colon x
\to y \), we have a functor \( f_! \colon \cat C \comma x \to \cat C \comma y
\) given on objects by \( g \mapsto f \circ g \). For each morphism \( h
\colon z \to y \), we consider the following pullback diagram:
\begin{equation*}
  \begin{tikzcd}
    f^*(z) \ar[r,"\epsilon^f_h"] \ar[d,swap,"f^*(h)"] 
           \ar[rd,"\ulcorner",phantom,very near start]
    & z \ar[d,"h"] \\
    x \ar[r,"f",swap] & y
  \end{tikzcd}
\end{equation*}
We observe that \( \epsilon^f_h \colon f \circ f^*(h) \to h \) is a morphism
in \( \cat C \comma y \), whose universal property borne out of the pullback
diagram guarantees that the assignment \( f^* \colon \cat C \comma y \to \cat
C \comma x \) defines a functor right adjoint to \( f_! \).

Together with the canonical isomorphisms \( \id_x^* \iso \id_{\cat C \comma
x} \) and \( f^*g^* \iso (g \circ f)^* \), we obtain the \textit{basic
bifibration}
\begin{align*}
  \cat C \comma - \colon \cat C^\op &\to \CAT\\
  x &\mapsto \cat C \comma x \\
  f \colon x \to y &\mapsto f^* \colon \cat C \comma y \to \cat C \comma x
\end{align*}
which provides the context for our study of descent theory in Chapters
\ref{chap:enriched}, \ref{chap:internal-multi}, and \ref{chap:enriched-multi}.
In this setting, we write \( \Desc(p) \) instead of \( \Desc_{\cat C \comma
-}(p) \) for the category of lax descent data, and we say ``(effective/almost)
descent morphism'' instead of ``(effective/almost) \((\cat C \comma
-)\)-descent morphism''.

When \( \cat C \) is a category with finite limits, it can be shown (see
\cite[{}2.4]{JT94}, \cite[Corollary 0.3.5]{Cre99}, \cite[Theorem 3.4]{JST04},
\cite[Proposition 2.1]{PL23}) via Beck's monadicity theorem that
\begin{itemize}[label=--,noitemsep]
  \item
    \( p \) is an almost descent morphism \( \iff \) \( p \) is a
    pullback-stable epimorphism,
  \item
    \( p \) is a descent morphism \( \iff \) \( p \) is a pullback-stable
    regular epimorphism.
\end{itemize}

While it is true that descent morphisms are effective for descent when \( \cat
C \) is Barr-exact \cite{Bar71} or locally cartesian closed, in general,
effective descent morphisms are challenging to describe. For instance, we note
the characterization of \cite{RT94} for \( \cat C = \Top \), or the
characterization of \cite{Cre99} for \( \cat C = \CatV \), for suitable
categories \( \cat V \). 

Nevertheless, the study of effective descent morphisms can be approached
directly, by studying the (essential) image of (fully faithful) comparison
functors \( \mathcal K^p \) for descent morphisms \( p \). Thus, the following
elementary observation regarding the (essential) image of \( \mathcal K^p \)
is of particular interest for our work.

\begin{proposition}[{\cite[Corollary 2.3]{PL23}}]
  \label{prop:comp.ess.surj}
  The comparison functor \( \mathcal K^p \) is essentially surjective if and
  only if, for all descent data \( (a,\gamma) \), there exists a morphism \( f
  \colon w \to b \) such that \( p^*f \iso a \) in \( \cat C \comma x \), and
  \begin{equation}
    \label{eq:img.cond}
    \epsilon^p_f \circ \gamma = \epsilon^p_f \circ \epsilon^p_{p \circ a}.
  \end{equation}
\end{proposition}

\begin{proof}
  We begin by noting that \( \mathcal K^pf = (p^*f, p^*_{\epsilon^p_f}) \) is
  a lax descent datum satisfying \eqref{eq:img.cond}, by naturality.

  Conversely, if \( (p^*f,\gamma) \) satisfies \eqref{eq:img.cond}, then
  \begin{equation*}
    \gamma = p^*\epsilon^p_f \circ \eta^p_{p^*f} \circ \gamma
           = p^*\epsilon^p_f \circ p^*\gamma  
                           \circ \eta^p_{p^*(p \circ p^*f)}
           = p^*\epsilon^p_f \circ p^*\epsilon^p_{p \circ a}
                           \circ \eta^p_{p^*(p \circ p^*f)}
           = p^*\epsilon^p_f,
  \end{equation*}
  hence \( (p^*f, \gamma) = \mathcal K^pf \).
\end{proof}

\begin{remark}
  We should point out that Proposition \ref{prop:comp.ess.surj} is often
  implicitly used in the study of effective descent morphisms (for instance,
  \cite[Proposition 3.2.4]{Cre99}). Moreover, it can be shown that this result
  also holds in the general context for a pseudofunctor \( F \colon \cat C^\op
  \to \CAT \), so its applicability in descent arguments does not rely on the
  Beck-Chevalley condition.
\end{remark}

Another fruitful strategy, undertaken by both \cite{RT94} and \cite{Cre99},
and justified by the following Proposition \ref{prop:full.subcat.desc}, is to
suitably embed the category whose effective descent morphisms we wish to study
in a larger one in which those are well-understood. 

\begin{proposition}[{\cite[{}2.7]{JT94}}]
  \label{prop:full.subcat.desc}
  Let \( U \colon \cat C \to \cat D \) be a fully faithful,
  pullback-preserving functor between categories with pullbacks, and let \( p
  \colon x \to y \) be a morphism such that \( Up \) is an effective descent
  morphism. Then \(p\) is an effective descent morphism if and only if, for
  all pullback diagrams of the form
  \begin{equation}
    \label{eq:pb.diag.desc}
    \begin{tikzcd}
      Uv \ar[r] \ar[d] 
         \ar[rd,"\ulcorner",phantom,very near start]
      & w \ar[d] \\
      Ux \ar[r,"Up",swap] & Uy
    \end{tikzcd}
  \end{equation}
  we have \( w \iso Uz \) for some object \(z\) of \( \cat C \).
\end{proposition}

Throughout our study of effective descent morphisms in categorical structures,
we have found the following particular instance of Proposition
\ref{prop:full.subcat.desc} to be particularly useful:

\begin{corollary}[{\cite[Corollary 2.5]{PL23}}]
  \label{cor:eff.desc.iso}
  Let \( U \colon \cat C \to \cat D \) be a fully faithful,
  pullback-preserving functor between categories with pullbacks. If for every
  effective descent morphism \( g \colon Ux \to z \) there exists an
  isomorphism \( z \iso Uy \), then \( U \) reflects effective descent
  morphisms.
\end{corollary}

\begin{proof}
  We recall that effective descent morphisms are stable under pullback. Thus,
  if \eqref{eq:pb.diag.desc} is a pullback square, and \( Up \) is an
  effective descent morphism, then so is \( Uv \to w \) by pullback-stability.
  By hypothesis, there exists an isomorphism \( w \iso Uy \), whence we
  conclude that \( p \) is effective for descent by Proposition
  \ref{prop:full.subcat.desc}.
\end{proof}

As one of the byproducts of the study of commutativity of bilimits carried out
in \cite{Luc18}, Lucatelli Nunes provides a description of the effective
descent morphisms for the bilimit of a diagram of categories with pullbacks and
pullback-preserving functors, in terms of the (effective) descent morphisms of
the categories in the underlying diagram. A particularly important consequence
is that this provides a second, widely applicable approach to the study of
effective descent morphisms, as exhibited by the following Propositions:

\begin{proposition}[{\cite[Theorem 1.6, Corollary 9.6]{Luc18}}]
  \label{prop:pspb.descent}
  If we have a pseudopullback diagram of categories with pullbacks and
  pullback-preserving functors
  \begin{equation*}
    \begin{tikzcd}
      \cat A \ar[r,"F"] \ar[d,"G",swap]
             \ar[rd,"\iso" description,phantom]
        & \cat B \ar[d,"H"] \\
      \cat C \ar[r,"K",swap] & \cat D
    \end{tikzcd}
  \end{equation*}
  and a morphism \( f \) in \( \cat A \) such that
  \begin{itemize}[label=--,noitemsep]
    \item
      \( Ff \) and \( Gf \) are effective descent morphisms, and
    \item
      \( KGf \iso HFf \) is a descent morphism,
  \end{itemize}
  then \( f \) is an effective descent morphism.
\end{proposition}

\begin{proposition}[{\cite[Theorem 9.2]{Luc18}}]
  \label{prop:pseq.descent}
  If we have a pseudoequalizer diagram of categories with pullback and
  pullback-preserving functors
  \begin{equation*}
    \begin{tikzcd}
      \cat A \ar[r,"F"] 
        & \cat B \ar[r,shift left,"G"]
                 \ar[r,shift right,"H",swap]
        & \cat C
    \end{tikzcd}
  \end{equation*}
  and a morphism \(f\) in \( \cat A \) such that
  \begin{itemize}[label=--,noitemsep]
    \item
      \( Ff \) is an effective descent morphism, and
    \item
      \( GFf \iso HFf \) is a descent morphism,
  \end{itemize}
  then \(f\) is an effective descent morphism.
\end{proposition}

To apply Proposition \ref{prop:pspb.descent}, the following result is useful:
\begin{lemma}[{\cite[Lemma 2.5]{Pre23}}]
  \label{lem:preserve.desc}
  Let \( H \colon \cat B \to \cat D \) be a functor between categories with
  finite limits. If \( H \) preserves coequalizers, reflects pullbacks, and
  has a fully faithful left adjoint \( L \colon \cat D \to \cat B \), then \(H\)
  preserves descent morphisms.
\end{lemma}

\begin{proof}
  Let \( h \colon x \to y \) be a descent morphism in \( \cat B \), and we
  consider the following pullback diagram:
  \begin{equation*}
    \begin{tikzcd}
      (Hh)^*z \ar[r,"\epsilon_\phi"] \ar[d,"(Hh)^*\phi",swap]
              \ar[rd,"\ulcorner",phantom,very near start]
        & z \ar[d,"\phi"] \\
      Hx \ar[r,"Hh",swap] & Hy
    \end{tikzcd}
  \end{equation*}
  Since the unit \( \id \to HL \) is an isomorphism, the following is a
  pullback diagram as well:
  \begin{equation*}
    \begin{tikzcd}
      HL(Hh)^*z \ar[r,"HL\epsilon_\phi"] \ar[d,"H\psi^\sharp",swap]
                \ar[rd,"\ulcorner",phantom,very near start]
        & HLz \ar[d,"H\phi^\sharp"] \\
      Hx \ar[r,"Hh",swap] & Hy
    \end{tikzcd}
  \end{equation*}
  where \( \phi^\sharp \colon Lz \to y \) is obtained from \( \phi \) via the
  hom-isomorphism \( \cat D(z,Hy) \iso \cat B(Lz,y) \), and, likewise, \(
  \psi^\sharp \colon L(Hh)^*z \to x \) is obtained from \( \psi = (Hh)^*\phi \).
 
  Since \( H \) reflects pullbacks, we conclude \( L\epsilon_\phi \) is a
  regular epimorphism. This property is preserved by \(H\), as it preserves
  coequalizers. Thus, we conclude that \( Hh \) is a pullback-stable regular
  epimorphism, as desired.
\end{proof}

\begin{remark}
  \label{rem:fln.gen}
  The applications of Lemma \ref{lem:preserve.desc} we have in mind are 
  \begin{itemize}[label=--,noitemsep]
    \item
      the underlying object-of-objects functor \( (-)_0 \colon \CatV \to \cat
      V \),
    \item
      the canonical fibration \( \FamV \to \Set \),
    \item
      and the underlying object-of-objects functor \( (-)_0 \colon \CatTV \to
      \cat V \).
  \end{itemize}
  Each functor has fully faithful left and right adjoints, and therefore
  satisfies the hypotheses of Lemma~\ref{lem:preserve.desc}. Hence, we
  conclude that each functor preserves descent morphisms.

  Thanks to the first functor, we can obtain the conclusion of \cite[Theorem
  9.11]{Luc18} for \( \cat V \) a lextensive category such that the functor \(
  - \pt 1 \colon \Set \to \cat V \)\footnote{This is the left adjoint to $
  \cat V(1,-) \colon \cat V \to \Set $}, is fully faithful, without assuming
  \( \cat V \) has a (regular epi, mono)-factorization system, using precisely
  the same proof. 

  The second functor plays an important role in Chapter \ref{chap:enriched} in
  obtaining effective descent morphisms in \( \VCat \), in the more general
  setting of a cartesian monoidal category \( \cat V \) with finite
  limits; this is Theorem \ref{thm:desc.vcat}.

  The third functor plays a similar role in Chapter \ref{chap:enriched-multi},
  to obtain effective descent morphisms in a suitable category of generalized
  enriched multicategories. The statement of this result is given by Theorem
  \ref{thm:desc.tvcat}.
\end{remark}

\subsection{Descent theory in categorical structures} 

In the context of effective descent morphisms in categorical structures, the
results of Le Creurer \cite{Cre99} are the cornerstone upon which we obtain
our own.  Indeed, since enriched categorical structures can be embedded into
an internal setting, under suitable conditions, the general strategy is to
first study the effective descent morphisms for internal categorical
structures directly (such results are given by Theorems \ref{thm:int.lec} and
\ref{thm:int.eff.desc}), and then apply Propositions
\ref{prop:full.subcat.desc} and \ref{prop:pspb.descent} to study whether the
embedding reflects the effective descent morphisms back to the enriched
setting.

\begin{theorem}[{\cite[Corollary 3.3.1]{Cre99}}]
  \label{thm:int.lec}
  Let \( \cat V \) be a category with finite limits, and \( p \colon \cat C
  \to \cat D \) be a functor of categories internal to \( \cat V \). If
  \begin{itemize}[noitemsep,label=--]
    \item
      \( p_0 \colon \cat C_0 \to \cat D_0 \) is an effective descent morphism,
    \item
      \( p_1 \colon \cat C_1 \to \cat D_1 \) is an effective descent morphism,
    \item
      \( p_2 \colon \cat C_2 \to \cat D_2 \) is a descent morphism, and
    \item
      \( p_3 \colon \cat C_3 \to \cat D_3 \) is an almost descent morphism,
  \end{itemize}
  then \( p \) is an effective descent morphism in \( \CatV \).
\end{theorem}

\begin{remark}
  \label{rem:chains}
  Let \( \cat C \) be a category internal to \( \cat V \). Since Theorem
  \ref{thm:int.lec} is stated in terms of various epimorphic conditions on the
  objects on composable tuples of morphisms, it is convenient for \( \cat C_n
  \) to denote the object of \(n\)\textit{-chains} of \( \cat C \), therefore
  allowing ``\(n\)-chain'' to be synonymous with ``composable \(n\)-tuple of
  morphisms''.

  In this vein, let \( \cat W \) be a monoidal category, and let \( \cat D \)
  be an enriched \( \cat W \)-category. For objects \( x_i \) on \( \cat D \)
  for \( i = 0,\, 1,\, 2,\, 3 \), we let
  \begin{equation*}
    \cat D(x_0,x_1,x_2) = \cat D(x_1,x_2) \otimes \cat D(x_0,x_1),
    \quad\text{and}\quad
    \cat D(x_0,x_1,x_2,x_3) = \cat D(x_1,x_2,x_3) \otimes \cat D(x_0,x_1). 
  \end{equation*}
  Likewise, if \( F \colon \cat W \to \cat X \) is a lax monoidal functor, we
  write
  \begin{equation*}
    \m F \colon (F_!\cat D)(x_0,x_1,x_2) \to F(\cat D(x_0,x_1,x_2))
  \end{equation*}
  for the ``shortened'' version of the comparison morphism for the tensor
  product.
\end{remark}

As an example of a reflection result, we have the following consequence of
\cite[Theorem 9.11]{Luc18} and Theorem \ref{thm:int.lec} (see also Remark
\ref{rem:fln.gen}):

\begin{theorem}[{\cite[Theorem 9.11]{Luc18}}]
  \label{thm:fln.desc}
  Let \( \cat V \) be a lextensive category, and assume that \( - \pt 1 \colon
  \Set \to \cat V \) is fully faithful. An enriched \( \cat V \)-functor \( F
  \colon \cat C \to \cat D \) such that
  \begin{itemize}[noitemsep,label=--]
    \item
      \( F_0 \) is surjective,
    \item
      \( F_1 \pt 1 \colon \sum_{x_i} \cat C(x_0,x_1) \to \sum_{y_i} \cat
      D(y_0,y_1) \) is an effective descent morphism,
    \item
      \( F_2 \pt 1 \colon \sum_{x_i} \cat C(x_0,x_1,x_2) \to \sum_{y_i} \cat
      D(y_0,y_1,y_2) \) is a descent morphism,
    \item
      \( F_3 \pt 1 \colon \sum_{x_i} \cat C(x_0,x_1,x_2,x_3) \to \sum_{y_i}
      \cat D(y_0,y_1,y_2,y_3) \) is an almost descent morphism,
  \end{itemize}
  is an effective descent morphism in \( \VCat \).
\end{theorem}

\section{Bifibration of split opfibrations}
\label{sect:fib.split.opfib}

Let \( \bicat A \) be a 2-category with 2-pullbacks and lax codescent
objects\footnote{This is the notion dual to \textit{lax descent object}, which
can be described in any 2-category, via the universal property in Subsection
2.1.1.}.  Just like in Section \ref{sect:basic.bifib}, we consider the kernel
pair \( \Ker(p) \) \eqref{eq:eqp} of a morphism \( p \colon e \to b \) in \(
\bicat A \). Since \( \bicat A \) has lax codescent objects, there is a unique
\( K^{\Ker(p)} \colon \CoDesc(\Ker(p)) \to b \) making the triangle in Diagram
\eqref{eq:lax.codesc.fact} commute:
\begin{equation}
  \label{eq:lax.codesc.fact}
  \begin{tikzcd}
    p \times_b p \times_b p
      \ar[r,shift left=2mm] \ar[r,shift right=2mm] \ar[r]
    & p \times_b p \ar[r,shift left=2mm,"d_1"] 
                   \ar[r,shift right=2mm,"d_0",swap]
    & e \ar[l] \ar[rr,"p"] \ar[rd] && b \\
    &&& \CoDesc(\Ker(p)) \ar[ur,"K^{\Ker(p)}",swap,dashed]
  \end{tikzcd}
\end{equation}

\begin{lemma}[{\cite[Lemma 1.1]{LPS23}}]
  \label{lem:lax.desc.preserve}
  If a 2-functor \( F \colon \bicat A^\op \to \CAT \) preserves lax descent
  objects, then a morphism \( p \colon e \to b \) is of effective
  \(F\)-descent (\(F\)-descent) if and only if \( F(K^{\Ker(p)}) \) is an
  equivalence (fully faithful).
\end{lemma}

\begin{proof}
  When such a 2-functor \(F\) is composed with Diagram
  \eqref{eq:lax.codesc.fact}, we obtain
  \begin{equation}
    \label{eq:lax.desc.fact}
    \begin{tikzcd}
      Fb \ar[rr,"Fp"] \ar[rd,"F(K^{\Ker(p)})",swap]
        && Fe \ar[r,shift left=2mm,"Fd_1"]
              \ar[r,shift right=2mm,"Fd_0",swap]
        & F(p \times_b p) \ar[l] \ar[r,shift left=2mm]
                          \ar[r] \ar[r,shift right=2mm]
        & F(p \times_b p \times_b p) \\
      & F(\CoDesc(\Ker(p))), \ar[ur]
    \end{tikzcd}
  \end{equation}
  and we observe that \( F(\CoDesc(\Ker(p))) \eqv \Desc(F(\Ker(p))) \), from
  which our result follows.
\end{proof}

Naturally, the representable 2-functors \( \bicat A(-,a) \colon \bicat A^\op
\to \CAT \) for each object \(a\) preserve lax descent objects. The
bifibrations \( F \), \( F_D \) of split, respectively discrete,
opfibrations are obtained by composing \( \CAT(-,\Cat) \), respectively  \(
\CAT(-,\Set) \), with the inclusion \( \Cat \to \CAT \). 

We remark that these bifibrations do not satisfy the Beck-Chevalley condition,
as the conclusion of the Bénabou-Roubaud theorem \cite{BR70} does not hold
for \( F \) nor \( F_D \). Indeed, \cite[Remark 7]{Sob04} gives an
example of a functor \( p \) such that \( F_Dp \) is monadic, but \( p \)
is not an effective \( F_D \)-descent morphism. 

We provide another example, given in \cite[Remark 3.3]{LPS23}: let \( p \colon
1 \to b \) be a functor, where \( 1 \) is the terminal category. The functor
\( Fp = \CAT(p,\Cat) \) is monadic if and only if \( p \) is
(essentially) surjective, but, as a consequence of Theorem
\ref{thm:desc.fib.opfib}, \(p\) is an effective \( F \)-descent morphism
if and only if \(p\) is an equivalence.

\chapter{Enriched $\cat V$-functors}
\label{chap:enriched}

Effective descent morphisms for the category \( \VCat \) of \( \cat V
\)-categories were studied in \cite[Section~5]{CH04} when \( \cat V \) is a
(co)complete symmetric monoidal closed thin category, and in \cite[Theorem
9.11]{Luc18}, when \( \cat V \) is a lextensive, cartesian monoidal category
such that the copower functor \( - \pt 1 \colon \Set \to \cat V \) is fully
faithful. Despite the different approaches to the problem, in both works, the
conditions for a \( \cat V \)-functor \(F\) to be an effective descent
morphism in \( \VCat \) are expressed in terms of surjectivity of \(F\) in
chains of hom-objects.

The goal of this chapter is to prove that the same conditions remain
sufficient for a \( \cat V \)-functor to be effective for descent when \( \cat
V \) is a cartesian monoidal category with finite limits, placing the results
of \cite[Section 5]{CH04} (when \( \cat V \) is cartesian monoidal) and
\cite{Luc18} on common ground. We shall prove that if a \( \cat V \)-functor
\( F \) is
\begin{itemize}[noitemsep,label=--]
  \item
    an effective descent morphism on hom-objects,
  \item
    a descent morphism on 2-chains of hom-objects, 
  \item
    an almost descent morphism on 3-chains of hom-objects,
\end{itemize}
in a suitable sense, then \( F \) is an effective descent morphism in \( \VCat
\) (Theorem \ref{thm:desc.vcat}). 

There are two key ideas for this result: if \( \cat V \) is a category with
finite limits, then
\begin{itemize}[noitemsep,label=--]
  \item 
    \( \FamV \) is a lextensive category, and \( - \pt 1 \colon \Set \to \FamV
    \) is fully faithful (Proposition \ref{prop:famv.cat.suits}), so that we
    may apply Theorem \ref{thm:fln.desc} to obtain a description of the
    effective descent morphisms in \( \FamVCat \), and
  \item
    the direct image \( \eta_! \colon \VCat \to \FamVCat \) of \( \eta \colon
    \cat V \to \FamV \) reflects effective descent morphisms.
\end{itemize}
Thus, we conclude that the composite functor reflects effective descent
morphisms as well. We go over the first idea in Section \ref{sect:prop.fam},
where we review the central properties of the free coproduct completion of a
category.  The second idea is obtained via Proposition
\ref{prop:pspb.descent}, so, in Section \ref{sect:embed.vcat.famvcat} we study
the relevant pseudopullbacks and their preservation by the enrichment
2-functor \( \cat V \mapsto \VCat \).

Via the description of the effective descent morphisms in \( \FamVCat \) given
by Lucatelli Nunes's result (Theorem \ref{thm:fln.desc}), we can state the
conditions for a \( \cat V \)-functor \(F\) to be an effective descent
morphism in \( \VCat \) in terms of (effective, almost) descent conditions on
the underlying morphisms of \(F\) on chains of hom-objects. Since these refer
to morphisms in the category \( \FamV \), this motivates the study of
(effective) descent morphisms in the free coproduct completion of a category,
which is carried out in Section \ref{sect:fam.desc}. 

We highlight the relationship between the results of \cite{CH04} and
\cite{Luc18} in Section \ref{sect:famdesc.ex}, where we apply our results on
(effective) descent morphisms in \( \FamV \) to study effective descent
morphisms in \( \VCat \) for special families of categories \( \cat V \).
Among such categories, we draw our attention to the category \( \CHaus \) of
compact Hausdorff spaces and the category \( \Stn \) of Stone spaces, giving a
description of effective descent \( \CHaus \)-functors and effective descent
\( \Stn \)-functors.

\section{Properties of the free coproduct completion}
\label{sect:prop.fam}

Let \( \cat V \) be a category. The \textit{free coproduct cocompletion} of \(
\cat V \), denoted \( \FamV \), consists of
\begin{itemize}[label=--,noitemsep]
  \item
    objects which given by set-indexed families \( (X_j)_{j\in J} \) of objects \(
    X_j \) in \( \cat V \),
  \item
    morphisms \( (X_j)_{j \in J} \to (Y_k)_{k \in K} \) which are given by a
    function \( f \colon J \to K \), and a set-indexed family of morphisms \(
    (\phi_j \colon X_j \to Y_{fj})_{j \in J} \), with \( \phi_j \) in \( \cat
    V \),
\end{itemize}
with suitable identities and composition law. It may also be obtained via the
Grothendieck construction~\cite{Gro71} of the pseudofunctor \( \Set^\op \to
\CAT \) given by \( X \mapsto \cat V^X \) on objects and \( f \mapsto f^* \)
on morphisms, whose fibration we denote by \( P \colon \Fam(\cat V) \to \Set
\).

We recall the following properties:

\begin{proposition}[{\cite[Lemma 3.2]{Pre23}}]
  \label{prop:famv.cat.suits}
  Let \( \cat V \) be a category.
  \begin{enumerate}[label=(\alph*),noitemsep]
    \item
      \label{enum:fam.ext}
      \( \FamV \) is extensive.
    \item
      \label{enum:fam.conn}
      If \( \cat V \) has a terminal object, \( - \pt 1 \colon \Set \to \FamV
      \) is fully faithful.
    \item
      \label{enum:fam.lex}
      If \( \cat V \) has finite limits, so does \( \FamV \).
  \end{enumerate}
\end{proposition}

\begin{proof}
  Property \ref{enum:fam.ext} is well-known, and is present in
  \cite[Proposition 2.4]{CLW93}, for instance. Moreover, it was verified that
  \ref{enum:fam.conn} holds in \cite[Proposition 6.2.1]{BJ01}. Property
  \ref{enum:fam.lex} is also well-known; see \cite{Gra66, Her99, BJ01}.
  Nevertheless, to illustrate the mechanics of \( \FamV \), we will revisit
  the arguments. Recall that

  \begin{itemize}[label=--,noitemsep]
    \item
      \( \Set \) has all (finite) limits,
    \item
      \( \cat V^X \) has all finite limits, given componentwise in \( \cat V
      \),
    \item
      the change-of-base functor \( f^* \colon \cat V^Y \to \cat V^X \)
      preserves finite limits, for every function \( f \colon X \to Y \), 
  \end{itemize}
  so, if \( \cat A \colon \cat J \to \FamV \) is a finite diagram, with \(
  \cat A_j = (P\cat A_j, x_j) \), we consider the limit cone \( \lambda_j
  \colon \lim(P\cat A) \to P\cat A j \), and we define a diagram
  \begin{align*}
    \Phi \colon \cat J &\to \cat V^{\lim(P\cat A)} \\
                     j &\mapsto \lambda_j^*(x_j)
  \end{align*}
  and since \( \cat V^{\lim(P\cat A)} \) has finite limits, \( \lim \Phi \)
  exists. To verify that \( \lim \cat A \iso (\lim P\cat A, \lim \Phi) \),
  given a cone \( (\gamma_j \colon b \to P\cat A_j, \zeta_j \colon w \to
  \gamma_j^*x_j) \), we let \( \omega \colon b \to \lim P\cat A \) be the
  unique function such that \( \gamma_j = \lambda_j \circ \omega \), and we
  observe that \( \omega^*(\lim \Phi) \iso \lim \omega^*\Phi \), since \(
  \omega^* \) preserves limits. Then, there exists a unique \( \xi \colon w
  \to \omega^*(\lim \Phi) \) such that \( \zeta_j = \m {\cat
  V^{-}}_{\lambda_j, \omega} \circ \omega^*\phi_j \circ \xi \), where \(
  \phi_j \colon \lim \Phi \to x_j \) is the limit cone of \( \Phi \), and \(
  \m {\cat V^{-}}_{\lambda_j,\omega} \colon \omega^* \lambda_j^* \iso
  (\lambda_j \circ \omega)^* \) is the natural comparison isomorphism of the
  pseudofunctor.

  Indeed, we have \( (\gamma_j,\zeta_j) = (\lambda_j,\phi_j) \circ
  (\omega,\xi) \), and if \( (\gamma_j,\zeta_j) = (\lambda_j,\phi_j) \circ
  (\theta,\chi) \), then \( \gamma_j = \lambda_j \circ \theta \) which
  confirms \( \theta = \omega \), and \( \zeta_j = \m{\cat
  V^-}_{\lambda_j,\omega} \circ \omega^*\phi_j \circ \chi \), confirming \(
  \chi = \xi \).
\end{proof}

In particular, it follows that
\begin{corollary}[{\cite[p. 10]{Pre23}}]
  \label{cor:famv.catfamv.refl}
  The functor \( \FamVCat \to \CatFamV \) reflects effective descent
  morphisms.
\end{corollary}
\begin{proof}
  Since \( \FamV \) is lextensive and \( - \pt 1 \colon \Set \to \FamV \) is
  fully faithful by Proposition \ref{prop:famv.cat.suits}, we may apply
  Theorem \ref{thm:fln.desc}.
\end{proof}

This result was the original motivation to study the problem of whether \(
\VCat \to \FamVCat \) reflects effective descent morphisms as well. In this
direction, we recall from \cite{Web07} that the inclusion \( \cat V \to \FamV
\) is a 2-cartesian natural transformation:

\begin{proposition}[{\cite[5.15 Proposition]{Web07}}]
  \label{prop:unit.cart}
  For any functor \( F \colon \cat V \to \cat W \), the following diagram
  \begin{equation}
    \label{eq:unit.cart}
    \begin{tikzcd}
      \cat V \ar[d,"F",swap] \ar[r,"\eta"]
             \ar[rd,"\ulcorner",phantom,very near start]
        & \FamV \ar[d,"\Fam(F)"] \\
      \cat W \ar[r,"\eta",swap]
        & \FamW 
    \end{tikzcd}
  \end{equation}
  is a 2-pullback.
\end{proposition}

In particular, we note there is a unique functor \( ! \colon \cat V \to 1 \),
and \( \Fam(!) \colon \FamV \to \Fam(1) \eqv \Set \) is the fibration
associated to \( \FamV \). Thus, by \cite{JS93}, \eqref{eq:unit.cart} is a
pseudopullback when \( \cat W \eqv 1 \). 

The last preliminary result we shall need is the following reformulation of
\cite[Proposition 6.1.5]{BJ01}:

\begin{proposition}
  \label{thm:a.fam.c.when}
  Let \( \cat A \) be a category with coproducts. If \( \cat C \) is a full
  subcategory of \( \cat A \) such that
  \begin{enumerate}[label=(\roman*),noitemsep]
    \item
      \label{enum:ess.surj}
      for all objects \( a \) of \( \cat A \), there exists a set \( J \) and
      objects \( c_j \) of \( \cat C \) for each \( j \in J \) such that \( a
      \iso \sum_{j \in J} c_j \),
    \item
      \label{enum:ff}
      for all objects \( x \), \( d_k \) of \( \cat C \) for each \( k \in K
      \), any morphism \( x \to \sum_{k \in K} d_k \) factors uniquely via
      \( \iota_k \colon d_k \to \sum_{k \in K} d_k \) for some \( k \in K \),
  \end{enumerate}
  then we have an equivalence \( \FamC \eqv \cat A \). 
\end{proposition}

Under the above hypotheses, we conclude that \( \cat A \) is an extensive
category, and that \( \cat C \) (essentially) consists of the
\textit{connected} objects of \( \cat A \) -- that is, those objects \( a \in
\cat A \) such that \( \cat A(a,-) \) preserves coproducts.

\begin{proof}
  We have a coproduct functor 
  \begin{equation}
    \label{eq:coprod.funct}
    \sum \colon \FamC \to \cat A
  \end{equation}
  defined on objects by \( (c_j)_{j \in J} \mapsto \sum_{j \in J} c_j \). We
  observe that \eqref{eq:coprod.funct} is essentially surjective if and only
  if~\ref{enum:ess.surj} holds.

  On morphisms, \eqref{eq:coprod.funct} is given by the composite
  \begin{align*}
    \FamC((c_j)_{j\in J},(d_k)_{k \in K})
      &\iso \prod_{j \in J} \sum_{k \in K} \cat C(c_j,d_k) \\
      &\iso \prod_{j \in J} \sum_{k \in K} \cat A(c_j,d_k) \\
      &\to \prod_{j \in J} \cat A\big(c_j, \sum_{k \in K} d_k\big) \\
      &\iso \cat A\big(\sum_{j \in J} c_j, \sum_{k \in K} d_k \big)
  \end{align*}
  where \( \cat A(c_j,d_k) \to \cat A\big(c_j,\sum_{k \in K} d_k\big) \) is
  given by \( \cat A(c_j,\iota_k) \). We observe that \ref{enum:ff} holds if
  and only if
  \begin{equation*}
    \sum_{k \in K} \cat A(x,d_k) \to \cat A\big(x,\sum_{k \in K} d_k\big) 
  \end{equation*}
  is an isomorphism, so we conclude that \eqref{eq:coprod.funct} is fully
  faithful if and only if \ref{enum:ff} holds.
\end{proof}

One consequential application of Proposition \ref{thm:a.fam.c.when} is that it
allows us to reduce the study of (effective, almost) descent morphisms in \(
\FamV \) to the study of \textit{covers}, that is, morphisms in \( \FamV \) of
the form \( \phi \colon (X_j)_{j\in J} \to Y \).

\begin{lemma}
  \label{lem:e.fameconn}
  Let \( \cat E \) be a pullback-stable class of morphisms in \( \FamV \) that
  is closed under coproducts (as a full subcategory of \( [2, \cat V] \)). We
  have an equivalence of categories \( \cat E \eqv \Fam(\cat E_\conn) \),
  where \( \cat E_\conn \subseteq \cat E \) is the subclass of covers in \(
  \cat E \); that is, those morphisms of the form \( \phi \colon (X_i)_{j \in
  J} \to Y \).
\end{lemma}

\begin{proof}
  If \( (f,\psi) \colon (X_j)_{j \in J} \to (Y_k)_{k \in K} \) is in \( \cat E
  \), we consider the following pullback
  \begin{equation*}
    \begin{tikzcd}
      (X_j)_{j \in f^*k} \ar[r,"\psi|_k"] \ar[d]
             \ar[rd,"\ulcorner",phantom,very near start]
      & Y_k \ar[d] \\
      (X_j)_{j \in J} \ar[r,"{(f,\psi)}",swap]
      & (Y_k)_{k \in K}
    \end{tikzcd}
  \end{equation*}
  for each \( k \in K \). By pullback stability, we find that \( \psi|_k \in
  \cat E_\conn \) for all \(k \in K\), and \( \psi \iso \sum_{k \in K} \psi|_k
  \).

  If \( \phi \colon (V_i)_{i \in I} \to W \) is in \( \cat E_\conn \), and we
  have a commutative square
  \begin{equation}
    \label{eq:sq.fam.e}
    \begin{tikzcd}
      (V_i) \ar[d,"\phi",swap] \ar[r,"{(g,\chi)}"]
        & (X_j)_{j \in J} \ar[d,"{(f,\psi)}"] \\
      W \ar[r,"{(k,\omega)}",swap] & (Y_k)_{k \in K}
    \end{tikzcd}
  \end{equation}
  then \( f(g(i)) = k \), so that \( g(i) \in f^*k \) for all \( i \in I \),
  and hence \eqref{eq:sq.fam.e} factors uniquely as
  \begin{equation*}
    \begin{tikzcd}
      (V_i) \ar[d,"\phi",swap] \ar[r,"{(g,\chi)}"]
        & (X_j)_{j \in f^*k} \ar[r,"\iota_k"] \ar[d,"\psi|_k"]
        & (X_j)_{j \in J} \ar[d,"{(f,\psi)}"] \\
      W \ar[r,"\omega",swap] 
        & Y_k \ar[r,"\iota_k",swap]
        & (Y_k)_{k \in K}
    \end{tikzcd}
  \end{equation*}
  so we may apply Proposition \ref{thm:a.fam.c.when}.
\end{proof}

\section{Embedding $ \VCat \to \FamVCat $}
\label{sect:embed.vcat.famvcat}

In \cite[Corollary 3.8]{FL22}, the authors have shown that the enrichment
2-functor \( (-)\dash\Cat \colon \Bicat \to 2\dash\CAT \) preserves all
weighted, connected 2-limits, where \( \Bicat \) is the 2-category of
bicategories, pseudofunctors and icons. Denoting by \( \SymCat \) the
2-category of symmetric monoidal categories, monoidal functors and monoidal
natural transformations, we obtain the following corollary:

\begin{proposition}[{\cite[Corollary 3.8]{FL22}}]
  \label{prop:enrich.lim}
  The enrichment 2-functor \( \SymCat \to \CAT \) preserves pseudopullbacks.
\end{proposition}
\begin{proof}
  Pseudopullbacks are \( \Cat \)-connected limits.
\end{proof}

It is known that a morphism \( p \colon a \to b \) in a 2-category is fully
faithful if and only if \( p \comma p \eqv 2 \pow a \)\footnote{\( f \comma g
\) is the \textit{comma object} of two morphisms \(f,g\) in a 2-category \(
\bicat A \). \( \cat C \pow x \) is the \textit{power object} of an object
\(x\) in \( \bicat A \) by a category \( \cat C \).}. However, it was shown in
\cite{FL22} that powers and comma objects are \textit{not} \( \Cat
\)-connected limits. Despite this, we can still obtain the following result:

\begin{lemma}[{\cite[Lemma 2.2]{Pre23}}]
  \label{lem:enrich.ff}
  The enrichment 2-functor \( \SymCat \to \CAT \) preserves fully faithful
  functors.
\end{lemma}

\begin{proof}
  Let \( F \colon \cat V \to \cat W \) be a fully faithful monoidal functor,
  and let  \( \cat C \), \( \cat D \) be \( \cat V \)-categories. A \( \cat W
  \)-functor \( \Psi \colon F_!\cat C \to F_!\cat D \) consists of 
  \begin{itemize}[noitemsep,label=--]
    \item
      A function \( \ob \Psi \colon \ob \cat C \to \ob \cat D \) on objects,
    \item
      A morphism \( \Psi_{x,y} \colon F\cat C(x,y) \to F\cat D(\Psi x,\Psi y) \)
      in \( \cat W \) for each pair \( x,y \in \ob \cat C \).
  \end{itemize}
  We claim that we have a \( \cat V \)-functor \( \Phi \colon \cat C \to \cat
  D \) given \( \ob \Phi = \ob \Psi \) and \( \Phi_{x,y} \) is the unique
  morphism \( \cat C(x,y) \to \cat D(\Psi x, \Psi y) \) such that \(
  F\Phi_{x,y} = \Psi_{x,y} \) for all \( x,y \in \ob \cat C \), by full
  faithfulness of \(F\).

  Recalling that \( \un{F_!\cat X} = F\un{\cat X} \circ \e F \) and \(
  \cp{F_!\cat X} = F\cp{\cat X} \circ \m F \) for any \( \cat V \)-category \(
  \cat X \), we note that the following diagrams commute
  \begin{equation*}
    \begin{tikzcd}
      & I \ar[d,"\e F"] \\
      & FI \ar[rd,"F\un{\cat D}"] \ar[ld,"F\un{\cat C}",swap] \\ 
      F\cat C(x,y) \ar[rr,"F\Phi_{x,y}",swap]
      && F\cat D(\Phi x,\Phi y)
    \end{tikzcd}
    \quad
    \begin{tikzcd}[column sep=large]
      (F_!\cat C)(x, y, z)
        \ar[r,"(F\Phi)_{x,y,z}"]
        \ar[d,"\m F",swap]
      & (F_!\cat D)(\Phi x, \Phi y, \Phi z)
        \ar[d,"\m F"] \\
      F(\cat C(x, y, z))
        \ar[r,"F(\Phi_{x,y,z})"]
        \ar[d,"F\cp{\cat C}",swap]
      & F(\cat D(\Phi x, \Phi y, \Phi z))
        \ar[d,"F\cp{\cat D}"] \\
      F\cat C(x, z)
        \ar[r,"F\Phi_{x,z}",swap]
      & F\cat D(\Phi x, \Phi z)
    \end{tikzcd}
  \end{equation*}
  since \(\Psi\) is a \( \cat W \)-functor. Since \( \e F \) and \( \m F \)
  are invertible, and \(F\) is fully faithful, we deduce that \( \Phi \) must
  be a \( \cat V \)-functor, as desired.
\end{proof}

\begin{lemma}[{\cite[p. 9]{Pre23}}]
  \label{lem:key.pspb}
  Let \( \cat V \) be a category with finite limits. We have a pseudopullback
  diagram 
  \begin{equation}
    \label{eq:key.pspb}
    \begin{tikzcd}
      \VCat \ar[r,"\eta_!"] \ar[d] 
            \ar[rd,"\iso" description,phantom]
      & \FamVCat \ar[d] \\
      \Set \ar[r] & \Set\dash \Cat
    \end{tikzcd}
  \end{equation}
  of categories with pullbacks and pullback-preserving functors.
\end{lemma}
\begin{proof}
  We observe that \eqref{eq:key.pspb} is the composite of the enrichment
  2-functor with the diagram
  \begin{equation*}
    \begin{tikzcd}
      \cat V \ar[r,"\eta"] \ar[d]
             \ar[rd,"\ulcorner",phantom,very near start]
      & \FamV \ar[d] \\
      1 \ar[r] & \Set
    \end{tikzcd}
  \end{equation*}
  which is a 2-pullback by \cite[Proposition 5.15]{Web07}, and since \( \FamV
  \to \Set \) is an (iso)fibration, it is in fact a pseudopullback by
  \cite{JS93}. Now the result follows by Proposition \ref{prop:enrich.lim}.
\end{proof}

\begin{theorem}[{\cite[Lemma 3.1]{Pre23}}]
  \label{thm:vcat.famv.cat.refl}
  If \( \cat V \) has finite limits, \( \eta_! \colon \VCat \to \FamVCat \)
  reflects effective descent morphisms.
\end{theorem}

\begin{proof}
  Let \( F \colon \cat C \to \cat D \) be a \( \cat V \)-functor such that \(
  \eta_!F \) is an effective descent \( \FamV \)-functor. Since \( \FamV \to
  \Set \) has fully faithful left and right adjoints, the same holds for \(
  \FamVCat \to \Set\dash \Cat \), since enrichment is a 2-functor which
  preserves fully faithful functors by Lemma \ref{lem:enrich.ff}.

  Thus, \( \FamVCat \to \Set\dash \Cat \) maps \( \eta_!F \) to a descent
  functor, which, in turn, is reflected along \( \Set \to \Set \dash \Cat \)
  to its underlying function on objects, which must be surjective. Since
  surjections are effective descent morphisms in \( \Set \), we may apply
  Proposition \ref{prop:pspb.descent} to conclude that \(F\) is an effective
  descent functor.
\end{proof}

Now, by applying Lucatelli Nunes's criteria for the effective descent
morphisms (Theorem \ref{thm:fln.desc}) to Theorem
\ref{thm:vcat.famv.cat.refl}, we obtain (see also Lemma \ref{lem:e.fameconn}):

\begin{theorem}[{\cite[Theorem 3.3]{Pre23}}]
  \label{thm:desc.vcat}
  If \( \cat V \) is a category with finite limits, and \( F \colon \cat C \to
  \cat D \) is a \( \cat V \)-functor such that 
  \begin{enumerate}[label=(\Roman*),noitemsep]
    \item
      \label{enum:vcat.eff.desc.sings}
      \( F \colon (\cat C(x_0,x_1))_{x_i \in F^*y_i} \to \cat D(y_0,y_1) \) is
      an effective descent morphism, 
    \item
      \label{enum:vcat.desc.pairs}
      \( F \times F \colon (\cat C(x_0,x_1,x_2))_{x_i \in F^*y_i} \to \cat
      D(y_0, y_1, y_2) \) is a descent morphism, and
    \item
      \label{enum:vcat.al.desc.trips}
      \( F \times F \times F \colon (\cat C(x_0, x_1, x_2, x_3))_{x_i \in
      F^*y_i} \to \cat D(y_0, y_1, y_2,y_3) \) is an almost descent morphism
  \end{enumerate}
  in the category \( \FamV \) for all \( y_0,\,y_1,\,y_2,\,y_3 \in \cat V \),
  then \(F\) is an effective descent morphism in \( \VCat \).
\end{theorem}

Unlike Le Creurer's and Lucatelli Nunes's results, which describe the
effective descent morphisms in \( \CatV \) and \( \VCat \) in terms of
morphisms in \( \cat V \), the description we provide for the effective
descent morphisms of \( \VCat \) (Theorem \ref{thm:desc.vcat}) relies on
understanding (effective, almost) descent morphisms in \( \FamV \), prompting
their study. 

\section{Familial descent and effective descent morphisms}
\label{sect:fam.desc}

By Lemma \ref{lem:e.fameconn}, we learn that when studying (effective, almost)
descent morphisms in \( \FamV \), it is enough to the consider covers \( \phi
\colon (X_j)_{j \in J} \to Y \).

There is not much to say about (pullback-stable) epimorphisms in this
generality; they are, tautologically, the jointly epimorphic covers (preserved
by pullbacks). 

Regarding (effective) descent morphisms, it is useful to compute the kernel
pair of a cover \( \phi \colon (X_j)_{j \in J} \to Y \); we use the
construction described in Proposition \ref{prop:famv.cat.suits}. For each
\(j,k \in J\), we have a pullback diagram
\begin{equation}
  \label{eq:fam.ker}
  \begin{tikzcd}
    \phi_j \times_Y \phi_k \ar[r,"\delta_{1,j,k}"] \ar[d,"\delta_{0,j,k}",swap]
      \ar[rd,"\ulcorner",phantom, very near start]
      & X_j \ar[d,"\phi_j"] \\
    X_k \ar[r,"\phi_k",swap] & Y
  \end{tikzcd}
\end{equation}
so the kernel pair of \( \phi \) is given by
\begin{equation*}
  \begin{tikzcd}
    (\phi_j \times_Y \phi_k)_{j,k \in J\times J}
      \ar[r,"{(d_1,\delta_1)}",shift left]
      \ar[r,"{(d_0,\delta_0)}",shift right,swap]
    & (X_j)_{j \in J}
  \end{tikzcd}
\end{equation*}
where \( d_i \colon J \times J \to J \) discards the \(i\)th component, for \(
i = 0,\,1 \).

We also define a category \( \cat D_J \) with set of objects \( J + J^2 \),
and for each \( j,k \in J \), two different morphisms \( (j,k) \to j \) and \(
(j,k) \to k \).

For each cover \( \phi \colon (X_j)_{j \in J} \to Y \), we define a diagram \(
K_\phi \colon \cat D_J \to \cat V \), mapping
\begin{itemize}[label=--,noitemsep]
  \item
    \( (j,k) \to j \) to \( \delta_{1,j,k} \colon \phi_j \times_Y \phi_k \to
    X_j \)
  \item
    \( (j,k) \to k \) to \( \delta_{0,j,k} \colon \phi_j \times_Y \phi_k \to
    X_k \),
\end{itemize}

We recall that in a category with finite limits, regular epimorphisms are the
coequalizers of their kernel pairs. With that in mind, we may obtain the
following result:

\begin{lemma}[{\cite[Lemma 4.1]{Pre23}}]
  \label{lem:famv.reg.epi}
  A cover \( \phi \) is a (pullback-stable) regular epimorphism if and only if
  \( \colim K_\phi \iso Y \) (and the colimit is stable).
\end{lemma}

\begin{proof}
  We shall assume \(J\) is non-empty throughout. We have a natural
  isomorphism
  \begin{equation*}
    [\rightrightarrows, \FamV](\ker \phi, \Delta_{(Z_k)_{k \in K}})
    \iso \sum_{k \in K} [\cat D_J, \cat V](K_\phi, \Delta_{Z_k}),
  \end{equation*}
  and to verify this, note that the composite of \( \ker \phi \) with \( \FamV
  \to \Set \) is given by
  \begin{equation*}
    \begin{tikzcd}
      J \times J \ar[r,shift left,"d_1"]
                 \ar[r,shift right,"d_0",swap]
      & J,
    \end{tikzcd}
  \end{equation*}
  whose coequalizer is terminal. Since the data for a natural transformation
  \( \ker \phi \to \Delta_{(Z_k)_{k \in K}} \) consists of a morphism \(
  (f,\psi) \colon (X_j)_{j \in J} \to (Z_k)_{k \in K} \) such that 
  \begin{equation*}
    (f,\chi) \circ (d_1,\delta_1) = (f,\psi) \circ (d_0,\delta_0),
  \end{equation*}
  there exists a unique \(k \in K\) such that \(fj = k \) for all \( j \in J
  \), and \( \chi_j \circ \delta_{1,i,j} = \chi_i \circ \delta_{0,i,j} \)
  for all \( i,j \in J \). This data corresponds to a unique natural
  transformation \( K_\phi \to \Delta_{Z_k} \).

  So, if \( K_\phi \) has a colimit, then
  \begin{align*}
    \FamV(\colim K_\phi,(Z_k)_{k \in K}) 
      &\iso \sum_{k \in K} \cat V(\colim K_\phi,Z_k) \\
      &\iso \sum_{k \in K} [\cat D_J, \cat V](K_\phi,\Delta_{Z_k}) \\
      &\iso [\rightrightarrows, \FamV](\ker \phi, \Delta_{(Z_k)_{k \in K}}),
  \end{align*}
  hence \( \ker \phi \) has a colimit and \( \colim \ker \phi \iso \colim
  K_\phi \). Conversely, if \( \ker \phi \) has a colimit, then its underlying
  set is the coequalizer of
  \begin{equation*}
    \begin{tikzcd}
      J \times J \ar[r,shift left,"d_1"]
                 \ar[r,shift right,"d_0",swap]
      & J,
    \end{tikzcd}
  \end{equation*}
  which is the terminal object, hence \( \ker \phi \) is connected and can be
  identified with an object of \( \cat V \). Thus,
  \begin{align*}
    \cat V(\colim(\ker \phi), Z)
      &\iso \FamV(\colim(\ker \phi), Z) \\
      &\iso [\rightrightarrows, \FamV](\ker \phi, \Delta_Z) \\
      &\iso [\cat D_J, \cat V](K_\phi, \Delta_Z),
  \end{align*}
  so we conclude that \( \colim \ker \phi \iso \colim K_\phi \).

  If the colimit of \( K_\phi \) is stable under pullback, for each \(j \in J
  \) we write
  \begin{equation*}
    \begin{tikzcd}
      V_j \ar[d,"\phi_j^*(\omega)",swap] \ar[r,"\psi_j"] 
          \ar[rd,"\ulcorner",very near start, phantom]
      & Z \ar[d,"\omega"] \\
      X_j \ar[r,"\phi_j",swap]
      & Y
    \end{tikzcd}
  \end{equation*}
  for the pullback of \( \phi_j \) and a morphism \( \omega \colon Z \to Y \)
  for each \(j \in J\), so that \( Z \iso \colim K_\psi \), hence \( Z \iso
  \colim \ker \psi \). So, if we have a morphism \( \xi \colon (W_l)_{l \in L}
  \to Y \), we can do this procedure for \( \xi_l \) for each \( l \in L \),
  and then take the coproduct of the results, confirming the pullback
  stability of the colimit of \( \ker \phi \).

  Conversely, if the colimit of \( \ker \phi \) is stable under pullback, for
  any morphism \( \omega \colon Z \to Y \), we may consider the pullback
 \begin{equation*}
    \begin{tikzcd}
      (V_j) \ar[d,"\phi_j^*(\omega)",swap] \ar[r,"\psi"] 
          \ar[rd,"\ulcorner",very near start, phantom]
      & Z \ar[d,"\omega"] \\
      X_j \ar[r,"\phi_j",swap]
      & Y
    \end{tikzcd}
  \end{equation*}
  and since \( Y \iso \colim \ker \psi \iso \colim K_\psi \), we immediately
  conclude that the colimit of \( K_\phi \) is stable under pullback.
\end{proof}

To understand effective descent morphisms in \( \FamV \), we recall that
lax descent data for a cover \( \phi \colon (X_j)_{j \in J} \to Y \) consists
of
\begin{itemize}[label=--,noitemsep]
  \item
    a morphism \( (p,\pi) \colon (W_k)_{k \in K} \to (X_j)_{j \in J} \) in
    \( \FamV \),
  \item
    and a morphism \( (\gamma,\Gamma) \colon D_1^*(p,\pi) \to D_0^*(p, \pi) \)
    in \( \FamV \comma (\phi_i \times_Y \phi_j)_{i,j \in J \times J} \)
\end{itemize}
satisfying reflexivity \eqref{eq:refl.desc} and transitivity
\eqref{eq:trans.desc} conditions, namely,
\begin{itemize}[label=--,noitemsep]
  \item
    \( \nu_0 = S_0^*(\gamma, \Gamma) \circ \nu_1 \),
  \item
    \( \theta_{01} \circ D_1^*(\gamma,\Gamma) \circ \theta_{12}
        = D_0^*(\gamma,\Gamma) \circ \theta_{02} 
                               \circ D_2^*(\gamma,\Gamma) \),
\end{itemize}

Since \( \FamV \to \Set \) preserves pullbacks, we recover descent data \(
(p,\pi) \) for the unique morphism \( J \to 1 \), implying that \( K \iso I
\times J \) for some set \(I\), and, under this isomorphism, we have \(p \iso
d_0 \colon I \times J \to J \). Moreover, we note that the underlying \( \Set
\)-pullbacks of \( D_i \colon (\phi_j \times_Y \phi_k)_{j,k \in J\times J} \to
(X_j)_{j \in J} \) for \( i = 0, 1 \) are given by
\begin{equation*}
  \begin{tikzcd}
    K \times J \ar[d,"p \times \id",swap] \ar[r,"d_1"]
      \ar[rd,"\ulcorner",very near start,phantom]
    & K \ar[d,"p"] \\ 
    J \times J \ar[r,"d_1",swap]
    & J
  \end{tikzcd}
  \qquad
  \begin{tikzcd}
    J \times K \ar[d,"\id \times p",swap] \ar[r,"d_0"]
      \ar[rd,"\ulcorner",very near start,phantom]
    & K \ar[d,"p"] \\ 
    J \times J \ar[r,"d_0",swap]
    & J
  \end{tikzcd}
\end{equation*}
and since \( (\id \times p) \circ \gamma = p \times \id \), we deduce that \(
\gamma \) is isomorphic to the function
\begin{align*}
  I \times J \times J &\to J \times I \times J \\
  (i,j,k) &\mapsto (j,i,k),
\end{align*}
and the reflexivity and transitivity conditions are given in components by 
\begin{equation*}
  \label{eq:comp.refl}
  \begin{tikzcd}
    & W_{ij} \ar[ld,"\nu_{1,j}",swap] 
             \ar[rd,"\nu_{0,j}"] \\
    \sigma^*_{0,j}\delta_{1,j,j}^*(\pi_{i,j})
      \ar[rr,"\sigma^*_{0,j}(\Gamma_{i,j,j})",swap]
    && \sigma^*_{0,j}(\delta_{0,j,j}^*(\pi_{i,j})
  \end{tikzcd}
\end{equation*}
\begin{equation}
  \label{eq:comp.trans}
  \begin{tikzcd}[column sep=small]
    & \delta_{1,j,k,l}^*(\delta_{1,j,l}^*(\pi_{i,j}))
      \ar[rr,"\delta_{1,j,k,l}^*(\Gamma_{i,j,l})"]
    && \delta_{1,j,k,l}^*(\delta_{0,j,l}^*(\pi_{i,l}))
      \ar[rd,"\theta_{01,j,k,l}(\pi_{i,l})"] \\
    \delta_{2,j,k,l}^*(\delta_{1,j,l}^*(\pi_{i,j}))
      \ar[ur,"\theta_{12,j,k,l}(\pi_{i,j})"]
      \ar[dr,"\delta_{2,j,k,l}^*(\Gamma_{i,j,k})",swap]
    &&&& \delta_{0,j,k,l}^*(\delta_{0,j,l}^*(\pi_{i,l})) \\
    & \delta_{2,j,k,l}^*(\delta_{0,j,l}^*(\pi_{i,k})) 
      \ar[rr,"\theta_{02,j,k,l}(\pi_{i,k})",swap]
    && \delta_{0,j,k,l}^*(\delta_{1,j,l}^*(\pi_{i,k})) 
      \ar[ru,"\delta^*_{0,j,k,l}(\Gamma_{i,k,l})",swap]
  \end{tikzcd}
\end{equation}
for each \( i \in I \), \( j,k,l \in J \). This observation allows us to prove
the following result:

\begin{lemma}[{\cite[Lemma 4.2]{Pre23}}]
  \label{lem:desc.fam.conn.desc}
  Let \( \phi \colon (X_j)_{j \in J} \to Y \) be a cover in \( \FamV \). We
  have an equivalence 
  \begin{equation}
    \label{eq:desc.fam.conn.desc}
    \Desc(\phi) \eqv \Fam(\Desc_\conn(\phi)),
  \end{equation}
  where \( \Desc_\conn(\phi) \) is the full subcategory of \( \Desc(\phi) \)
  consisting of the connected objects.
\end{lemma}

\begin{proof}
  If \( (p,\pi) \colon (W_{i,j})_{i,j \in I \times J} \to (X_j)_{j \in J} \)
  and \( (\gamma,\Gamma) \) is a lax descent datum for \( \phi \) as given
  above, then for each \(i \in I\), we define \( W_{i,-} = (W_{i,j})_{j \in J}
  \), as well as a morphism \( (\iota_i, \id) \colon W_{i,-} \to (W_{i,j})_{i,j
  \in I \times J} \), where \( \iota_i(j) = (i,j) \).

  We note that the composites \( (p, \pi) \circ (\iota_i, \id) = (\id,
  \pi_{i,-}) \), and \( (\id,\Gamma_{i,-.-}) \colon D_1^*(\id, \pi_{i,-}) \to
  D_0^*(\id, \pi_{i,-}) \) constitute a lax descent datum for \( \phi \), for
  each \( i \in I \). Indeed, Diagrams \eqref{eq:comp.refl} and
  \eqref{eq:comp.trans} commute for each fixed \( i \in I \), confirming
  reflexivity and transitivity for each component.

  Thus, the lax descent datum \( (p,\pi) \), \( (\gamma , \Gamma) \) is the
  coproduct of the lax descent data \( (\id,\pi_{i,-}) \), \( (\id,
  \Gamma_{i,-}) \) for each \( i \in I \).

  Now, we let
  \begin{equation*}
    (\id,\xi) \colon (V_j)_{j \in J} \to (X_j)_{j \in J},
    \qquad (\id,\Xi) \colon (\id, \delta^*_1 \circ \xi) 
                     \to (\id, \delta^*_0 \circ \xi) 
  \end{equation*}
  be a connected lax descent datum, and let \( (g,\chi) \colon (V_j)_{j \in J}
  \to (W_{i,j})_{i,j \in I\times J} \) be a morphism such that \( (d_0,\pi)
  \circ (g,\chi) = (\id, \xi) \) and the following diagram
  \begin{equation*}
    \begin{tikzcd}
      D^*_1(\id, \xi) \ar[d,"{D^*_1(g,\chi)}",swap]
                      \ar[r,"{(\id,\Xi)}"]
        & D^*_0(\id, \xi) \ar[d,"{D^*_0(g,\chi)}"] \\
      D^*_1(d_0,\pi) \ar[r,"{(\gamma,\Gamma)}",swap]
        & D^*_0(d_0,\pi)
    \end{tikzcd}
  \end{equation*}
  commutes. Since \( p \circ g = \id \), we conclude that \( g(j) = (h(j),j)
  \) for a function \( h \colon J \to I \), and since \( \gamma(h(j),j,k) =
  (j,h(k),k) \) for all \(j,k \in J \), we conclude that \( h \) is constant;
  let \(i \in I\) be its value, so that \( g = \iota_i \). We obtain a
  factorization 
  \begin{equation*}
    \begin{tikzcd}
      (V_j)_{j \in J} \ar[r,"{(\id,\chi_{i,-})}"]
        & (W_{i,j})_{j \in J} \ar[r,"{(\iota_i,\id)}"]
        & (W_{i,j})_{i,j \in I \times J},
    \end{tikzcd}
  \end{equation*}
  so, we may apply Proposition \ref{thm:a.fam.c.when} to conclude our proof.
\end{proof}

In fact, by noticing that the underlying objects, morphisms and properties for
each connected descent datum lie in a slice category of a fiber of \( \FamV
\to \Set \), we deduce that:

\begin{corollary}
  The category \( \Desc_\conn(\phi) \) is the lax descent category of the
  following diagram:
  \begin{equation}
    \label{eq:fam.free.cosimp}
    \begin{tikzcd}
      \cat V^J \comma (X_j)_{j \in J}
        \ar[r,shift left=2.5mm,"D_1^*"]
        \ar[r,shift right=2.5mm,"D_0^*",swap]
      & \cat V^{J \times J} \comma (\phi_i \times \phi_j)_{i,j \in J \times J}
        \ar[r,shift left=2.5mm,"D_2^*"]
        \ar[l,"S_0^*" description]
        \ar[r,"D_1^*" description]
        \ar[r,shift right=2.5mm,"D_0^*",swap]
      & \cat V^{J \times J \times J}
          \comma (\phi_i \times_Y \phi_j \times_Y \phi_k)_{i,j,k \in J \times J
          \times J}.
    \end{tikzcd}
  \end{equation}
\end{corollary}

\begin{remark}
  \label{rem:vk.colims}
  This corollary provides us with another point of view; let \( \cat D^+_J \)
  be the category whose set of objects is given by \( J + J^2 + J^3 \),
  containing \( \cat D_J \) as a subcategory, such that for each triple \(
  i,j,k \in J \), we have three distinct morphisms \( (i,j,k) \to (j,k) \), \(
  (i,j,k) \to (i,k) \) and \( (i,j,k) \to (i,j) \), such that the following
  diagrams commute:
  \begin{equation*}
    \begin{tikzcd}
      (i,j,k) \ar[r] \ar[d] & (i,j) \ar[d] \\
      (i,k) \ar[r] & i
    \end{tikzcd}
    \quad
    \begin{tikzcd}
      (i,j,k) \ar[r] \ar[d] & (j,k) \ar[d] \\
      (i,j) \ar[r] & j
    \end{tikzcd}
    \quad
    \begin{tikzcd}
      (i,j,k) \ar[r] \ar[d] & (i,k) \ar[d] \\
      (j,k) \ar[r] & k
    \end{tikzcd}
  \end{equation*}
  and for each \( j \in J \), a morphism \( j \to (j,j) \) such that both of
  the composites below are the identity:
  \begin{equation*}
    \begin{tikzcd}
      j \ar[r] & (j,j) \ar[r,shift left] \ar[r,shift right]
               & j
    \end{tikzcd}
  \end{equation*}

  We define a diagram \( K^+_\phi \colon \cat D^+_J \to \cat V \), extending
  \( K_\phi \), where \( (i,j,k) \mapsto \phi_i \times_Y \phi_j \times_Y
  \phi_k \), and the morphisms from objects in \( J^3 \) to \( J^2 \) are
  mapped to the respective projections, while the morphisms from \( J \) to \(
  J^2 \) are mapped to the respective diagonals \( \sigma_{0,i} \colon X_i \to
  \phi_i \times_Y \phi_i \). It can be shown that the 2-limit of the
  composite
  \begin{equation*}
    \begin{tikzcd}
      (\cat D^+_J)^\op \ar[r,"(K^+_\phi)^\op"]
        & \cat V^\op \ar[r,"\cat V \comma -"]
        & \Cat
    \end{tikzcd}
  \end{equation*}
  is equivalent to \( \Desc_\conn(\phi) \); by taking products of categories,
  we recover Diagram \eqref{eq:fam.free.cosimp}.
\end{remark}

\begin{theorem}[{\cite[Theorem 4.3]{Pre23}}]
  \label{thm:fam.eff.desc}
  Let \( \phi \colon (X_j)_{j \in J} \to Y \) be a cover in \( \FamV \).
  The following are equivalent:
  \begin{enumerate}[label=(\roman*),noitemsep]
    \item
      \label{enum:fam.eff.desc}
      \( \phi \) is an effective descent morphism.
    \item
      \label{enum:slice.conn.desc}
      We have an equivalence \( \cat V \comma Y \eqv \Desc_\conn(\phi) \).
  \end{enumerate}
\end{theorem}

\begin{proof}
  First, we observe that \( \Fam(\cat V \comma Y) \eqv \FamV \comma Y \),
  since for any morphism \( \phi \colon (W_j)_{i \in J} \to Y \), we have \(
  \phi = \prod_{j \in J} \phi_j \) as objects in \( \FamV \comma Y \), and 
  if we have a commutative triangle
  \begin{equation*}
    \begin{tikzcd}
      W \ar[rr,"{(f,\omega)}"]
        \ar[rd,"\psi",swap]
      && (X_j)_{j \in J} \ar[ld,"\phi"] \\
      & Y
    \end{tikzcd}
  \end{equation*}
  then there exists a unique \(j \in J\) (given by \(f\)) factoring \(
  (f,\omega) = (\iota_j,\id) \circ \omega \) uniquely, so we may apply
  Theorem~\ref{thm:a.fam.c.when}.

  Since the comparison functor \( \mathcal K^{\Ker(\phi)} \colon \FamV \comma
  Y \to \Desc(\phi) \) preserves connected objects, we obtain an equivalence
  \begin{equation}
    \label{eq:k.fam.k.conn}
    \mathcal K^{\Ker(\phi)} \eqv \Fam(\mathcal K_\conn^{\Ker(\phi)}),
  \end{equation}
  where \( \mathcal K_\conn^{\Ker(\phi)} \colon \cat V \comma Y \to
  \Desc_\conn(\phi) \) is the restriction of \( \mathcal K^{\Ker(\phi)} \) to
  the connected objects.

  We have \ref{enum:fam.eff.desc} \( \implies \) \ref{enum:slice.conn.desc},
  since we have \eqref{eq:k.fam.k.conn}, and \( \Fam \) reflects equivalences,
  as the unit is 2-cartesian. Conversely, \ref{enum:slice.conn.desc} \(
  \implies \) \ref{enum:fam.eff.desc} follows immediately by
  \eqref{eq:k.fam.k.conn}.
\end{proof}

\section{Examples}
\label{sect:famdesc.ex}

The study of descent morphisms in \( \FamV \) for certain categories \( \cat V
\) with finite limits inspired us to highlight the following specialization of
Lemma \ref{lem:famv.reg.epi}:

\begin{lemma}
  \label{lem:colim.kphi.join}
  Let \( \phi \colon (X_j)_{j \in J} \to Y \) be a cover in \( \FamV \) such
  that for all \(j \in J\), \( \phi_j \) is a monomorphism; that is, \(
  (X_j)_{j \in J} \) is a family of subobjects of \(Y\).  If the kernel pair
  of \( \phi \) has a coequalizer, then \( \colim K_\phi \iso \Join_{j \in J}
  X_j \) as a subobject of \(Y\), where \( K_\phi \) is given as in Lemma
  \ref{lem:famv.reg.epi}.
\end{lemma}

\begin{proof}
  Let \( \xi \colon \colim K_\phi \to Y \) be the unique morphism such that \(
  \phi_j = \xi \circ q_j \) for all \( j \in J \), where \( q \colon (X_j)_{j
  \in J} \to \colim K_\phi \) is the coequalizer, which is pullback-stable by
  hypothesis. 

  It is enough to prove that \( \xi \) is a monomorphism. We consider the
  following diagram in \( \FamV \)
  \begin{equation*}
    \begin{tikzcd}
      (X_j \land X_k)_{j,k \in J \times J} \ar[r] \ar[d]
             \ar[rd,"\ulcorner",phantom,very near start]
        & (\pi^*_0(X_j))_{j \in J} \ar[r,"\pi_0|_j"] \ar[d]
             \ar[rd,"\ulcorner",phantom,very near start]
        & (X_j)_{j \in J} \ar[d,"q"] \\
      (\pi^*_1(X_j))_{j \in J} \ar[r] \ar[d,"\pi_1|_j"]
             \ar[rd,"\ulcorner",phantom,very near start]
        & \xi \times_Y \xi \ar[r,"\pi_0"] \ar[d,"\pi_1",swap]
             \ar[rd,"\ulcorner",phantom,very near start]
        & \colim K_\phi \ar[d,"\xi"] \\
      (X_j)_{j \in J} \ar[r,"q",swap]
        & \colim K_\phi \ar[r,"\xi",swap]
        & Y
    \end{tikzcd}
  \end{equation*}
  whose squares are pullbacks. Since \( q \) is pullback-stable, it follows
  that \( (X_j \land X_k)_{j,k \in J \times J} \to (\pi^*_i(X_j))_{j \in J} \)
  is a regular epimorphism for \( i = 0,1 \). Its kernel pair is the kernel
  pair of \( (X_j \land X_k)_{j,k \in J \times J} \to (X_j)_{j \in J} \),
  hence \( \pi_i|_j \colon \pi^*_i(X_j) \iso X_j \) is an isomorphism for all
  \(j \in J \) and \( i = 0,1 \). We observe that \( q^* \) is conservative,
  so \( \pi_i \) is an isomorphism for \( i=0,1 \). But \( \pi_0, \pi_1 \) is
  the kernel pair of \( \xi \), so it must be a monomorphism.

  If \( W \) is a subobject of \(Y\) such that \( X_j \leq W \) for all \(j\),
  then the above observation (with \(W\) replacing \(Y\)) also confirms that \(
  \colim K_\phi \leq W \). Thus, \( \colim K_\phi \iso \Join_j X_j \) in the
  (thin) category of subobjects of \(Y\).
\end{proof}

Thus, if a cover \( \phi \colon (X_j)_{j\in J} \to Y \) of monomorphisms
is a descent morphism in \( \FamV \), we conclude that \( Y \iso \Join_{j \in
J} X_j \), a perspective that is helpful when \( \cat V \) is thin or regular.

\subsection{Meet semilattices}

Let \( \cat V \) be a thin category. A morphism \( \phi \colon (X_j)_{j \in J}
\to Y \) in \( \FamV \) is the assertion that ``for all \(j \in J\), we have
\( X_j \leq Y \)''. Therefore, we simply write \( (X_j)_{j \in J} \leq Y \) in
this setting.

A thin category \( \cat V \) is said to be a \textit{meet semilattice with a
top element} if \( \cat V \) is a thin category with finite limits, which are
called \textit{(finite) meets} in this context. 

\begin{lemma}[{\cite[Lemma 4.4]{Pre23}}]
  \label{lem:famv.epis.vthin}
  Let \( \cat V \) be a meet semilattice with a top element, and let \(
  (X_j)_{j \in J} \leq Y \) be a cover.
  \begin{enumerate}[label=(\alph*),noitemsep]
    \item
      \label{enum:vthin.famv.epi}
      It is an epimorphism if and only if \( J \) is non-empty.
    \item
      \label{enum:vthin.famv.al.desc}
      If it is an epimorphism, then it is pullback-stable.
    \item
      \label{enum:vthin.famv.reg}
      It is a regular epimorphism if and only if \( \Join_{j \in J} X_j \iso Y
      \).
    \item
      \label{enum:vthin.famv.desc}
      It is a pullback-stable regular epimorphism if and only if we have
      \begin{equation}
        \label{eq:distr.vthin}
        Z \iso \Join_{j \in J} Z \meet X_j.
      \end{equation}
      for all \( Z \leq Y \).
    \item
      \label{enum:vthin.famv.eff.desc}
      If it is a descent morphism, and \( \cat V \comma Y \) is (co)complete,
      then it is an effective descent morphism if, and only if, for every
      family \( (W_j)_{j \in J} \) with \( W_j \leq X_j \) for all \( j \in J
      \), satisfying
      \begin{equation*}
        W_j \meet X_i \iso X_j \meet W_i
      \end{equation*}
      for every pair \( i,j \in J \), we have
      \begin{equation*}
        X_j \meet \Join_{i \in J} W_i \iso W_j
      \end{equation*}
      for all \( j \in J \).
  \end{enumerate}
\end{lemma}

\begin{proof}
  If \( (X_j)_{j \in J} \leq Y \), then for all \( j \in J \), \( X_j \leq Y
  \) is an epimorphism. Thus, we conclude that \( (X_j)_{j \in J} \leq Y \) is
  an epimorphism if and only if the underlying function \( J \to 1 \) is an
  epimorphism, which is the case if and only if \( J \) is non-empty,
  confirming \ref{enum:vthin.famv.epi}.

  If \( (X_j)_{j \in J} \) is an epimorphism, then for all \( Z \leq Y \), we
  have \( (X_j \meet Z)_{j \in J} \leq Z \), which is still an epimorphism, as
  \( J \) is non-empty, proving \ref{enum:vthin.famv.al.desc}.

  We also have \ref{enum:vthin.famv.reg} as a consequence of Lemma
  \ref{lem:famv.reg.epi}, and the condition for pullback-stability is
  precisely \eqref{eq:distr.vthin}, giving \ref{enum:vthin.famv.desc}. 

  Let \( \phi \colon (X_j)_{j \in J} \leq Y \) be a pullback-stable regular
  epimorphism. We have that \( \phi \) is an effective descent morphism if and
  only if \( \mathcal K^{\Ker(\phi)}_\conn \colon \cat V \comma Y \to
  \Desc_\conn(\phi) \) is essentially surjective, by
  Theorem~\ref{thm:fam.eff.desc}.

  We highlight that connected descent data for \( \phi \) consists of a family
  of subobjects \( W_j \leq X_j \) indexed by \( j \in J \) such that \( W_j
  \land X_i \iso X_j \land W_i \) for all \(i,j \in J \); the reflexivity and
  transitivity conditions are automatically satisfied, as \( \cat V \) is
  thin. 

  Thus, \( \mathcal K^{\Ker(\phi)}_\conn \) is essentially surjective if and
  only if for all connected descent data \( (W_j)_{j \in J} \leq (X_j)_{j \in
  J} \), there exists \( Z\leq Y \) such that \( X_j \meet Z \iso W_j \).
  Given that \( \cat V \comma Y \) is (co)complete, we let \( Z \iso \Join_{j
  \in J} W_j \). We have
  \begin{equation*}
    X_j \land Z  \iso \Join_{i \in J} X_j \land W_i
                 \iso \Join_{i \in J} W_j \land X_i \iso W_j,
  \end{equation*}
  which concludes our proof.
\end{proof}

A bounded meet semilattice \( \cat V \) is said to be a \textit{Heyting
lattice}\footnote{Also known as \textit{implicative semilattices} \cite{Nem65}
and \textit{Brouwerian semilattices} \cite{Köh81}.} if \( \cat V \) is
cartesian closed; that is, the functor \( A \meet - \) has a right adjoint
functor for all objects \( A \). As a corollary, we obtain:

\begin{corollary}[{\cite[Corollary 4.5]{Pre23}}]
  If \( \cat V \) is a Heyting semilattice, regular epimorphisms in \( \FamV
  \) are pullback-stable.
\end{corollary}
\begin{proof}
  Since \( Z \meet - \) preserves joins, \eqref{eq:distr.vthin} is always
  satisfied.
\end{proof}

\begin{corollary}
  If \( \cat V \) is a (co)complete lattice, pullback-stable regular
  epimorphisms in \( \FamV \) are effective for descent.
\end{corollary}
\begin{proof}
  For all \( Y \), \( \cat V \comma Y \) is cocomplete.
\end{proof}

Combining both of the previous results yields:

\begin{corollary}[{\cite[Corollary 4.6]{Pre23}}]
  \label{cor:comp.heyt.lat}
  If \( \cat V \) is a (co)complete Heyting (semi)lattice, regular
  epimorphisms in \( \FamV \) are effective for descent.
\end{corollary}

So, we recover one implication of \cite[Theorem 5.4]{CH04}:

\begin{theorem}[{\cite[Theorem 4.7]{Pre23}}]
  \label{thm:eff.desc.vcat}
  Let \( \cat V \) be a (co)complete Heyting lattice, and let \( F \colon \cat
  C \to \cat D \) be a \( \cat V \)-functor. If 
  \begin{equation}
    \label{eq:vcat.join}
    \Join_{x_i \in F^*y_i} \cat C(x_0, x_1,x_2) \iso \cat D(y_0, y_1, y_2) 
  \end{equation}
  for all \( y_0, y_1, y_2 \in \cat D \), then \(F\) is an effective descent
  \( \cat V \)-functor.
\end{theorem}

\begin{proof}
  By hypothesis, \ref{enum:vcat.desc.pairs} is satisfied, and since \(F\) is
  surjective on objects (consider \eqref{eq:vcat.join} with \( y_0 = y_1 = y_2
  \)), it follows by Lemma~\ref{lem:famv.epis.vthin} that condition
  \ref{enum:vcat.al.desc.trips} holds. Moreover, if we consider
  \eqref{eq:vcat.join} with \( y_1 = y_2 \), so that \( \cat D(y_1,y_2) \iso 1
  \), we have
  \begin{equation*}
    \cat D(y_0,y_1) 
      \iso \Join_{x_i \in F^*y_i} \cat C(x_0,x_1,x_2)
      \leq \Join_{x_i \in F^*y_i} \cat C(x_0,x_1),
  \end{equation*}
  and since \( \cat C(x_0,x_1) \leq \cat D(y_0,y_1) \) for all \( x_i \in
  F^*y_i \), we conclude that \( (\cat C(x_0,x_1))_{x_i\in F^*y_i} \leq \cat
  D(y_0,y_1) \) is a regular epimorphism in \( \FamV \), which is effective
  for descent by Corollary \ref{cor:comp.heyt.lat}, guaranteeing
  \ref{enum:vcat.eff.desc.sings}. Thus, Theorem \ref{thm:desc.vcat} can be
  applied to conclude that \(F\) is effective for descent.
\end{proof}

\begin{remark}
  \label{rem:link}
  As alluded to in the Introduction, Theorem \ref{thm:desc.vcat} confirms
  the link between the idea of ``chain-surjectivity'' conditions of
  \cite{Cre99} and the ``chain-surjectivity'' of \cite{JS02b}, as evidenced by
  Theorem~\ref{thm:eff.desc.vcat}.
\end{remark}

\subsection{Regular categories}

The same ideas work here, if we employ the (regular epi, mono)-factorization
system of a regular category.

\begin{lemma}[{\cite[Lemma 4.8]{Pre23}}]
  \label{lem:vreg.famv}
  Let \( \cat V \) be a regular category, and let \( \phi \colon (X_j)_{j \in
  J} \to Y \) be a cover. For each \( j \in J \), we consider the (regular
  epi,mono)-factorization of \( \phi_j \), given by
  \begin{equation}
    \label{eq:fam.fact}
    \begin{tikzcd}
      X_j \ar[r,"\pi_j"] & M_j \ar[r,"\iota_j"] & Y,
    \end{tikzcd}
  \end{equation}
  where \( \pi_j \) is a descent morphism, and \( \iota_j \) is a
  monomorphism for all \(j \in J \). Thus, we consider the cover \( \iota
  \colon (M_j)_{j \in J} \to Y \).

  \begin{enumerate}[label=(\alph*),noitemsep]
    \item
      \label{enum:vreg.al.desc}
      \( \phi \) is a (pullback-stable) epimorphism if and only if \( \iota \)
      is a (pullback-stable) epimorphism.
    \item
      \label{enum:vreg.desc}
      \( \phi \) is a (pullback-stable) regular epimorphism if and only if \(
      \iota \) is a (pullback-stable) regular epimorphism.
    \item
      \label{enum:vreg.eff.desc}
      If \( \pi_j \) is an effective descent morphism for all \(j \in J\),
      then \( \phi \) is effective for descent if and only if \( \iota \)
      is effective for descent.
  \end{enumerate}
\end{lemma}

\begin{proof}
  The factorization \eqref{eq:fam.fact} gives a factorization \( \phi = \iota
  \circ (\id, \pi) \) in \( \FamV \), and since \( \pi_j \) is a descent
  morphism for all \(j\), \( (\id, \pi) \) is a coproduct of descent
  morphisms, therefore it is a descent morphism in \( \FamV \). Thus, we
  obtain \ref{enum:vreg.al.desc} and \ref{enum:vreg.desc} by composition and
  cancellation \cite[Propositions 1.3 and 1.5]{JST04}.

  Moreover, if \( \pi_j \) is effective for descent for all \(j \in J \), then
  so is \( (\id,\pi) \), and \ref{enum:vreg.eff.desc} follows by \cite[Theorem
  4.5]{JT97}, as the basic bifibration respects the BED (see \cite[{}4.4,
  4.6]{JT97}).
\end{proof}

Thus, the study of (effective) descent covers in \( \FamV \) can be reduced
to the study of (effective) descent covers of monomorphisms (and effective
descent morphisms in \( \cat V \)). When applied to the study of effective
descent morphisms in \( \VCat \), we obtain:

\begin{theorem}[{\cite[Theorem 4.9]{Pre23}}]
  \label{thm:vcat.desc.reg}
  Let \( \cat V \) be a regular category, let \( F \colon \cat C \to \cat D \)
  be a \( \cat V \)-functor, and consider the hom-covers
  \begin{align*}
    F = (F_{x_0,x_1})_{x_i \in F^*y_i} 
      &\colon (\cat C(x_0,x_1))_{x_i \in F^*y_i} \to \cat D(y_0,y_1)  \\
    F \times F = (F_{x_0,x_1,x_2})_{x_i \in F^*y_i} 
      &\colon (\cat C(x_0, x_1, x_2))_{x_i \in F^*y_i} 
        \to \cat D(y_0,y_1,y_2) \\
    F \times F \times F = (F_{x_0,x_1,x_2,x_3})_{x_i \in F^*y_i} 
      &\colon (\cat C(x_0,x_1,x_2, x_3))_{x_i \in F^*y_i} 
        \to \cat D(y_0,y_1,y_2,y_3) 
  \end{align*}
  and their respective (regular epi, mono)-factorizations
  \begin{align*}
    F_{x_0,x_1} &= I_{x_0,x_1} \circ P_{x_0,x_1}, \\
    F_{x_0,x_1,x_2} &= I_{x_0,x_1,x_2} \circ P_{x_0,x_1,x_2}, \\
    F_{x_0,x_1,x_2,x_3} &= I_{x_0,x_1,x_2,x_3} 
                            \circ P_{x_0,x_1,x_2,x_3}.
  \end{align*}
  If
  \begin{enumerate}[label=(\alph*),noitemsep]
    \item
      \label{enum:vfn.eff.desc.i}
      \( P_{x_0,x_1} \) is an effective descent morphism for all \( x_0,
      x_1 \),
    \item
      \label{enum:vfn.eff.desc.ii}
      \( (I_{x_0,x_1})_{x_i \in F^*y_i} \) is an effective descent
      morphism,
    \item
      \label{enum:vfn.desc}
      \( (I_{x_0,x_1,x_2})_{x_i \in F^*y_i} \) is a descent morphism, and
    \item
      \label{enum:vfn.al.desc}
      \( (I_{x_0,x_1,x_2,x_3})_{x_i \in F^*y_i} \) is an almost descent
      morphism,
  \end{enumerate}
  then \(F\) is an effective descent morphism in \( \VCat \).
\end{theorem}

\begin{proof}
  By Lemma \ref{lem:vreg.famv}, 
  \begin{itemize}[noitemsep,label=--]
    \item
      conditions \ref{enum:vfn.eff.desc.i} and \ref{enum:vfn.eff.desc.ii}
      together guarantee \ref{enum:vcat.eff.desc.sings},
    \item
      conditions \ref{enum:vfn.desc} and \ref{enum:vfn.al.desc} respectively
      guarantee \ref{enum:vcat.desc.pairs} and \ref{enum:vcat.al.desc.trips},
  \end{itemize}
  so that Theorem \ref{thm:desc.vcat} can be applied.
\end{proof}

The above list of conditions can be further reduced if \( \cat V \) satisfies
extra properties. For instance, if \( \cat V \) is Barr-exact, or locally
cartesian closed, then descent morphisms are effective descent morphisms, so
condition \ref{enum:vfn.eff.desc.i} is redundant.

More specifically, a \( \CHaus \)-functor \( F \colon \cat C \to \cat D \)
satisfying \ref{enum:vfn.eff.desc.ii}, \ref{enum:vfn.desc} and
\ref{enum:vfn.al.desc} is an effective descent morphism in \( \CHaus \dash
\Cat \), as \( \CHaus \) is an Barr-exact category \cite{MR20}. Moreover, it
is shown therein that \( \Stn \) is a regular category, thus if a \( \Stn
\)-functor satisfies all of the hypotheses of Theorem \ref{thm:vcat.desc.reg},
then it is an effective descent \( \Stn \)-functor.

\chapter{Generalized internal multicategory functors}
\label{chap:internal-multi}

A \textit{multicategory} is a categorical structure which models the notion of
``multimorphisms'': morphisms which map from a (possibly empty) finite string
of inputs to a single output. The quintessential example of such a structure
is the multicategory \( \Vect \) of vector spaces over some field \( \mathbb F
\), and \textit{multilinear maps}. A multilinear map \( f \colon (V_1, \ldots,
V_n) \to W \) has a finite string \( V_1, \ldots, V_n \) of vector spaces as
the domain, and a vector space \( W \) as codomain. It consists of a function
\begin{equation*}
  f \colon V_1 \times \ldots \times V_n \to W 
\end{equation*}
which is linear in each component:
\begin{equation*}
  f(v_1,\ldots,v_i+\lambda w,\ldots,v_n)
    = f(v_1,\ldots,v_i,\ldots,v_n) + \lambda f(v_1,\ldots,w,\ldots,v_n),
\end{equation*}
where \( v_j \in V_j \) for all \(j \in \{1,\ldots, n\}\), \( w \in V_i \) for
each \( i \in \{1, \ldots, n\} \), and \( \lambda \in \mathbb F \). In case \(
n = 0 \), a multilinear map \(f \colon () \to W \) is simply a vector \( f \in
W \), or equivalently, a linear map \( f \colon \mathbb F \to W \). In case \(
n = 1 \), a multilinear map \(f \colon (V_1) \to W \) is an ordinary linear
map.

As is the case with categories, multicategories also have an adequate
\textit{composition operation}. In the case of \( \Vect \), if we have a
finite string of multilinear maps \( g_1, \ldots, g_n \) given by
\begin{equation*}
  g_j \colon (U_{j1}, \ldots, U_{jk_j}) \to V_j,
\end{equation*}
then we have the composite multilinear map
\begin{equation*}
  f \circ (g_1, \ldots, g_n) 
    \colon (U_{11}, \ldots, U_{1k_1}, \ldots, U_{n1}, \ldots, U_{nk_n})
    \to W
\end{equation*}
whose underlying function is given by \( f \circ (g_1 \times \ldots \times
g_n) \). Of course, the identity linear map \( \id_W \colon (W) \to W \) is a
(multi)linear map, and these satisfy suitable associativity and unit laws.
Naturally, if we consider the multilinear maps whose domain is a string of
length 1, we precisely recover the ordinary category of vector spaces and
linear maps.

The notion of multicategory can be traced back to \cite[p. 103]{Lam69}, where
it was developed for the purpose of studying deductive systems in logic, and
it has since found applications in algebraic topology and higher category
theory.  A comprehensive introduction to these categorical structures is given
in \cite{Lei04}. 

The composition operation of \( \Vect \) carries an underlying structure on
the domains, given by \textit{concatenation} of strings, as does the identity
multilinear map, \textit{casting} each vector space as a string of vector
spaces of length 1. These operations are well modeled by the multiplication
and unit natural transformations for the \textit{free monoid monad} \( (-)^*
\) on \( \Set \).  Indeed, the multicategory of vector spaces may be described
by a span of functions
\begin{equation*}
  \begin{tikzcd}
    \{ \text{set of vector spaces} \}^*
    & \{ \text{multilinear maps} \} \ar[l,"\text{domain}",swap]
                                    \ar[r,"\text{codomain}"] 
    & \{ \text{set of vector spaces} \}
  \end{tikzcd}
\end{equation*}
and the identity and composition operation, as well as the associativity and
unit laws, can be described diagrammatically as well, via the monad structure
of \( (-)^* \), and its properties. Let \( \Vect_0 \) be the set of vector
spaces, and \( \Vect_1 \) be the set of multilinear maps. We obtain the set \(
\Vect_2 \) of ``multicomposable'' pairs of multilinear maps via the pullback
\begin{equation*}
  \begin{tikzcd}
    \Vect_2 \ar[r] \ar[d] \ar[dr,"\ulcorner",phantom,very near start]
    & \Vect_1 \ar[d,"\text{domain}"] \\
    \Vect_1^* \ar[r,"\text{codomain}^*",swap]
    & \Vect_0^*
  \end{tikzcd}
\end{equation*}
so that the composition operation is given by
\begin{equation*}
  \begin{tikzcd}
    && \Vect_2 \ar[dd,dashed,no head,
                   "\text{composition}" description,shorten=-1mm] 
               \ar[rd] \ar[ld] \\
    & \Vect_1^* \ar[ld,"\text{domain}^*" description] 
                \ar[rd,"\text{codomain}^*" description]
    && \Vect_1  \ar[ld,"\text{domain}" description] 
                \ar[rd,"\text{codomain}" description] \\
    \Vect_0^{**} \ar[ddr,"\text{concatenation}",swap]
    && \Vect_0^* \ar[d,dashed,shorten=-1mm] 
    && \Vect_0   \ar[ddl,equal] \\
    && \Vect_1  \ar[ld,"\text{domain}" description] 
                \ar[rd,"\text{codomain}" description] \\
    & \Vect_0^* && \Vect_0 \\
  \end{tikzcd}
\end{equation*}
and the identity maps are given by
\begin{equation*}
  \begin{tikzcd}[column sep=huge]
    \Vect_0 \ar[d,"\text{cast}",swap] 
    & \Vect_0 \ar[l,equal] \ar[r,equal] 
              \ar[d,"\text{identity}" description,shorten=-1mm] 
    & \Vect_0 \ar[d,equal]\\
    \Vect_0^* 
    & \Vect_1 \ar[l,"\text{domain}" description]
              \ar[r,"\text{codomain}" description] 
    & \Vect_0
  \end{tikzcd}
\end{equation*}

This diagrammatic description lends itself to the ``internalization'' of the
notion of multicategory to any category \( \cat V \) with pullbacks, provided
we also replace the free monoid monad on \( \Set \) by an arbitrary monad
\(T=(T,m,e) \) on \( \cat V \). Incidentally, this also allows for the
``shape'' of the domain to be more general than ``finite strings''. Indeed,
these ideas gave rise to \textit{\(T\)-catégories}, first defined in
\cite{Bur71}, and later studied by \cite{Her00} when \(T\) is a
\textit{cartesian} monad. In these works, generalized multicategories are
defined to be monads in the bicategory \( \Span_T(\cat V) \).

The main theme of this thesis is to obtain a unified perspective on the
effective descent morphisms in generalized categorical structures, and the
purpose of this chapter, covering the work done in \cite{PL23}, is to provide
an understanding of these morphisms in the category \( \CatTV \) of
\textit{T-categories} internal to \( \cat V \). We will undertake two
approaches.

Our first approach to the study of effective descent morphisms in \( \CatTV \)
can be summed up in four steps:
\begin{itemize}[label=--,noitemsep]
  \item
    we construct \( \CatTV \) as a 2-equalizer of a diagram of categories of
    essentially algebraic theories internal to \( \cat V \),
  \item
    we recall from \cite[Proposition 3.2.4]{Cre99} that effective descent
    morphisms in essentially algebraic theories internal to \( \cat V \) can
    be described via descent conditions on the underlying data,
  \item
    we recall the description of effective descent morphisms of a
    pseudoequalizer (isoinserter) via Proposition \ref{prop:pseq.descent},
  \item
    we confirm that the embedding of \( \CatTV \) into the associated
    pseudoequalizer reflects effective descent morphisms.
\end{itemize}
Section \ref{sect:refl.tgrph} illustrates the tools and techniques used for
the construction of \( \CatTV \) in a simpler \(T\)-structure, that of
\textit{reflexive} \(T\)\textit{-graphs}, and the full construction is carried
out in Section \ref{sect:cat.tv}. Afterwards, Section \ref{sect:bilim} is
devoted to confirming that effective descent morphisms are reflected along the
embedding of \( \CatTV \) into the associated pseudoequalizer.

In Section \ref{sect:direct}, we provide a second method to obtain a
description of the effective descent functors of \( T \)-categories.  Here, we
employ the ideas of \cite{Cre99} to extend his results to our setting, by
directly studying the ``sketch'' of these generalized multicategories.

\section{Reflexive $T$-graphs}
\label{sect:refl.tgrph}

Let \(T=(T,m,e)\) be a monad on a category \( \cat V \) with pullbacks. For
the purpose of studying effective descent functors of \(T\)-categories, we
obtain sharper results by describing \( \CatTV \) as a 2~-~equalizer of a
suitable diagram of categories. Before providing such a description, we will
first consider a simpler \(T\)-structure as a guiding example. 

For a pointed endofunctor \( T=(T,e) \) on \( \cat V \), a \textit{reflexive}
\(T\)\textit{-graph} \( x \) consists of
\begin{itemize}[noitemsep,label=--]
  \item
    an object \( x_0 \) of \textit{objects},
  \item
    an object \( x_1 \) of \textit{arrows},
  \item
    a \textit{domain} morphism \( d_1 \colon x_1 \to Tx_0 \), 
  \item
    a \textit{codomain} morphism \( d_0 \colon x_1 \to x_0 \),
  \item
    a \textit{loop} morphism \( s_0 \colon x_0 \to x_1 \),
\end{itemize}
which must satisfy \( d_0 \circ s_0 = \id \) and \( d_1 \circ s_0 = e \). We
note that this data can be organized in the following diagram:
\begin{equation}
  \label{eq:refl.tgraph}
  \begin{tikzcd}
    x_0 \ar[r,"s_0",shift left] \ar[rd,"e",swap]
    & x_1 \ar[l,"d_0",shift left] \ar[d,"d_1"] \\
    & Tx_0
  \end{tikzcd}
\end{equation}
Moreover, a morphism of reflexive \(T\)-graphs \( f \colon x \to y \)
consists of
\begin{itemize}[noitemsep,label=--]
  \item
    an \textit{object} morphism \( f_0 \colon x_0 \to y_0 \),
  \item
    an \textit{arrow} morphism \( f_1 \colon x_1 \to y_1 \),
\end{itemize}
satisfying \( d_0 \circ f_1 = f_0 \circ d_0 \), \( d_1 \circ f_1 = Tf_0 \circ
d_1 \) and \( f_1 \circ s_0 = s_0 \circ f_0 \). These form a category \(
\RGrphTV \), with componentwise composition and identities. We observe that
these are the \textit{T-graphes pointés} of~\cite{Bur71}. 

We take this opportunity to remark that reflexive \(T\)-graphs allow us to
draw conclusions about the descent theory of categorical structures:

\begin{lemma}[{\cite[Lemma A.3]{PL23}}]
  \label{lem:refl.graph}
  Let \( \cat E \) be a class of epimorphisms in \( \RGrphTV \) such that
  \begin{itemize}[label=--,noitemsep]
    \item
      \( \cat E \) contains all split epimorphisms,
    \item
      \( \cat E \) is closed under composition,
    \item
      if \( g \circ f, f \in \cat E \), then \(g \in \cat E\) (right
      cancellation),
  \end{itemize}
  and let \( f \colon x \to y  \) be a morphism of reflexive \(T\)-graphs. If
  \( f_1 \in \cat E \), then \( f_0 \in \cat E \). 
\end{lemma}

\begin{proof}
  In any reflexive \(T\)-graph \(x\), the codomain morphism \( d_0 \colon x_1
  \to x_0 \) is a split epimorphism since \( d_0 \circ s_0 = \id \), so that
  \( d_0 \in \cat E \). If \( f_1 \in \cat E \), then we have \( d_0 \circ f_1
  = f_0 \circ d_0 \in \cat E \) by closure under composition, and \( f_0 \in
  \cat E \) by right cancellation.
\end{proof}

Of particular interest are the classes \( \cat E \) given by the (effective,
almost) descent morphisms, which satisfy each of the properties. Taking \( T =
\id \), we observe that the condition that \( p_0 \) is an effective descent
morphism is redundant in Theorem \ref{thm:int.lec}, and can be omitted.

Returning to our main point, we observe that the category \( \RGrphTV \) can
be described by a (2-)equalizer of diagram categories as well: we consider the
graph
\begin{equation*}
  \cat G= \begin{tikzcd}
            x_0 \ar[r,"s_0",shift left] \ar[rd,"e_0",swap]
            & x_1 \ar[l,"d_0",shift left] \ar[d,"d_1"] \\
            & x'_0
          \end{tikzcd}
\end{equation*}
with relations \( d_1 \circ s_0 = e_0 \) and \( d_0 \circ s_0 = \id \), and we
consider the diagram category \( [\cat G, \cat V] \), together with
functors
\begin{equation*}
  \begin{tikzcd}
    \cat V & {[\cat G,\cat V]} \ar[l,"x_0^*",swap]
                               \ar[r,"e_0^*"]
           & {[2,\cat V]}
  \end{tikzcd}
\end{equation*}
induced by the inclusions \( x_0 \to \cat G \) and \( (x_0 \xrightarrow{e_0}
y_0) \to \cat G \), where \( 2 = (\cdot \to \cdot) \).

We recall that any natural transformation \( \phi \colon F \to G \) of
functors \( \cat C \to \cat D \) is precisely determined by a functor \(
\phi^\sharp \colon \cat C \to [2,\cat D] \), which satisfies \( \ev_0 \circ
\phi^\sharp = F \) and \( \ev_1 \circ \phi^\sharp = G \), where \( \ev_j
\colon [\cat J, \cat D] \to \cat D \) is the evaluation functor. 

For example, if we take the point \( e \colon \id \to T \) of the endofunctor
\( T \), \( e^\sharp \colon \cat V \to [2, \cat V]\) is a functor which
satisfies \( \ev_0 \circ e^\sharp = \id \) and \( \ev_1 \circ e^\sharp = T \).
With this notation, we obtain the following statement:

\begin{lemma}
  We have a 2-equalizer diagram
  \begin{equation}
    \begin{tikzcd}
      \RGrphTV \ar[r]
      & {[\cat G, \cat V]} \ar[r,"e^*_0",shift left]
                           \ar[r,"e^\sharp \circ x_0^*",swap,shift right]
      & {[2, \cat V]}
    \end{tikzcd}
  \end{equation}
\end{lemma}

\begin{proof}
  From the condition imposed by the 2-equalizer in \( \CAT \) we obtain the
  full subcategory of \( [\cat G, \cat V] \) whose diagrams are of the form
  \eqref{eq:refl.tgraph}.
\end{proof}

\section{Internal $T$-categories}
\label{sect:cat.tv}

Let \( T=(T,m,e) \) be a monad on a category \( \cat V \) with finite limits.
Recall that \(T\) is said to be \textit{cartesian} if \(T\) preserves
pullbacks and
\begin{equation*}
  \begin{tikzcd}
    x \ar[d,"f",swap] \ar[r,"e_x"]
      \ar[rd,"\ulcorner",phantom,very near start]
      & Tx \ar[d,"Tf"] \\
    y \ar[r,"e_y",swap] & Ty
  \end{tikzcd}
  \qquad
  \begin{tikzcd}
    TTx \ar[r,"m_x"] \ar[d,"TTf",swap]
        \ar[rd,"\ulcorner",phantom,very near start]
      & Tx \ar[d,"Tf"] \\
    TTy \ar[r,"m_y",swap] & Ty
  \end{tikzcd}
\end{equation*}
are pullback squares for all \(f \colon x \to y \). 

The category \( \CatTV \) of \textit{T-categories} was defined
diagrammatically in \cite[{}I.1]{Bur71} for general monads \(T\), and this is
the definition we will use throughout this chapter. However, Burroni also
observed that the category of \(T\)-categories can equivalently be defined as
the category of monads for the (proarrow) equipment \( \Span_T(\cat V) \) of
\(T\)\textit{-spans}, for \(T\) a cartesian monad; indeed, this is precisely
how \(T\)-categories were defined in \cite{Her00}. 

Here, we shall verify that \( \CatTV \) can be given via a 2-equalizer
involving the category of \( \cat V \)-models for a finite limit sketch \(
\cat S \). Its underlying graph is given by
\begin{equation}
  \label{almost.cosimp}
  \begin{tikzcd}[row sep=large, column sep=large]
    x_0 \ar[r,shift left,"s_0"]
        \ar[rd,"e_0",swap]
    & x_1 \ar[r,shift left=0.6em,"s_0" description,near start]
          \ar[r,shift left=1.2em,"s_1",near start]
          \ar[l,"d_0",shift left]
          \ar[d,"d_1",swap]
          \ar[rd,"e_1" description] 
    & x_2 \ar[l,"d_0" description,swap,near start]
          \ar[l,shift left=0.6em,"d_1",near start]
          \ar[d,"d_2",swap]
    & x_3 \ar[l,shift left=0.6em,"d_2"]
          \ar[l,"d_1" description]
          \ar[l,shift right=0.6em,"d_0",swap] 
          \ar[d,"d_3"] \\
    & x'_0 \ar[r,"s'_0",shift left]
    & x'_1 \ar[d,"d'_1",swap]
           \ar[l,"d'_0",shift left]
    & x'_2 \ar[l,shift right,"d'_0",swap]
           \ar[l,shift left,"d'_1"] 
           \ar[d,"d'_2"] \\
    && x''_0  \ar[lu,"m_0"] 
              \ar[r,"s''_0",shift left]
    & x''_1 \ar[lu,"m_1" description] \ar[l,"d_0''",shift left]
  \end{tikzcd}
\end{equation}
with the following relations\footnote{We point out the resemblance of these
relations with the cosimplicial identities.}
\begin{itemize}[noitemsep,label=--]
  \item
    \( s_1 \circ s_0 = s_0 \circ s_0 \colon x_0 \to x_2 \),
  \item
    \( d_{1+i} \circ s_i = e_i \colon x_i \to x'_i \), 
  \item
    \( d_i \circ s_j = \id \colon x_i \to x_i \), 
  \item
    \( d_2 \circ s_0 = s_0' \circ d_1 \colon x_1 \to x'_1 \),
  \item
    \( d_0 \circ s_1 = s_0 \circ d_0 \colon x_1 \to x_1 \),
  \item
    \( d'_0 \circ s'_0 = \id \colon x'_0 \to x'_0 \),
  \item
    \( d_{1+i} \circ d_{1+i} = m_i \circ d'_{1+i} \circ d_{2+i} \colon x_{2+i}
    \to x'_i\), 
  \item
    \( d_{1+i} \circ d_0 = d_0' \circ d_{2+i} \colon x_{2+i} \to x_i \),  
  \item
    \( d'_j \circ d_{2+i} = d_{1+i} \circ d_j \colon x_{2+i} \to x'_i \), 
  \item
    \( d_0 \circ d_1 = d_0 \circ d_0 \colon x_2 \to x_0 \),
  \item
    \( d_j \circ d_{1+i} = d_i \circ d_j \colon x_3 \to x_1 \),
  \item
    \( d'_1 \circ d_0' = d_0'' \circ d_2' \colon x_2' \to x_0'' \),
  \item
    \( d'_0 \circ d'_1 = d'_0 \circ d'_0 \colon x'_2 \to x_0' \),
\end{itemize}
and limit cones
\begin{equation}
  \label{eq:pb.squares}
  \begin{tikzcd}
    x_2 \ar[r,"d_0"] \ar[d,"d_2",swap] 
        \ar[rd,"\ulcorner",phantom,very near start]
      & x_1 \ar[d,"d_1"] \\
    x'_1 \ar[r,"d'_0",swap] & x'_0
  \end{tikzcd}
  \qquad
  \begin{tikzcd}
    x_3 \ar[r,"d_0"] \ar[d,"d_3",swap] 
        \ar[rd,"\ulcorner",phantom,very near start]
      & x_2 \ar[d,"d_2"] \\
    x'_2 \ar[r,"d'_0",swap] & x'_1
  \end{tikzcd}
  \qquad
  \begin{tikzcd}
    x_2' \ar[r,"d'_0"] \ar[d,"d'_2",swap] 
         \ar[rd,"\ulcorner",phantom,very near start]
      & x_1' \ar[d,"d'_1"] \\
    x''_1 \ar[r,"d''_0",swap] & x''_0
  \end{tikzcd}
\end{equation}
with \( i=0,1 \) and \( j \leq i \). We let \( \ModSV \) be the category of
\(\cat V \)-models for the sketch \( \cat S \). Moreover, abusing notation, we
will also denote by \( \cat S \) the category freely generated by the
underlying graph of \( \cat S \) modulo the given relations.

\begin{remark}[Objects of \(n\)-chains]
  As in Remark \ref{rem:chains}, it is convenient to denote
  \(x_2\) and \( x_3 \) to be the objects of \textit{2-chains} and
  \textit{3-chains} of morphisms respectively, for an internal
  (\(T\)-)category \( x \).
\end{remark}

\begin{lemma}[{\cite[Lemma 3.1]{PL23}}]
  \label{lem:cat.tv.eq}
  If \(T\) preserves pullbacks, then we have a 2-equalizer diagram
  \begin{equation}
    \label{eq:cat.tv.eq}
    \begin{tikzcd}
      \CatTV \ar[r]
      & \ModSV \ar[r,shift left,"d^1"] 
               \ar[r,shift right,"d^0",swap]
      & {[\cat S_T,\cat V]} \times {[2, \cat V]}
    \end{tikzcd}
  \end{equation}
  where the functors \( d^1 \), \( d^0 \) are respectively given by
  \begin{equation*}
    \begin{tikzcd}[column sep=large]
      \ModSV \ar[r] 
      & {[\cat S, \cat V]} 
          \ar[r,"{(I^*_T,\,I^*_{d'_1})}"]
      & {[\cat S_T, \cat V]} \times {[2, \cat V]} \\
      \ModSV \ar[r] 
      & {[\cat S, \cat V]} 
          \ar[r,"{(I^*_{s_0,d_0},I^*_{d_1})}"]
      & {[\cat S_{s_0,d_0}, \cat V]} \times {[2, \cat V]} 
          \ar[r,"{(T,m,e)_! \times T_!}"]
      & {[\cat S_T, \cat V]} \times {[2, \cat V]} 
    \end{tikzcd}
  \end{equation*}
  and \( I_T \colon \cat S_T \to \cat S \), \( I_{d'_1} \colon 2 \to \cat S \), 
  \( I_{s_0,d_0} \colon \cat S_{s_0,d_0} \to \cat S \) and \( I_{d_1} \colon 2
  \to \cat S \) are the subcategories of \( \cat S \) respectively determined
  by the subgraphs
  \begin{equation*}
    \begin{tikzcd}
      x_0 \ar[d,shift left,"s_0"]
           \ar[r,"e_0"]
      & x'_0 \ar[d,shift left,"s'_0"]
      & x''_0 \ar[d,shift left,"s''_0"] \ar[l,"m_0",swap] \\
      x_1 \ar[u,"d'_0",shift left]
          \ar[r,"e_1",swap]
      & x'_1 \ar[u,"d'_0",shift left]
      & x''_1 \ar[u,"d''_0",shift left]
              \ar[l,"m_1"]
    \end{tikzcd},
    \qquad
    \begin{tikzcd}
      x'_1 \ar[r,"d'_1"] & x''_0
    \end{tikzcd},
    \qquad
    \begin{tikzcd}
      x_0 \ar[r,shift left,"s_0"]
      & x_1 \ar[l,"d'_0",shift left]
    \end{tikzcd},
    \qquad
    \begin{tikzcd}
      x_1 \ar[r,"d_1"] & x'_0
    \end{tikzcd},
  \end{equation*}
  the functor \( T_! \colon [2, \cat V] \to [2, \cat V] \) is given by the
  direct image of \(T\), and \( (T,m,e)_! \colon [\cat S_{s_0,d_0}, \cat V]
  \to [\cat S_T, \cat V] \) is given by
  \begin{equation*}
    \begin{tikzcd}
      a \ar[d,"f",shift left] \\ b \ar[u,"g",shift left]
    \end{tikzcd}
    \quad \mapsto \quad
    \begin{tikzcd}
      a \ar[d,"f",shift left] 
        \ar[r,"e_a"]
      & Ta \ar[d,"Tf",shift left] 
      & TTa \ar[d,"TTf",shift left] 
            \ar[l,"m_a",swap] \\ 
      b \ar[u,"g",shift left]
        \ar[r,"e_b",swap]
      & Tb \ar[u,"Tg",shift left]
      & TTb \ar[u,"TTg",shift left] \ar[l,"m_b"]
    \end{tikzcd}.
  \end{equation*}

  Moreover, if \(T\) is cartesian, then \( \CatTV \) has pullbacks, and the
  inclusion \( \CatTV \to \Mod(\cat S, \cat V) \) preserves them.
\end{lemma}

\begin{proof}
  The objects of the 2-equalizer are precisely those diagrams of the form
  \begin{equation}
    \label{eq:cat.tv.diag}
    \begin{tikzcd}[row sep=large, column sep=large]
      x_0 \ar[r,shift left,"s_0"]
          \ar[rd,"e_{x_0}",swap]
      & x_1 \ar[r,shift left=0.6em,"s_0" description,near start]
            \ar[r,shift left=1.2em,"s_1",near start]
            \ar[l,"d_0",shift left]
            \ar[d,"d_1",swap]
            \ar[rd,"e_{x_1}" description] 
      & x_2 \ar[l,"d_0" description,swap,near start]
            \ar[l,shift left=0.6em,"d_1",near start]
            \ar[d,"d_2",swap]
      & x_3 \ar[l,shift left=0.6em,"d_2"]
            \ar[l,"d_1" description]
            \ar[l,shift right=0.6em,"d_0",swap] 
            \ar[d,"d_3"] \\
      & Tx_0 \ar[r,"Ts_0",shift left]
      & Tx_1 \ar[d,"Td_1",swap]
             \ar[l,"Td_0",shift left]
      & x'_2 \ar[l,shift right,"d'_0",swap]
             \ar[l,shift left,"d'_1"] 
             \ar[d,"d'_2"] \\
      && TTx_0  \ar[lu,"m_{x_0}"]
      & TTx_1 \ar[lu,"m_{x_1}" description] \ar[l,"TTd_0"]
    \end{tikzcd}
  \end{equation}
  satisfying the relations imposed by \( \cat S \), such that the following
  squares
  \begin{equation}
    \label{eq:pb.sqs.2}
    \begin{tikzcd}
      x_2 \ar[r,"d_0"] \ar[d,"d_2",swap] 
          \ar[rd,"\ulcorner",phantom,very near start]
        & x_1 \ar[d,"d_1"] \\
      Tx_1 \ar[r,"Td_0",swap] & Tx_0
    \end{tikzcd}
    \qquad
    \begin{tikzcd}
      x_3 \ar[r,"d_0"] \ar[d,"d_3",swap] 
          \ar[rd,"\ulcorner",phantom,very near start]
        & x_2 \ar[d,"d_2"] \\
      Tx_2 \ar[r,"Td_0",swap] & Tx_1
    \end{tikzcd}
    \qquad
    \begin{tikzcd}
      x'_2 \ar[r,"d'_0"] \ar[d,"d'_2",swap] 
           \ar[rd,"\ulcorner",phantom,very near start]
      & Tx_1 \ar[d,"Td_1"] \\
      TTx_1 \ar[r,"TTd_0",swap] & TTx_0
    \end{tikzcd}
  \end{equation}
  are pullback diagrams; compare \eqref{eq:cat.tv.diag} with \cite[Figure
  1]{Bur71}. Since \(T\) preserves pullbacks, the rightmost pullback diagram
  in \eqref{eq:pb.sqs.2} can be replaced by
  \begin{equation*}
    \begin{tikzcd}
      Tx_2 \ar[r,"Td_0"] \ar[d,"Td_2",swap] 
           \ar[rd,"\ulcorner",phantom,very near start]
      & Tx_1 \ar[d,"Td_1"] \\
      TTx_1 \ar[r,"TTd_0",swap] & TTx_0
    \end{tikzcd}
  \end{equation*}

  When \(T\) is cartesian, both \( d^1 \) and \( d^0 \) preserve pullbacks,
  hence \eqref{eq:cat.tv.eq} can be seen as a 2-equalizer in the category of
  categories with pullbacks and pullback-preserving functors.
\end{proof}

\section{Effective descent morphisms via bilimits}
\label{sect:bilim}

Let \( T = (T,m,e) \) be a cartesian monad on a category \( \cat V \) with
pullbacks. We begin by observing that:

\begin{lemma}[{\cite[Lemma A.1]{PL23}}]
  \label{lem:pb.class}
  Let \( \cat P \) be a pullback-stable class of morphisms of \( \cat V \).
  \(T\) creates such morphisms in its essential image.
\end{lemma}
\begin{proof}
  Let \( f \) be a morphism such that \( Tf \in \cat P \).  Since the
  naturality squares for \(m\) and \(e\) at \( f \) are pullbacks, we conclude
  that \( TTf,\, f \in \cat P \).
\end{proof}

We now consider the sketch \( \cat S \) constructed in Section
\ref{sect:cat.tv}.

\begin{lemma}[{\cite[Proposition 4.2]{PL23}}]
  \label{lem:sketch.ess.alg}
  Let \( p \colon x \to y \) be a morphism in \( \ModSV \). If
  \begin{itemize}[noitemsep,label=--]
    \item
      \( p_0, p'_0, p''_0, p_1, p'_1, p''_1 \) are effective descent
      morphisms,
    \item
      \( p_2, p'_2 \) are descent morphisms, and
    \item
      \( p_3 \) is an almost descent morphism,
  \end{itemize}
  then \(p\) is an effective descent morphism.
\end{lemma}
\begin{proof}
  The proof rests in describing \( \ModSV \) as the category of \( \cat V
  \)-models for a suitable essentially algebraic theory; then the result is an
  immediate consequence of \cite[Proposition 3.2.4]{Cre99}.

  Indeed, we obtain a sketch \( \overline{\cat S} \), Morita equivalent to \(
  \cat S \), from the essentially algebraic theory \( \cat A \) defined by
  \begin{itemize}[label=--,noitemsep]
    \item
      sorts \( x_0,\, x_1,\, x'_0,\, x'_1,\, x''_0,\, x''_1 \),
    \item
      total operations given by the arrows of the full subgraph of the
      underlying graph of \( \cat S \) consisting of the aforementioned sorts,
    \item
      partial operations given by \( d_1 \colon x_1 \times x'_1 \to x_1 \), \(
      d'_1 \colon x'_1 \times x''_1 \to x'_1 \), and the limit cones of \(
      \cat S \), with \( x_2 \) and \( x'_2 \) respectively replaced by \( x_1
      \times x'_1 \) and \( x'_1 \times x''_1 \), give the equations for these
      partial operations,
    \item 
      the remaining equations are given by the underlying relations of \( \cat
      S \), with \( x_2 \), \( x_3 \) and \( x'_2 \) replaced by \( x_1 \times
      x'_1 \), \( x_1 \times x'_1 \times x''_1 \) and \( x'_1 \times x''_1 \),
      respectively.
  \end{itemize}
  The sketch \( \overline{\cat S} \) constructed from \( \cat A \) via the
  procedure described in \cite[{}3.2.1]{Cre99} contains \( \cat S \) as a
  ``subsketch'', containing extra limit cones for the formal products \( x_1
  \times x'_1 \), \( x'_1 \times x''_1 \) and \( x_1 \times x'_1 \times x''_1
  \), which are used to construct the limit cones of \( x_2,\, x'_2 \) and \(
  x_3 \), as well as the (derived) partial operations and equations.

\end{proof}

Denoting by \( \PsEq(F,G) \) the \textit{pseudoequalizer} of a pair of
functors \( F,G \colon \cat A \to \cat B \), we obtain:

\begin{lemma}[{\cite[Lemma 3.3]{PL23}}]
  The induced inclusion \( \CatTV \to \PsEq(d^1,d^0) \) is full and preserves
  pullbacks.
\end{lemma}

\begin{proof}
  We recall that a morphism \( f \colon (x,\zeta) \to (y,\xi) \) of \(
  \PsEq(d^1, d^0) \) is a morphism \( f \colon x \to y \) in \( \ModSV \)
  satisfying \( \xi \circ d^0 f = d^1f \circ \zeta \).  Thus, when \( \zeta \)
  and \( \xi \) are identities, we precisely obtain a \(T\)-category functor.

  Since \( d^1 \) and \( d^0 \) are pullback-preserving functors
  between categories with pullbacks, it follows that \( \PsEq(d^1,d^0) \) has
  pullbacks, and \( \PsEq(d^1,d^0) \to \ModSV \) creates them. We also note
  that \( \CatTV \to \ModSV \) preserves pullbacks as well, concluding the
  proof.
\end{proof}

\begin{lemma}[{\cite[Theorem 3.5]{PL23}}]
  \label{lem:coh.pseq}
  If a morphism \( p \colon (x,\id) \to (y,\theta) \) in \( \PsEq(d^1, d^0) \)
  is a componentwise epimorphism, then \( (y, \theta) \iso (z, \id) \) for a
  \(T\)-category \(z\).
\end{lemma}

\begin{proof}
  By hypothesis, we have \( d^0 p = \theta \circ d^1 p  \). Writing \(
  \theta = (\theta_T, \theta_d) \), we obtain equations
  \begin{align*}
    p_i &= \theta_{T,i} \circ p_i
    & Tp_1 &= \theta_{d,1} \circ p'_1 \\
    Tp_i &= \theta'_{T,i} \circ p'_i
    & Tp'_0 &= \theta_{d,0} \circ p''_0 \\
    TTp_i &= \theta''_{T,i} \circ p''_i,
  \end{align*}
  for \( i = 0,1 \) from which we deduce
  \begin{equation*}
    \theta_{T,i} = \id,
    \qquad \theta''_{T,0} = T\theta'_{T,0} \circ \theta_{d,0},
    \qquad \theta'_{T,1} = \theta_{d,1},
  \end{equation*}
  since \( p \) is a componentwise epimorphism.

  We claim \( (y,\theta) \) is isomorphic to a \(T\)-category. The
  construction of the \(T\)-category presented below is similar to \cite[Lemma
  3.4]{PL23}. This \(T\)-category has underlying reflexive \(T\)-graph 
  \begin{equation*}
    \begin{tikzcd}
      Ty_0 & \ar[l,"\theta'_{T,0}",swap] y'_0 
           & y_1 \ar[l,"d_1",swap] \ar[r,"d_0",shift left] 
           & y_0, \ar[l,"s_0",shift left]
    \end{tikzcd}
  \end{equation*}
  and we observe that
  \begin{equation*}
    \begin{tikzcd}
      y_2 \ar[r,"d_2"] \ar[d,"d'_0",swap]
        & y'_1 \ar[r,"\theta'_{T,1}"] \ar[d,"d'_0" description]
        & Ty_1 \ar[d,"Td_0"] \\
      y_1 \ar[r,"d_1",swap]
        & y'_0 \ar[r,"\theta'_{T,0}",swap]
        & Ty_0
    \end{tikzcd}
    \qquad
    \begin{tikzcd}
      y_3 \ar[r,"d_3"] \ar[d,"d'_0",swap]
        & y'_2 \ar[r,equal] \ar[d,"d'_0" description]
        & y'_2 \ar[d,"\theta'_{T,1} \circ d'_0"] \\
      y_2 \ar[r,"d_2",swap]
        & y'_1 \ar[r,"\theta'_{T,1}",swap]
        & Ty_1
    \end{tikzcd}
  \end{equation*}
  are pullback diagrams; the left squares are a pullback, and the right
  squares commute, and their parallel sides are isomorphisms, as \( \theta \)
  is an isomorphism.

  Likewise,
  \begin{equation*}
    \begin{tikzcd}
      y'_2 \ar[r,"d''_2"] \ar[d,"d'_0",swap]
        & y''_1 \ar[r,"\theta''_{T,1}"]
                \ar[d,"d'_0" description]
        & TTy_1 \ar[d,"TTd_0"] \\
      y'_1 \ar[r,"d'_1" description]
           \ar[d,"\theta_{d,1}",swap]
        & y''_0 \ar[r,"\theta''_{T,0}" description]
                \ar[d,"\theta_{d,0}" description]
        & TTy_0 \ar[d,equal] \\
      Ty_1 \ar[r,"Td'_1",swap]
        & Ty'_0 \ar[r,"T\theta'_{T,0}",swap]
        & TTy_0
    \end{tikzcd}
  \end{equation*}
  is a pullback diagram, since the top left square is a pullback, and the
  remaining squares, whose parallel sides are isomorphisms, are commutative,
  and therefore pullback diagrams.

  Now, we are left with verifying that the relations hold. It is enough to
  verify relations of morphisms with (co)domain \( Ty_0, Ty_1, TTy_0, TTy_1
  \), as the remaining hold by definition. We have
  \begin{equation*}
    (\theta'_{T,i} \circ d_{1+i}) \circ s_i 
      = \theta'_{T,i} \circ e_i
      = e_{x_i} \circ \theta_{T,i} = e_{x_i},
  \end{equation*}
  \begin{equation*}
    (\theta'_{T,1} \circ d_2) \circ s_0
      = \theta'_{T,1} \circ s'_0 \circ d_1
      = Ts_0 \circ (\theta'_{T,0} \circ d_1),
  \end{equation*}
  \begin{equation*}
    Td_0 \circ Ts_0 = T(d_0 \circ s_0) = \id,
  \end{equation*}
  \begin{align*}
    \theta'_{T,0} \circ d_1 \circ d_1
      &= \theta'_{T,0} \circ m_0 \circ d'_1 \circ d_2 \\
      &= m_{x_0} \circ \theta''_{T,0} \circ d'_1 \circ d_2 \\
      &= m_{x_0} \circ T\theta'_{T,0} \circ \theta_{d,0}
                 \circ d'_1 \circ d_2 \\
      &= m_{x_0} \circ T\theta'_{T,0} \circ Td_1
                 \circ \theta_{d,1} \circ d_2 \\
      &= m_{x_0} \circ T(\theta'_{T,0} \circ d_1)
                 \circ (\theta'_{T,1} \circ d_2),
  \end{align*}
  \begin{equation*}
    \theta'_{T,1} \circ d_2 \circ d_2
      = \theta'_{T,1} \circ m_1 \circ d'_2 \circ d_3
      = m_{x_1} \circ (\theta''_{T,1} \circ d'_2) \circ d_3,
  \end{equation*}
  \begin{equation*}
    \theta'_{T,i} \circ d_{1+i} \circ d_j
      = \theta'_{T,i} \circ d'_j \circ d_{2+i}
      = \begin{cases}
          (\theta'_{T,1} \circ d'_1) \circ d_3 & i=j=1 \\
          Td_0 \circ (\theta'_{T,i} \circ d_{2+i}) & j=0
        \end{cases}
    \qquad i,j=0,1,\quad j \leq i
  \end{equation*}
  \begin{equation*}
    T(\theta'_{T,0} \circ d_1) \circ \theta'_{T,1} \circ d'_0 
      = T\theta'_{T,0} \circ \theta_{d,0} \circ d'_1 \circ d'_0
      = \theta''_{T,0} \circ d''_0 \circ d'_2
      = TTd_0 \circ (\theta''_{T,1} \circ d'_2),
  \end{equation*}
  \begin{equation*}
    Td_0 \circ \theta'_{T,1} \circ d'_1
      = \theta'_{T,0} \circ d'_0 \circ d'_1
      = \theta'_{T,0} \circ d'_0 \circ d'_0
      = Td_0 \circ \theta'_{T,1} \circ d'_0,
  \end{equation*}
  and this concludes the proof.
\end{proof}

\begin{corollary}[{\cite[Lemma 4.4]{PL23}}]
  \label{cor:refl.eff.desc.cattv}
  The embedding \( \CatTV \to \PsEq(d^1, d^0) \) reflects effective descent
  morphisms.
\end{corollary}

\begin{proof}
  Since any effective descent morphism is an epimorphism, this result is an
  immediate consequence of Lemma \ref{lem:coh.pseq} and Corollary
  \ref{cor:eff.desc.iso}.
\end{proof}

\begin{theorem}[{\cite[Theorem 4.5]{PL23}}]
  \label{thm:desc.cattv}
  A functor \( p \colon x \to y \) of \(T\)-categories is an effective descent
  morphism in \( \CatTV \), provided that
  \begin{itemize}[noitemsep,label=--]
    \item
      \( Tp_1 \) is an effective descent morphism,
    \item
      \( Tp_2 \) is a descent morphism, and
    \item
      \( p_3 \) is an almost descent morphism.
  \end{itemize}
\end{theorem}

\begin{proof}
  By Lemma \ref{lem:pb.class}, if \( Tp_2 \) is a descent morphism, then so is
  \( p_2 \), and if \( Tp_1 \) is an effective descent morphism, then so are
  \( p_1 \) and \( TTp_1 \). Moreover, by Lemma \ref{lem:refl.graph}, we may
  also deduce that \( p_0, Tp_0 \) and \( TTp_0 \) are effective descent
  morphisms.

  These conditions, and the fact that \( p_3 \) is an almost descent morphism,
  guarantee that \(p\) is an effective descent morphism in \( \ModSV \), and
  the morphism \( d^1 p = d^0 p \) is a descent morphism, as these are
  determined componentwise. 

  Thus, we conclude that \( (p,\id) \) is an effective descent morphism in \(
  \PsEq(d^1,d^0) \) by Proposition~\ref{prop:pseq.descent}, and, by Corollary
  \ref{cor:refl.eff.desc.cattv}, so is \( p \).
\end{proof}

\section{A direct description of effective descent morphisms}
\label{sect:direct}

We return to the setting of Burroni's \textit{$ T $-catégories}, where \( \cat
V \) is any category with finite limits, and \(T\) is a monad on \( \cat V \),
not necessarily cartesian. We confirm that the arguments of Le Creurer for
effective descent morphisms of essentially algebraic theories can be applied
just as well to \textit{$T$-catégories}.

Throughout, we assume that \( p \colon x \to y \) is a functor of \( T
\)-categories.  We recall that since \(p\) is a reflexive \(T\)-graph
morphism, if \( \cat E \) is the class of (effective/almost) descent morphisms
in \( \cat V \) and \( p_1 \in \cat E \), then \( p_0 \in \cat E \) as well by
Lemma \ref{lem:refl.graph}.

\begin{lemma}[{\cite[Lemma 5.1]{PL23}}]
  \label{lem:int.al.desc}
  If \( p_1 \) is a (pullback-stable) epimorphism in \( \cat V \), then so is
  \( p \) in \( \CatTV \).
\end{lemma}

\begin{proof}
  If \(q, r \colon y \to z \) are functors such that \( q \circ p = r \circ p
  \), then \( q_i \circ p_i = r_i \circ p_i \) for \( i = 0,1 \), so \( q_i =
  r_i \), implying \( q = r \). We conclude that \(p\) is an epimorphism. 

  Since pullbacks in \( \CatTV \) are calculated componentwise, our claim
  is verified.
\end{proof}

\begin{lemma}[{\cite[Lemma 5.2]{PL23}}]
  \label{lem:int.desc}
  If \( p_1 \) is a (pullback-stable) regular epimorphism, and \( p_2 \) is a
  (pullback-stable) epimorphism in \( \cat V \), then \( p \) is a
  (pullback-stable) regular epimorphism in \( \CatTV \).
\end{lemma}

\begin{proof}
  We consider the kernel pair of \(p\):
  \begin{equation}
    \label{eq:ker.p}
    \begin{tikzcd}
      k \ar[r,"r"] \ar[d,"s",swap] 
        \ar[rd,"\ulcorner",phantom,very near start]
      & x \ar[d,"p"] \\
      x \ar[r,"p",swap] & y
    \end{tikzcd}
  \end{equation}
  If \( q \colon x \to z \) is a functor such that \( r \circ q = s \circ q
  \), then, when \( p_i \) is a regular epimorphism for \( i=0,1 \), there
  exists a unique \( t_i \colon y_i \to z_i \) making the triangle of Diagram
  \eqref{eq:coeq.p1} commute
  \begin{equation}
    \label{eq:coeq.p1}
    \begin{tikzcd}
      k_i \ar[r,"r_i",shift left]
          \ar[r,"s_1",shift right,swap]
      & x_i \ar[r,"p_i"] \ar[rd,"q_i",swap]
      & y_i \ar[d,"t_i",dashed] \\
      && z_i
    \end{tikzcd}
  \end{equation}
  for \( i = 0, 1 \). The morphisms \( t_0, t_1 \) define a functor \( t
  \colon y \to z \) of \(T\)-categories; indeed, we note that, by the
  universal property, \( t_2 \circ (g,f) = (t_1 \circ g, Tt_1 \circ f) \),
  from which we deduce \( q_2 = t_2 \circ p_2 \). Then, the following
  calculations
  \begin{align*}
    t_1 \circ d_1 \circ p_2 &= t_1 \circ p_1 \circ d_1 
                             = q_1 \circ d_1 
                             = d_1 \circ q_2 
                             = d_1 \circ t_2 \circ p_2, \\
    d_i \circ t_1 \circ p_1 &= d_i \circ q_1 
                             = T^iq_0 \circ d_i 
                             = T^it_0 \circ T^ip_0 \circ d_i 
                             = T^it_0 \circ d_i \circ p_1
  \end{align*}
  plus the fact that \( p_1, p_2 \) are epimorphisms confirm our claim.
  Naturally, \( t \colon y \to z \) is the unique functor making the triangle
  below commute
  \begin{equation*}
    \begin{tikzcd}
      k \ar[r,"r",shift left]
          \ar[r,"s",shift right,swap]
      & x \ar[r,"p"] \ar[rd,"q",swap]
      & y \ar[d,"t",dashed] \\
      && z,
    \end{tikzcd}
  \end{equation*}
  for if \( l \colon y \to z \) were another such functor, we would deduce
  that \( l_i = t_i \) by uniqueness given at \eqref{eq:coeq.p1} for \( i = 0,
  1 \), so \( l = t \).

  Again, by componentwise calculation of pullbacks of \( \CatTV \), we
  conclude that if \( p_1 \) and \( p_2 \) are pullback-stable, then so is
  \(p\).
\end{proof}

\begin{theorem}[{\cite[Theorem 5.3]{PL23}}]
  \label{thm:int.eff.desc}
  If \(T\) preserves pullbacks, and
  \begin{itemize}[noitemsep,label=--]
    \item
      \( p_1 \colon x_1 \to y_1 \) is an effective descent morphism,
    \item
      \( p_2 \colon x_2 \to y_2 \) is a descent morphism, and
    \item
      \( p_3 \colon x_3 \to y_3 \) is an almost descent morphism
  \end{itemize}
  in \( \cat V \), then \(p\) is an effective descent morphism in \( \CatTV
  \).
\end{theorem}

\begin{proof}
  By Lemma \ref{lem:int.desc}, our hypotheses guarantee that \( \mathcal K^p
  \colon \CatTV \comma y \to \Desc(p) \) is fully faithful. Hence, our goal is
  to confirm \( \mathcal K^p \) is essentially surjective, and we shall do so
  via Proposition \ref{prop:comp.ess.surj}. 

  We consider the kernel pair \eqref{eq:ker.p} of \(p\), and we let \( (a
  \colon v \to x, \gamma \colon r \times_x a \to v) \) be a discrete fibration
  (internal to \( \CatTV \)) over \(\Ker(p)\). 

  If \( p_i \) is an effective descent morphism for \( i=0,1\), we obtain an
  equivalence \( \mathcal K^{p_i} \colon \cat V \comma y_i \to \Desc(p_i) \)
  for \( i = 0, 1 \). Thus, by Proposition \ref{prop:comp.ess.surj} there
  exist \( b_i \colon w_i \to y_i \) and \( h_i \colon v_i \to w_i \) such
  that
  \begin{equation}
    \label{eq:desc.pb}
    \begin{tikzcd}
      v_i \ar[d,"a_i",swap] \ar[r,"h_i"]
          \ar[rd,"\ulcorner",phantom,very near start]
        & w_i \ar[d,"b_i"] \\
      x_i \ar[r,"p_i",swap] & y_i
    \end{tikzcd}
  \end{equation}
  is a pullback diagram, satisfying
  \begin{equation}
    \label{eq:desc.cond}
    h_i \circ \gamma_i = h_i \circ \epsilon_{p_i \circ a_i}
  \end{equation}
  for \( i = 0, 1 \).

  We claim that \(w\) is a \(T\)-category, \( h_0, h_1 \) define a functor \(
  h \colon v \to w \), and \( b_0, b_1 \) define a functor \( b \colon w \to y
  \). To do so, we consider the kernel pairs of \( h_0 \) and \( h_1 \); given
  that \( p_0 \) and \( p_1 \) are descent morphisms, \( h_0 \) and \( h_1 \)
  are the coequalizers of their kernel pairs. Moreover, since \(T\) preserves
  kernel pairs, we obtain
  \begin{equation*}
    \begin{tikzcd}
      u_1 \ar[r,shift left] \ar[r,shift right]
          \ar[d,"d_i",swap]
        & v_1 \ar[r,"h_1"] \ar[d,"d_i"] & w_i \ar[d,"d_i",dashed] \\
      T^iu_0 \ar[r,shift left] \ar[r,shift right]
        & T^iv_0 \ar[r,"h_0",swap] & T^iw_0 
    \end{tikzcd}
  \end{equation*}
  which provides the \(T\)-graph structure of \(w\). With this, we obtain the
  following cospan of cospans
  \begin{equation}
    \label{eq:sp.of.sps}
    \begin{tikzcd}
      Tw_1 \ar[r,"Td_0"] \ar[d,"Tb_1",swap] & Tw_0 \ar[d,"b_0" description]
        & w_1 \ar[l,"d_1",swap] \ar[d,"b_1"] \\
      Ty_1 \ar[r,"Td_0"] & Ty_0
        & y_1 \ar[l,"d_1",swap] \\
      Tx_1 \ar[r,"Td_0"] \ar[u,"Tp_1"] & Tx_0 \ar[u,"p_0" description]
        & x_1 \ar[l,"d_1",swap] \ar[u,"p_1",swap]
    \end{tikzcd}
  \end{equation}
  and since \(T\) preserves pullbacks, the horizontal and vertical pullbacks
  of \eqref{eq:sp.of.sps} are, respectively
  \begin{equation}
    \label{eq:pb.of.sps}
    \begin{tikzcd}
      w_2 \ar[r,"b_2"] & y_2 & x_2 \ar[l,"p_2",swap] \\
      Tv_1 \ar[r,"Td_0"] & Tv_0 & v_1 \ar[l,"d_1",swap]
    \end{tikzcd}
  \end{equation}
  so, by commutativity of limits, the cospans \eqref{eq:pb.of.sps} have
  isomorphic pullbacks. Since the pullback of the last span defines \( v_2 \),
  we obtain the following pullback diagram:
  \begin{equation}
    \label{eq:pb.pairs}
    \begin{tikzcd}
      v_2 \ar[d,"a_2",swap] \ar[r,"h_2"]
          \ar[rd,"\ulcorner",phantom,very near start]
        & w_2 \ar[d,"b_2"] \\
      x_2 \ar[r,"p_2",swap] & y_2
    \end{tikzcd}
  \end{equation}

  Analogously, we deduce that
  \begin{equation}
    \label{eq:pb.triples}
    \begin{tikzcd}
      v_3 \ar[d,"a_3",swap] \ar[r,"h_3"]
          \ar[rd,"\ulcorner",phantom,very near start]
        & w_3 \ar[d,"b_3"] \\
      x_3 \ar[r,"p_3",swap] & y_3
    \end{tikzcd}
  \end{equation}
  is a pullback diagram as well. 
  
  If \( p_2 \) is also a descent morphism, we conclude via \eqref{eq:pb.pairs}
  that \( h_2 \) is a regular epimorphism as well. So, we may consider its
  kernel pair as well, to obtain
  \begin{equation*}
    \begin{tikzcd}
      u_0 \ar[r,shift left] \ar[r,shift right]
          \ar[d,"s_0",swap]
        & v_0 \ar[r,"h_0"] \ar[d,"s_0"] & w_0 \ar[d,"s_0",dashed] \\
      u_1 \ar[r,shift left] \ar[r,shift right]
        & v_1 \ar[r,"h_1" description] & w_1 \\
      u_2 \ar[r,shift left] \ar[r,shift right]
          \ar[u,"d_1"]
        & v_2 \ar[r,"h_2"] \ar[u,"d_1",swap] & w_2 \ar[u,"d_1",dashed,swap]
    \end{tikzcd}
  \end{equation*}
  which give the identity and composition structure morphisms for \( w \).
  Indeed, under the hypothesis \(w\) is a \(T\)-category, we can already
  conclude that \( h \colon v \to w \) is a functor.

  To prove this hypothesis, note that if \( p_3 \) is also an almost descent
  morphism, then so is \(h_3\) by~\eqref{eq:pb.triples}. Now, we note that we
  have
  \begin{align*}
    d_1 \circ s_0 \circ h_0 &= Th_0 \circ d_1 \circ s_0 
                             = Th_0 \circ e 
                             = e \circ h_0 \\
    d_0 \circ s_0 \circ h_0 &= h_0 \circ d_0 \circ s_0
                             = h_0 \\
    d_1 \circ d_1 \circ h_2 &= Th_0 \circ d_1 \circ d_1   
                             = Th_0 \circ m \circ Td_1 \circ d_2 
                             = m \circ Td_1 \circ d_2 \circ h_2 \\
    d_0 \circ d_1 \circ h_2 &= h_0\circ d_0 \circ d_1 
                             = h_0\circ d_0 \circ d_0 
                             = d_0\circ d_0 \circ h_2 \\
    d_1 \circ s_i \circ h_1 &= h_1\circ d_1 \circ s_i 
                             = h_1\circ s_0 \circ d_0 
                             = s_0\circ d_0 \circ h_1 \\
    d_1 \circ d_2 \circ h_3 &= h_1\circ d_1 \circ d_2 
                             = h_1\circ d_1 \circ d_1 
                             = d_1\circ d_1 \circ h_3
  \end{align*}
  so, by cancellation, we conclude that \( w \) is a \(T\)-category.

  Finally, we must confirm \( b_0, b_1 \) define a functor \( b \colon w \to y
  \). We recall that \( h_0, h_1, h_2 \) are epimorphisms, and we observe that
  \begin{align*}
    b_1 \circ d_1 \circ h_2 &= b_1 \circ h_1 \circ d_1
                             = p_1 \circ a_1 \circ d_1
                             = d_1 \circ p_2 \circ a_2
                             = d_1 \circ b_2 \circ h_2 \\
    d_i \circ b_1 \circ h_1 &= d_i \circ p_1 \circ a_1 
                             = T^ip_0 \circ T^ia_0 \circ d_i
                             = T^ib_0 \circ T^ih_0 \circ d_i
                             = T^ib_0 \circ d_i \circ h_1,
  \end{align*}
  confirming that the properties for a functor are satisfied, via cancellation.

  Since \eqref{eq:desc.pb} is a pullback diagram for \(i=0,1\), we obtain a
  pullback diagram
  \begin{equation*}
    \begin{tikzcd}
      v \ar[d,"a",swap] \ar[r,"h"]
          \ar[rd,"\ulcorner",phantom,very near start]
        & w \ar[d,"b"] \\
      x \ar[r,"p",swap] & y
    \end{tikzcd}
  \end{equation*}
  and we have
  \begin{equation}
    h \circ \gamma = h \circ \epsilon_{p \circ a}
  \end{equation}
  as a consequence of \eqref{eq:desc.cond} for \( i = 0,1 \). This concludes
  our proof, by Proposition \ref{prop:comp.ess.surj}.
\end{proof}

\section{Application to graded, operadic and enhanced multicategories}
\label{sect:application}

Let \( \cat V \) be a category with finite limits, and \(T\) a cartesian monad
on \( \cat V \). We recall that we have an equivalence
\begin{equation}
  \label{eq:cat.tv.slice}
  \CatTV \comma x \eqv \Cat(T_x,\cat V \comma x_0)
\end{equation}
for any internal \( (T,\cat V) \)-category \( x \), where \( T_x \) is a
cartesian monad on \( \cat V \comma x_0 \) induced by \(x\); see \cite[Section
6.2]{Lei04}. Since \( \cat C \comma x \to \cat C \) creates (effective,
almost) descent morphisms, we can obtain results about effective descent
morphisms in \( \Cat(T_x,\cat V \comma x_0) \) by studying those of \(
\Cat(T,\cat V) \).

To illustrate this, let \( \cat V \) be a category with finite limits. For
an internal category \(\cat C\) in \( \cat V \), the monad on \( \cat V \comma
\cat C_0 \) induced by the identity monad on \( \cat V \) is denoted \( \cat
C \times_{\cat C_0} - \). Via the equivalence~\eqref{eq:cat.tv.slice}, the
category of \( \cat C \)-graded categories internal to \( \cat V \) is the
category of internal functors over \( \cat C \):
\begin{equation*}
  \Cat(\cat C \times_{\cat C_0} -,\cat V \comma \cat C_0)
  \eqv \CatV \comma \cat C.
\end{equation*}
Hence, the study of effective descent functors of graded internal categories 
reduces to the study of effective descent functor of internal categories. Of
particular importance is the case \( \cat C_0 \iso 1 \); that is, when \( \cat
C \) is an internal \( \cat V \)-monoid.

Now, let \( \cat V \) be a lextensive category. The free monoid monad \(
\mathfrak M \) on \( \cat V \), given on objects by
\begin{equation*}
  X \mapsto \sum_{n \in \N} X^n,
\end{equation*}
is a cartesian monad, hence we may consider the category of
\textit{multicategories} internal to \( \cat V \), given by \( \MultiCat(\cat
V) = \Cat(\mathfrak M,\cat V) \). By Theorem \ref{thm:int.eff.desc}, a multicategory
functor \( p \colon x \to y \) internal to \( \cat V \) is an effective
descent morphism in \( \MultiCat(\cat V) \), provided that
\begin{itemize}[label=--,noitemsep]
  \item
    \(p_1 \colon x_1 \to y_1 \) is an effective descent morphism in \( \cat V
    \),
  \item
    \(p_2 \colon x_2 \to y_2 \) is a descent morphism in \( \cat V \),
  \item
    \(p_3 \colon x_3 \to y_3 \) is an almost descent morphism in \( \cat V \).
\end{itemize}
If we have an internal \( \cat V \)-operad \( \mathfrak O \) (an internal \(
\cat V \)-multicategory with terminal object-of-objects), the monad \(
T_{\mathfrak O} \) induced by the free monoid monad is given by
\begin{equation*}
  X \mapsto \sum_{n \in \N} \mathfrak O_n \times X^n,
\end{equation*}
and is cartesian as well. We define the objects of \( \Cat(T_{\mathfrak
O},\cat V) \) to be the \textit{operadic} multicategories internal to \( \cat
V \). Via \eqref{eq:cat.tv.slice}, we obtain
\begin{equation*}
  \Cat(T_{\mathfrak O},\cat V) \eqv \MultiCat(\cat V) \comma \mathfrak O,
\end{equation*}
so a morphism of operadic multicategories is effective for descent if it is so
as a morphism on the underlying internal multicategories, which we have
described above.

Finally, we consider the 2-monad \( \Fam_\fin \) on \( \Cat \), the
\textit{free finite coproduct completion} 2-monad. This is one of the central
objects of study in \cite{Web07}, where it was shown to be a cartesian
(2-)monad. Hence, we may consider \( \Fam_\fin \)-categories internal to
\( \Cat \), and we can obtain a description of the effective descent morphisms
of \( \Cat(\Fam_\fin,\Cat) \) via Theorem \ref{thm:int.eff.desc}.

Moreover, we have a cartesian 2-natural transformation \( \mathfrak S \to
\Fam_\fin \) (see \cite[Example 7.5]{Web07}), where \( \mathfrak S \) is the
free symmetric strict monoidal category 2-monad, and the objects of \(
\Cat(\mathfrak S,\Cat) \) were called \textit{enhanced symmetric
multicategories} in \cite[p. 212]{Lei04}, whose study of effective descent
functors reduces to the previous case.

By analogy, we may take the objects of \( \Cat(\Fam_\fin, \Cat) \) to be the
enhanced \textit{cocartesian} multicategories, and analogously, we have the
free finite product completion 2-monad \( \Fam_\fin^* \) on \( \Cat \),
defined on objects by \( \cat A \mapsto \Fam_\fin(\cat A^\op)^\op \), which is
also cartesian, and the objects of \( \Cat(\Fam_\fin^*,\Cat) \) may be called
enhanced \textit{cartesian} multicategories\footnote{These are related to the
\textit{multi-sorted Lawvere theories} -- see \cite{CS10}, or Subsection 5.4.2}.
 
For an ordinary functor \( F \colon \cat C \to \cat D \) of categories, we
have that
\begin{itemize}[label=--,noitemsep]
  \item
    \(F\) is an almost descent functor if \( F \) is surjective on morphisms,
  \item
    \(F\) is a descent functor if \( F \) is surjective on 2-chains,
  \item
    \(F\) is an effective descent functor if \(F\) is surjective on 3-chains,
\end{itemize}
so, for \(T\) a cartesian monad on \( \Cat \) (such as one of \( \Fam_\fin,
\Fam_\fin^*, \mathfrak S \)), a functor \( P \colon \cat X \to \cat Y \) of \(
T \)-categories is effective for descent if
\begin{itemize}[label=--,noitemsep]
  \item
    \(P_1 \colon \cat X_1 \to \cat Y_1 \) is surjective on 3-chains,
  \item
    \(P_2 \colon \cat X_2 \to \cat Y_2 \) is surjective on 2-chains,
  \item
    \(P_3 \colon \cat X_3 \to \cat Y_3 \) is surjective on morphisms.
\end{itemize}

\chapter{Generalized enriched multicategory functors}
\label{chap:enriched-multi}

Equipped with the description of effective descent morphisms in \( \CatTV \)
given in Theorem \ref{thm:int.eff.desc}, our goal is to study effective
descent morphisms in a category of enriched generalized multicategories, by
suitably embedding it into a category of internal generalized multicategories,
and studying whether the effective descent morphisms are reflected by this
embedding. In short, we aim to generalize the approach of \cite[Theorem
9.11]{Luc18} for the embedding \( \VCat \to \CatV \) to the setting of
generalized multicategorical structures.

Such an embedding is constructed via a suitable notion of
\textit{change-of-base} for generalized multicategories. This work was carried
out in \cite{PL23b} from a general point-of-view; here, we recount the details
that are relevant in the study of effective descent functors for enriched
generalized multicategories.

The notion of enriched \( (T, \cat V) \)-categories was introduced in
\cite{CT03}, under the terminology \( (T, \cat V) \)-categories, as a suitable
notion of \textit{lax algebras}. In Section \ref{sect:tvcats}, we will provide
the definition in a slightly more general setting, as done in \cite{Sea05} and
\cite{HST14} (when \( \cat V \) is a suitable quantale), as well
as~\cite{CS10}. We consider a monad \( T = (T,m,e) \) on \( \VMat \) in \(
\Equip_\lax \) -- the 2-category of \textit{equipments}, \textit{lax functors}
and \textit{icons}, and the enriched \( (T,\cat V) \)-categories are defined
as a suitable notion of lax algebra. The original setting of \cite{CT03} is
recovered when \(T\) is a \textit{normal} lax monad.

Working in the 2-category \( \Equip_\lax \), in Section
\ref{sect:base.change}, we review the lifting of the functor \( - \pt 1
\colon \Set \to \cat V \) to a functor \( - \pt 1 \colon \VMat \to \SpanV \)
in Proposition \ref{prop:eqp.adj}, and we describe the monad \( \Tt \) on \(
\VMat \), given by reflecting the monad \(T\) on \(\SpanV\) along \( - \pt 1
\). The results of \cite{PL23b} on change-of-base for generalized categorical
structures then confirm that we obtain an embedding
\begin{equation}
  \label{eq:emb.tvcat.cattv}
  - \pt 1 \colon \TtVCat \to \CatTV,
\end{equation}
under a suitable condition (Theorem \ref{prop:tvcat.sub.cattv}). 

In Section \ref{sect:eff.desc.refl}, we study the problem of reflecting
effective descent morphisms along the embedding \eqref{eq:emb.tvcat.cattv}. We
confirm that, under a second suitable condition, every effective descent
morphism is reflected (Lemma \ref{lem:lem.eff.desc.refl.tvcat}). Then, via our
description of the effective descent morphisms in \( \CatTV \), we obtain
Theorem \ref{thm:desc.tvcat}, the main result of this chapter.

We finish this chapter with Section \ref{sect:scope}, where we provide our
applications in the study of effective descent morphisms in categorical
structures, after briefly discussing whether the two extra conditions
introduced in Sections \ref{sect:base.change} and \ref{sect:eff.desc.refl}
affect the scope of the available examples.

\section{$(T,\cat V)$-categories}
\label{sect:tvcats}

Throughout this chapter, \( \cat V \) is assumed to be a lextensive, cartesian
monoidal category. We recall that \( \cat V \) is necessarily distributive;
see \cite[Proposition 4.5]{CLW93}.

The equipment of \( \cat V \)-matrices, denoted \( \VMat \), is defined in
\cite{BCSW83}; we simply recall that a \( \cat V \)-matrix \( r \) is a family
\( (r(x,y))_{x,y \in X \times Y} \) of objects in \( \cat V \) indexed by \( X
\times Y \). Such a \( \cat V \)-matrix shall be denoted as \( r \colon X \to
Y \). If \( s \colon Y \to Z \) is another \( \cat V \)-matrix, the composite
\( s \cdot r \colon X \to Z \) is given at \( x,z \) by
\begin{equation*}
  (s \cdot r)(x,z) = \sum_{y \in Y} s(y,z) \times r(x,y),
\end{equation*}
and for every set \(X\), the unit \( \cat V \)-matrix \( 1_X \colon X \to X \)
is given at \( x,y \) by the terminal object if \( x = y \), and the initial
object if \( x \neq y \). 

We let \( \Equip_\lax \) be the 2-category of \textit{equipments}, \textit{lax
functors}, and \textit{icons} \cite{Lac10a}, and we let \( T = (T,m,e) \) be a
monad on \( \VMat \) in the 2-category \( \Equip_\lax \). We remark that \( T
\) has an underlying monad on \( \Set \), and it can be shown that \(T\) is
its \textit{lax extension}, under the terminology of \cite[p. 18]{CT03},
provided we relax the condition \( T1_r = 1_{Tr} \) for \( \cat V \)-matrices
\(r\).  This was observed in \cite[Subsection~1.13]{HST14} when \( \cat V \)
is a quantale, and in \cite[Appendix B]{CS10}.

The data for such a lax monad \(T\) consists of
\begin{itemize}[label=--,noitemsep]
  \item
    a set \( TX \) for each set \(X\),
  \item
    a \( \cat V \)-matrix \( Tr \colon TX \to TY \) for each \( \cat V
    \)-matrix \(r \colon X \to Y \),
  \item
    for each set \(X\), a family of comparison morphisms 
    \begin{equation*}
      \e{T}_X \colon 1 \to (T1_X)(\mathfrak x,\mathfrak x) 
    \end{equation*}
    indexed by \( \mathfrak x \in TX \),
  \item
    for each pair of \( \cat V \)-matrices \( r \colon X \to Y \), \( s \colon
    Y \to Z \), a family of comparison morphisms
    \begin{equation*}
      \m{T}_{r,s} \colon (Ts \cdot Tr)(\mathfrak x,\mathfrak z) \to (T(s \cdot
      r))(\mathfrak x, \mathfrak z)
    \end{equation*}
    indexed by \( \mathfrak x \in TX \), \( \mathfrak z \in TZ \),
  \item
    for each \( \cat V \)-matrix \( r \colon X \to Y \), a family of morphisms
    \begin{equation*}
      e_{r,x,y} \colon r(x,y) \to (Tr)(e(x),e(y)) 
    \end{equation*} 
    indexed by \( x \in X \), \( y \in Y \),
  \item
    for each \( \cat V \)-matrix \( r \colon X \to Y \), a family of morphisms
    \begin{equation*}
      m_{r,\mathfrak x,\mathfrak y} 
        \colon (TTr)(\mathfrak x,\mathfrak y) \to
                (Tr)(m(\mathfrak x),m(\mathfrak y)) 
    \end{equation*}
    indexed by \( \mathfrak x \in TTX \), \( \mathfrak y \in TTY \).
\end{itemize}
satisfying the following coherence conditions, where we omit the indexing
elements, as well as the associator and unitor isomorphisms for convenience:
\begin{equation*}
  \begin{tikzcd}
    Tt \cdot Ts \cdot Tr 
    \ar[d,"\id \cdot \m T",swap] \ar[r,"\m T \cdot \id"]
    & T(t \cdot s) \cdot Tr \ar[d,"\m T"] \\
    Tt \cdot T(s \cdot r) \ar[r,"\m T",swap]
    & T(t \cdot s \cdot r)
  \end{tikzcd}
  \quad
  \begin{tikzcd}
    Tr \ar[r,"\id \cdot \e T"]
       \ar[rd,equal]
    & Tr \cdot T1 \ar[d,"\m T"] \\
    & Tr
  \end{tikzcd}
  \quad
  \begin{tikzcd}
    Tr \ar[r,"\e T \cdot \id"]
       \ar[rd,equal]
    & T1 \cdot Tr \ar[d,"\m T"] \\
    & Tr
  \end{tikzcd}
\end{equation*}
\begin{equation*}
  \begin{tikzcd}
    TTs \cdot TTr \ar[r,"m_s \cdot m_r"] 
                  \ar[d,"\m T",swap]
    & Ts \cdot Tr \\
    T(Ts \cdot Tr) \ar[r,"T\m T",swap]
    & TT(s \cdot r) \ar[u,swap,"m_{s \cdot r}"]
  \end{tikzcd}
  \qquad
  \begin{tikzcd}
    1_{TTx} \ar[d,"1_{m_x}",swap] \ar[r,"\e T"] 
    & T1_{Tx} \ar[r,"T\e T",swap]
    & TT1_x \ar[d,"m_{1_x}"] \\
    1_{Tx} \ar[rr,"\e T",swap]
    && T1_x
  \end{tikzcd}
\end{equation*}
\begin{equation*}
  \begin{tikzcd}
    s \cdot r \ar[r,"e_s \cdot e_r"] 
              \ar[rd,"e_{s \cdot r}",swap]
    & Ts \cdot Tr \ar[d,"\m T"] \\
    & T(s \cdot r)
  \end{tikzcd}
  \qquad
  \begin{tikzcd}
    1_x \ar[r,"1_{e_x}"]
        \ar[rd,"e_{1_x}",swap]
    & 1_{Tx} \ar[d,"\e T"] \\
    & T1_x
  \end{tikzcd}
\end{equation*}
for each 3-chain of \( \cat V \)-matrices \( r,\, s,\, t \), as well as the
following associativity and identity conditions:
\begin{equation*}
  \begin{tikzcd}
    Tr \ar[r,"Te_r"] \ar[rd,equal]
    & TTr \ar[d,"m_r"]  \\
    & Tr
  \end{tikzcd}
  \qquad
  \begin{tikzcd}
    Tr \ar[r,"e_{Tr}"] \ar[rd,equal]
    & TTr \ar[d,"m_r"]  \\
    & Tr
  \end{tikzcd}
  \qquad
  \begin{tikzcd}
    TTTr \ar[r,"m_{Tr}"] \ar[d,"Tm_r",swap]
    & TTr \ar[d,"m_r"]  \\
    TTr \ar[r,"m_r",swap] & Tr
  \end{tikzcd}
\end{equation*}
where we have also omitted the indexing elements.

An enriched \( (T,\cat V) \)-category is a quadruple \( (X,a,\upsilon,\mu) \),
where \( X \) is a set, \( a \colon TX \to X \) is a \( \cat V \)-matrix, \(
\upsilon \) is a family of morphisms \( \upsilon_x \colon 1 \to a(e(x),x) \)
indexed by \(x \in X\), and \( \mu \) is a family of morphisms
\begin{equation*}
  \mu_{x_0,x_1,x_2} \colon a(x_2,x_1,x_0) \to a(m(x_2),x_0)
\end{equation*}
indexed by \( x_i \in T^iX \), where we define
\begin{equation*}
  a(x_2,x_1,x_0) =  a(x_1,x_0) \times (Ta)(x_2,x_1) 
  \quad\text{and}\quad
  a(x_3,x_2,x_1,x_0) = a(x_2,x_1,x_0) \times (TTa)(x_3,x_2)
\end{equation*}
for \( x_i \in T^iX \). These families satisfy the following identity 
and associativity laws:
\begin{equation*}
  \begin{tikzcd}
    a(x_1,x_0) \ar[r,"\id \times ((T\upsilon)_{x_1} \circ \e T)"]
               \ar[rd,equal]
      & a(e_T(x_1),x_1,x_0) \ar[d,"\mu"] \\
    & a(x_1,x_0)
  \end{tikzcd}
  \qquad
  \begin{tikzcd}
    a(x_1,x_0) \ar[r,"\upsilon_{x_0} \times e_a"]
               \ar[rd,equal]
      & a(e_T(x_1),e(x_0),x_0) \ar[d,"\mu"] \\
    & a(x_1,x_0)
  \end{tikzcd}
\end{equation*}
\begin{equation*}
  \begin{tikzcd}
    a(x_3,x_2,x_1,x_0) \ar[r,"\id \times (T\mu \circ \m T)"] 
                       \ar[d,"\mu \times m_a",swap]
    & a((Tm)(x_3),x_1,x_0) \ar[d,"\mu"] \\
    a((mT)(x_3),m(x_2),x_0) \ar[r,"\mu",swap]
    & a((m \circ mT)(x_3),x_0)
  \end{tikzcd}
\end{equation*}

\section{Change-of-base and embedding}
\label{sect:base.change}

Let \( \cat V \) be a lextensive category. The following diagram depicts
the adjunction fundamental to our study of viewing enriched generalized
multicategories as internal generalized multicategories:
\begin{equation}
  \label{eq:set.v.adj}
  \begin{tikzcd}
    \Set \ar[r,bend left,"- \pt 1"{above,name=A}]
    & \cat V \ar[l,bend left,"\cat V(1{,}-)"{name=B,below}]
    \ar[from=A,to=B,phantom,"\adj" {anchor=center, rotate=-90}]
  \end{tikzcd}
\end{equation}
We shall denote the counit of this adjunction by \( \heps \). As was done in
\cite[Theorem 9.11]{Luc18}, we will assume for the remainder of this chapter
that \( - \pt 1 \) is fully faithful, so that \( \Set \) may be the seen as
the full subcategory of \( \cat V \) consisting of the \textit{discrete
objects}. As observed in \cite[Lemma 2.2.1]{Ver92} (when \( \cat V \) is a
presheaf category), we have:

\begin{proposition}[{\cite[Remark 7.4, Lemma 7.6]{PL23b}}] 
  \label{prop:eqp.adj}
  We have an adjunction
  \begin{equation}
    \label{eq:vmat.spanv.adj}
    \begin{tikzcd}
      \VMat \ar[r,bend left,"- \pt 1"{name=A}]
      & \SpanV \ar[l,bend left,"\cat V(1{,}-)"{name=B,below}]
      \ar[from=A,to=B,phantom,"\adj" {anchor=center, rotate=-90}]
    \end{tikzcd}
  \end{equation}
  in the 2-category \( \Equip_\lax \). Moreover, \( - \pt 1 \colon \VMat \to
  \SpanV \) is fully faithful, and the underlying adjunction on the categories
  of objects is precisely \eqref{eq:set.v.adj}.
\end{proposition}

Spans \( a \colon X \to Y \) in a category \( \cat V \) are denoted by a
diagram such as:
\begin{equation*}
  \begin{tikzcd}
    & M_a \ar[ld,"l_a",swap] \ar[rd,"r_a"] \\
    X && Y
  \end{tikzcd}
\end{equation*}
So, if \( a \colon X \to Y \) is a span, the \( \cat V \)-matrix \( \cat
V(1,a) \) is given at \( x \in \cat V(1,X) \), \( y \in \cat V(1,Y) \) by the
pullback
\begin{equation*}
  \begin{tikzcd}
    \cat V(1,a)(x,y) \ar[d] \ar[r] 
                      \ar[rd,"\ulcorner",phantom,very near start]
    & 1 \ar[d,"{x,y}"]\\
    M_a \ar[r,"{l_a,r_a}",swap] & X \times Y
  \end{tikzcd}
\end{equation*}
while if \( r \colon S \to T \) is a \( \cat V \)-matrix, the span \( r \pt 1
\) is given by the coproduct of
\begin{equation*}
  \begin{tikzcd}
    & r(x,y) \ar[ld] \ar[rd] \\
    1 && 1
  \end{tikzcd}
  \quad\text{indexed by}\quad
  \begin{tikzcd}
    & S \times T \ar[ld,"d_1",swap] \ar[rd,"d_0"] \\
    S && T
  \end{tikzcd}
\end{equation*}

Now, we let \( T = (T,m,e) \) be a cartesian monad on \( \cat V \). By
\cite[Proposition A.2]{Her00}, \( T \) induces a pseudomonad on \( \SpanV \),
which may be seen as a monad in \( \Equip_\lax \), which we also denote by \(
T \); see \cite[Example A.6]{CS10}.

Via \eqref{eq:vmat.spanv.adj}, the monad \( T \) on \( \SpanV \)
induces a monad \( \Tt = (\Tt, \overline m, \overline e) \) on \( \VMat \)
(see \cite[Proposition 8.1]{PL23b}), where
\begin{itemize}[label=--,noitemsep]
  \item
    \( \Tt a = \cat V(1,T(a \pt 1)) \),
  \item
    \( \overline m_a = \cat V(1,m_{a \pt 1} \circ T\heps_{T(a \pt 1)}) \)
  \item
    \( \overline e_a = \cat V(1,e_{a \pt 1}) \circ \heta_a \),
\end{itemize}
for each \( \cat V \)-matrix \( a \). This monad is the lax extension of the
monad (also denoted \( \Tt \)) induced by \( T \) on \( \Set
\) via \eqref{eq:set.v.adj}. Under suitable conditions, the category of
enriched \( (\Tt,\cat V) \)-categories is embedded in the category of internal
\( (T,\cat V) \)-categories, which are our objects of interest in this
Chapter. 

\begin{theorem}[{\cite[Lemma 9.1, Theorem 9.2]{PL23b}}]
  \label{prop:tvcat.sub.cattv}
  If \( \heps_{T(-\pt 1)} \) is a cartesian natural transformation,
  we have an adjunction
  \begin{equation}
    \label{eq:tvcat.cattv.adj}
    \begin{tikzcd}
      \TtVCat \ar[r,bend left,"- \pt 1"{name=A}]
      & \CatTV \ar[l,bend left,"\cat V(1{,}-)"{name=B,below}]
      \ar[from=A,to=B,phantom,"\adj" {anchor=center, rotate=-90}]
    \end{tikzcd}
  \end{equation}
  and the functor \( - \pt 1 \colon \TtVCat \to \CatTV \) is fully faithful.
\end{theorem}

The embedding functor \( - \pt 1 \colon \TtVCat \to \CatTV \) is given on an
enriched \( (\Tt,\cat V) \)-category \( (X,a,\upsilon,\mu) \) by the span
\begin{equation*}
  \begin{tikzcd}
    && M_{a \pt 1} \ar[ld,"l_{a \pt 1}",swap] \ar[rd,"r_{a \pt 1}"] \\
    & \Tt X \pt 1 \ar[ld,"\heps_{T(X\pt 1)}",swap] && X \pt 1 \\
    T(X \pt 1)
  \end{tikzcd}
\end{equation*}
with unit \( \hups \) given by
\begin{equation*}
  \begin{tikzcd}
    1 \ar[d,"\heta_x",swap] \ar[r,"\upsilon_x \pt 1"]
      & a(e(x),x) \pt 1 \ar[d,"\iota_{e(x),x}"] \\
    X \pt 1 \ar[r,dashed,"\hups",swap] 
      & M_{a \pt 1}
  \end{tikzcd}
\end{equation*}
via the universal property of the coproduct, and the composition \( \hmu \) is
given by
\begin{equation*}
  \begin{tikzcd}
    a(x_2,x_1,x_0)\pt 1
      \ar[r,"\mu_{x_2,x_1,x_0} \pt 1"]
      \ar[d]
    & a(m(x_2),x_0) \pt 1 \ar[d] \\
    \displaystyle\sum_{x_i \in T^iX} a(x_2,x_1,x_0) \pt 1
      \ar[r,dashed,swap,"\hmu"]
    & M_{a\pt 1}
  \end{tikzcd}
\end{equation*}
via the universal property of the coproduct. The hypothesis that \(
\heps_{T(- \pt 1)} \) is a cartesian natural transformation ensures that the
following diagram
\begin{equation*}
  \begin{tikzcd}
    \displaystyle\sum_{x_i \in T^iX}
          \Tt a(x_2,x_1) \pt 1 \ar[r] \ar[d,"\heps_{T(a\pt 1)}",swap]
                      \ar[rd,"\ulcorner",phantom,very near start]
    & \overline TX \pt 1 \ar[d,"\heps_{T(X \pt 1)}"] \\
    T M_{a \pt 1} \ar[r,"Tr_{a\pt 1}",swap] & T(X \pt 1)
  \end{tikzcd}
\end{equation*}
is a pullback square, by \cite[Lemma 8.3]{PL23b}, thereby guaranteeing that
\begin{equation*}
  \sum_{x_i \in T^iX} a(x_2,x_1,x_0)
\end{equation*}
is the object of 2-chains of \( (X,a,\upsilon,\mu) \pt 1 \). 

If \( (f,\phi) \colon (X,a,\upsilon,\mu) \to (Y,b,\upsilon,\mu) \) is an
enriched \( (\Tt,\cat V) \)-functor, \( (f,\phi) \pt 1 \) is given on objects
by \( f \pt 1 \colon X \pt 1 \to Y \pt 1 \), and \( \phi \pt 1 \colon M_{a \pt
1} \to M_{b \pt 1} \) on morphisms.

\section{Reflection of effective descent morphisms}
\label{sect:eff.desc.refl}

Having established all the necessary notation, we can now proceed to study the
effective descent morphisms in \( \TtVCat \), by studying whether the
embedding \( - \pt 1 \colon \TtVCat \to \CatTV \) reflects effective descent
morphisms. The key idea, developed by the next result, is that we must
guarantee that the full inclusion \( \TtVCat \to \CatTV \) consists precisely
of those internal \( (T,\cat V) \)-categories with a discrete object of
objects, which could not be guaranteed in general.

\begin{lemma}[{\cite[Lemma 10.2]{PL23b}}]
  \label{lem:counit.iso.when}
  Let \( X \) be an internal \( (T,\cat V) \)-category whose object-of-objects
  is discrete; that is, \( X_0 \iso S \pt 1 \) for a set \(S\).

  If we let \( a \) be the span in \( \cat V \) given by the underlying
  \(T\)-graph of \( X \), as depicted in \eqref{eq:span.of.x},
  \begin{equation}
    \label{eq:span.of.x}
    \begin{tikzcd}
      & X_1 \ar[ld,"d_1",swap] \ar[rd,"d_0"] \\
      TX_0 && X_0
    \end{tikzcd}
  \end{equation}
  then \( \heps_a \) is a split epimorphism. Moreover, if \( \heps_{T1} \colon
  \Tt 1 \pt 1 \to T1 \) is a monomorphism, then \( \heps_a \) is an
  isomorphism.
\end{lemma}

\begin{proof}
  We may assume that \( X_0 = S \pt 1 \). Our first step is to notice that \(
  d_1 \colon X_1 \to T(S \pt 1) \) factors uniquely through \( \heps_{T(S \pt
  1)} \); we have \( e_1 = \heps_{T1} \circ (\overline e_1 \pt 1) \) (by
  definition of \( \overline e \)), so there exists a unique \( \hat d_1 \),
  depicted by dashed morphism in \eqref{eq:epscart}
  \begin{equation}
    \label{eq:epscart}
    \begin{tikzcd}
      X_1 \ar[rd,dashed,"\hat d_1" description]
          \ar[rrd,bend left=15,"d_1"] \ar[d,"!",swap] \\
      1 \ar[rd,"\overline e_1 \pt 1",swap]
        & \Tt S \pt 1 \ar[r,"\heps_{T(S\pt 1)}"]
                      \ar[d,"\Tt ! \pt 1",swap] 
                      \ar[rd,"\ulcorner",phantom,very near start]
        & T(S \pt 1) \ar[d,"T(! \pt 1)"] \\
        & \Tt 1 \pt 1 \ar[r,"\heps_{T1}",swap]
        & T1
    \end{tikzcd}
  \end{equation}
  making the adjacent diagrams commute.

  We may conclude that there is a unique morphism \( \omega \colon X_1 \to
  M_{\cat V(1,a)\pt 1} \) such that \( \heps_a \circ \omega = \id \)
  (confirming that \( \heps_a \) is a split epimorphism) and \( (\hat d_1,
  d_0) = (d_1, d_0) \circ \omega \), as depicted in
  \eqref{eq:epssplit}\footnote{It should be noted that \( \heps_a \) is
  \textit{defined} by this pullback square, see {\cite[(2.5)]{PL23b}}, noting
  that \( \cat V \) is lextensive.}

  \begin{equation}
    \label{eq:epssplit}
    \begin{tikzcd}
      X_1 \ar[rrrd,"{\hat d_1,d_0}",bend left=15]
          \ar[rdd,equal,bend right=25]
          \ar[rd,"\omega" description,dashed] \\
      & M_{\cat V(1,a) \pt 1} \ar[d,"\heps_a",swap]
                              \ar[rr,"{l_{\cat V(1,a) \pt 1},
                                       r_{\cat V(1,a) \pt 1}}"] 
                              \ar[rrd,"\ulcorner",phantom,very near start]
      && \Tt S \pt 1 \times S \pt 1 \ar[d,"\heps_{T(S \pt 1)} \times \id"] \\
      & X_1 \ar[rr,swap,"{d_1,d_0}"] && T(X \pt 1) \times S \pt 1
    \end{tikzcd}
  \end{equation}

  Moreover, if \( \heps_{T1} \) is a monomorphism, then, by the pullback
  square in \eqref{eq:epscart}, so is \( \heps_{T(S \pt 1)} \), and by the
  pullback square in \eqref{eq:epssplit}, we conclude \( \heps_a \) is a
  monomorphism. Therefore, \( \heps_a \) must be invertible.
\end{proof}

As a corollary, we conclude that the enriched \( (\Tt,\cat V) \)-categories
are precisely the internal \( T \)-categories whose object of objects
is discrete. More precisely, we have:

\begin{lemma}[{\cite[Theorem 10.3]{PL23b}}]
  \label{lem:tvcat.pspb}
  If \( \heps_{T1} \) is a monomorphism, we have a pseudopullback diagram
  \begin{equation*}
    \begin{tikzcd}
      \TtVCat \ar[r,"- \pt 1"] \ar[d]
              \ar[rd,"\iso" description,phantom]
        & \CatTV \ar[d] \\
      \Set \ar[r,swap,"- \pt 1"]
        & \cat V
    \end{tikzcd}
  \end{equation*}
  of categories with pullbacks and pullback-preserving functors.
\end{lemma}

\begin{proof}
  The objects of the pseudopullback are triples \( (S,X,\phi) \) where \( S \)
  is a set, \( X \) is an internal \( (T,\cat V) \)-category, and \( \phi
  \) is an isomorphism \( \phi \colon S \pt 1 \to X \). By Lemma
  \ref{lem:counit.iso.when}, it follows that \( \heps_a \) is invertible,
  where \( a \) is the span given by the underlying \(T\)-graph of \(X\), as
  in \eqref{eq:span.of.x}. 

  By general remarks about change-of-base adjunctions between horizontal lax
  algebras given in~\cite[Section 6]{PL23b}, this implies that \( X \) is
  isomorphic to an enriched \( (\Tt, \cat V) \)-category.
\end{proof}

Everything is set up to apply the results about effective descent morphisms in
bilimits from Chapter \ref{chap:descent}, which gives the following reflection
result:

\begin{lemma}[{\cite[Lemma 10.4]{PL23b}}]
  \label{lem:lem.eff.desc.refl.tvcat}
  If \( \heps_{T1} \) is a monomorphism, then \( - \pt 1 \colon \TtVCat \to
  \CatTV \) reflects effective descent morphisms.
\end{lemma}

\begin{proof}
  We follow the same approach as Theorem \ref{thm:vcat.famv.cat.refl}. As
  stated in Remark \ref{rem:fln.gen}, the functor \( \CatTV \to \cat V \)
  preserves descent morphisms, which are reflected by \( - \pt 1 \colon \Set
  \to \cat V \). 

  Since descent morphisms in \( \Set \) are effective for descent, we may
  apply Proposition \ref{prop:pspb.descent} and Lemma \ref{lem:tvcat.pspb} to
  conclude our result.
\end{proof}

Now, we can apply our knowledge of effective descent morphisms in \( \CatTV \)
to obtain the main result of this chapter:

\begin{theorem}[{\cite[Theorem 10.5]{PL23b}}]
  \label{thm:desc.tvcat}
  Let \( (f, \phi) \colon (X,a,\upsilon,\mu) \to (Y,b,\upsilon,\mu) \) be a
  functor of enriched \( (\Tt,\cat V) \)-categories. If \( \heps_{T1} \)
  is a monomorphism, and if
  \begin{itemize}[label=--,noitemsep]
    \item
      \( ((f,\phi) \pt 1)_1 \) is an effective descent morphism,
    \item
      \( ((f,\phi) \pt 1)_2 \) is a descent morphism,
    \item
      \( ((f,\phi) \pt 1)_3 \) is an almost descent morphism,
  \end{itemize}
  then \( (f,\phi) \) is an effective descent morphism in \( \TtVCat \).
\end{theorem}

\begin{proof}
  By Theorem \ref{thm:int.eff.desc}, the above conditions guarantee that \(
  (f, \phi) \pt 1 \) is an effective descent morphism in \( \CatTV \). Since
  \( \heps_{T1} \) is a monomorphism, we may apply Lemma
  \ref{lem:lem.eff.desc.refl.tvcat} to conclude that \( (f,\phi) \) is an
  effective descent morphism in \( \TtVCat \).
\end{proof}

\section{Scope of the findings}
\label{sect:scope}

Our main result holds in the context of a lextensive category \( \cat V \)
such that \( - \pt 1 \colon \Set \to \cat V \) is fully faithful. These
properties are enjoyed by the categories \( \Cat \), \( \Top \), any connected
Grothendieck topos, and any free coproduct completion \( \FamB \) of a
category \( \cat B \) with finite limits. Moreover, we have required two more
hypotheses:
\begin{enumerate}[label=(\alph*),noitemsep]
  \item
    \label{enum:fibwise.disc}
    \( \heps_{T(- \pt 1)} \) is a cartesian natural transformation.
  \item
    \label{enum:disc.obj.objs}
    \( \heps_{T1} \) is a monomorphism.
\end{enumerate}

We can verify that \ref{enum:disc.obj.objs} holds when
\begin{itemize}[noitemsep]
  \item
    the terminal object is a \textit{separator}, that is, when \( \cat V(1,-)
    \) is faithful, so that \( \heps \) is a componentwise monomorphism. This
    is the case for \( \Cat \), \( \Top \), any \textit{hyperconnected}
    Grothendieck topos (by definition), but not the case for \( \Grph \) nor
    \( \Fam(\Set) \)\footnote{These are Grothendieck toposes which are not
    hyperconnected.},
  \item
    \( T \) is \textit{discrete}, that is \( \heps_{T1} \) is an isomorphism.
    This is the case when \(T\) is the free monoid monad on any category \(
    \cat V \) under the above conditions, but not when \(T\) is the free
    category monad on \( \Grph \).
\end{itemize}
Further discussion may be found in \cite[Section 10]{PL23b}.

On the other hand, we have found no examples of cartesian monads \( T \) on \(
\cat V \) that do not satisfy \ref{enum:fibwise.disc}. This is discussed at
length in \cite[Section 8]{PL23b}; here we merely recall the results required
to discuss our examples.

\subsection{Classical multicategories}

The free monoid monad \( \mathfrak M \) on \( \cat V \) satisfies
\ref{enum:fibwise.disc} and \ref{enum:disc.obj.objs} \cite[Lemma 8.7 and
Subsection 10.2]{PL23b}. In this case, we let \( \VMultiCat =
(\overline{\mathfrak M},\cat V)\dash \Cat \) be the category of
\textit{enriched $\cat V$-multicategories}.

If \( (f,\phi) \colon (X,a,\upsilon,\mu) \to (Y,b,\upsilon,\mu) \) is a
functor of enriched \( \cat V \)-multicategories such that
\begin{equation*}
  \phi \colon \sum_{x_i \in (\mathfrak M^if)^*(y_i)} 
                a(x_1, x_0) \to b(y_1, y_0)
\end{equation*}
is an effective descent morphism,
\begin{equation*}
  \phi \times \mathfrak M \phi 
    \colon \sum_{x_i \in (\mathfrak M^if)^*(y_i)}
                a(x_2, x_1, x_0) \to b(y_2, y_1, y_0) 
\end{equation*}
is a descent morphism, and
\begin{equation*}
  \phi \times \mathfrak M \phi \times \mathfrak M^2 \phi \colon
    \sum_{x_i \in (\mathfrak M^if)^*(y_i)}
         a(x_3, x_2, x_1, x_0) \to b(y_3, y_2, y_1, y_0) 
\end{equation*}
is an almost descent morphism, for all \( y_i \in \mathfrak M^iY \) with \(
i=0,1,2,3 \), then \( (f,\phi) \) is an effective descent morphism in \(
\VMultiCat \), by Theorem \ref{thm:desc.tvcat}.

\subsection{Cartesian and cocartesian multicategories}

The free finite coproduct completion \( \FinFam \) on \( \Cat \) satisfies
\ref{enum:fibwise.disc} \cite[Subsection 8.5]{PL23b}. Since the terminal
category is a separator in \( \Cat \), we conclude that
\ref{enum:disc.obj.objs} holds. Thus, by Lemma \ref{lem:tvcat.pspb}, \(
(\overline{\Fam}_\fin, \Cat) \dash \Cat \) is the category of enhanced
cocartesian multicategories with a discrete object-of-objects. For this
reason, we refer to its objects as cocartesian multicategories. Likewise, \(
(\overline{\Fam}^*_\fin, \Cat) \dash \Cat \) is the category of cartesian
multicategories\footnote{These are a \textit{wide} subcategory of the category
of multi-sorted Lawvere theories -- the morphisms between cartesian
multicategories are precisely the ``degree one'' morphisms between Lawvere
theories. See also \cite[Example 4.17]{CS10}.}.

Via the description of effective descent morphisms for functors of enhanced
multicategories, given in Section \ref{sect:application} we obtain a
description of the effective descent functors for cartesian and cocartesian
multicategories.

\subsection{Graded, operadic and symmetric multicategories}

Let \( S, T \) be endofunctors on \( \cat V \), and let \( \alpha \colon S
\to T \) be a cartesian natural transformation. If \(T\) satisfies
\ref{enum:fibwise.disc}, so does \( S \). Thus, it follows that
\ref{enum:fibwise.disc} is satisfied by \( \cat V \)-operadic monads, as well
as the free symmetric strict monoidal category monad \( \mathfrak S \), when
\( \cat V = \Cat \). 

However, we do not guarantee that every \( \cat V \)-operadic monad satisfies
\ref{enum:disc.obj.objs} in general. It certainly is true if the terminal
object of \( \cat V \) is a separator, and if \( \mathfrak O \) is a \( \cat V
\)-operad such that \( \mathfrak O_n \) is discrete for all \( n \in \N \),
then the \( \cat V \)-operadic monad \( \mathfrak M_{\mathfrak O} \) induced
by \( \mathfrak O \) is discrete, so the property is also satisfied in this
setting. We call such \( \cat V \)-operads \textit{discrete}.

If the \( \cat V \)-operad \( \mathfrak O \) is discrete, we let \(
(\overline{\mathfrak M}_{\mathfrak O},\cat V)\dash \Cat \) be the category of
\textit{enriched \( \mathfrak O \)-categories}. Via the results of
\ref{sect:application}, and Theorem \ref{thm:desc.tvcat}, we obtain a
description for the effective descent functors of enriched \( \mathfrak O
\)-categories. In particular, we also obtain the \textit{enriched graded
multicategories} by a discrete \( \cat V \)-monoid \(M\), and a description of
the respective effective descent functors. 

Since the terminal category is a separator in \( \Cat \), it follows that
\ref{enum:disc.obj.objs} is satisfied for \( \mathfrak S \). Arguing as we did
in the case of (co)cartesian multicategories, we let \( (\overline{\mathfrak
S},\Cat)\dash \Cat \) be the category of \textit{symmetric multicategories},
for which we also obtain a description of the respective effective descent
functors.

\chapter{Fibration of split opfibrations}
\label{chap:fib-descent}

The bifibration \( F_D = \CAT(-,\Set) \colon \Cat^\op \to \CAT \) of discrete
opfibrations is the main object of study in \cite{Sob04}. Therein, the
functors \(p \colon E \to B \) of (effective) \( F_D \)-descent were
characterized.  Indeed, \( p \) is a \( F_D \)-descent morphism if and only if
\(p\) is a \textit{lax epimorphism}, while \( p \) is an effective \( F_D
\)-descent morphism if and only if \(p\) is a \textit{fully faithful} lax
epimorphism.

Our goal for this chapter, covering the work done in \cite{LPS23}, is to show
that the results of \cite{Sob04} for discrete opfibrations can be applied
just as well to other settings. Among them, we are able to characterize the
effective \( F \)-descent morphisms for the bifibration \( F = \CAT(-,\Cat)
\colon \Cat^\op \to \CAT \) of split opfibrations. We confirm that a functor
\( p \) is of (effective) \( F \)-descent if and only if it is of (effective)
\( F_D \)-descent (Theorem \ref{thm:desc.fib.opfib}).

We begin Section \ref{sect:ffle} by recalling the notions of fully faithful
morphism and lax epimorphism in a 2-category \( \bicat A \), as well as a
couple of relevant results. Then we restrict our attention to the setting of
enriched categories, recalling from \cite{LS21} the notions of \( \cat V
\)-fully faithful and \(\cat V\)-lax epimorphic \( \cat V \)-functors, for \(
\cat V \) a complete and cocomplete symmetric monoidal closed category, and
comparing them with the notions of fully faithful morphism and lax epimorphism
in the 2-category \( \VCat \) of \textit{small} \( \cat V \)-categories. 

\section{Fully faithful morphisms and lax epimorphisms}
\label{sect:ffle}

Let \( \bicat A \) be a 2-category. A morphism \( f \colon x \to y \) is said
to be
\begin{itemize}[label=--,noitemsep]
  \item
    \textit{fully faithful} if \( \bicat A(w,f) \colon \bicat A(w,x) \to
    \bicat A(w,y) \) is fully faithful for all \( w \), 
  \item
    \textit{lax epimorphic} if \( \bicat A(f,z) \colon \bicat A(y,z) \to
    \bicat A(x,z) \) is fully faithful for all \( z \).
\end{itemize}

A comprehensive study of lax epimorphisms in a 2-category is carried out in
\cite{LS21}. We shall recall some fundamental aspects.

If we have an adjunction \( f \adj g \) in \( \bicat A \), the following are
equivalent:
\begin{itemize}[label=--,noitemsep]
  \item
    the unit of \( f \adj g \) is invertible,
  \item
    \( g \) is fully faithful,
  \item
    \( f \) is a lax epimorphism. 
\end{itemize}
Codually, it follows that the following are equivalent:
\begin{itemize}[label=--,noitemsep]
  \item
    the counit of \( f \adj g \) is invertible,
  \item
    \( g \) is a lax epimorphism,
  \item
    \( f \) is fully faithful.
\end{itemize}
From this, we may conclude that:

\begin{proposition}[{\cite[p. 134]{LPS23}}]
  \label{prop:adj.ff.lax.epi.is.eqv}
  If a morphism \( f \) has a left or right adjoint, and is a fully faithful
  lax epimorphism, then \(f\) is an equivalence.
\end{proposition}

For the remainder of this chapter, we assume that \( \cat V \) is a complete
and cocomplete, symmetric monoidal closed category. We now restrict our scope
to the 2-category \( \bicat A = \VCat \) of small \( \cat V \)-categories, \(
\cat V \)-functors and \( \cat V \)-natural transformations; this setting was
studied in \cite[Section 5]{LS21}. We will recall the pertinent definitions
and results therein.

A \( \cat V \)-functor \( p \colon E \to B \) between small \( \cat V
\)-categories is said to be
\begin{itemize}[label=--,noitemsep]
  \item
    \( \cat V \)-fully faithful if \( p \colon E(a,b) \to B(pa,pb) \) is an
    isomorphism for all objects \( a, b \),
  \item
    a \( \cat V \)-lax epimorphism if the \( \cat V \)-functor
    \begin{equation*}
      \VCat[p,C] \colon \VCat[B,C] \to \VCat[E,C]
    \end{equation*}
    is \( \cat V \)-fully faithful for all small \( \cat V \)-categories
    \(C\) \cite[Definition 5.4]{LS21}.
\end{itemize}

\begin{proposition}[{\cite[Lemma 5.1]{LS21}}]
  \label{prop:vcat.ff}
  If \(p \colon E \to B \) is a \( \cat V \)-fully faithful \( \cat V
  \)-functor, then \(p\) is a fully faithful morphism in \( \VCat \). The
  converse holds if \(p\) has a left or right adjoint.
\end{proposition}

\begin{proposition}[{\cite[Theorem 5.6]{LS21}}]
  \label{prop:vcat.laxepi}
  Let \( p \colon E \to B \) be a \( \cat V \)-functor. The following are
  equivalent:
  \begin{itemize}[label=--,noitemsep]
    \item
      \( p \) is a lax epimorphism in \( \VCat \),
    \item
      \( p \) is a \( \cat V \)-lax epimorphism,
    \item
      the functor \( \VCAT(p, \cat V) \colon \VCAT(B, \cat V) \to
      \VCAT(E, \cat V) \) is fully faithful.
  \end{itemize}
\end{proposition}

The result analogous to Proposition \ref{prop:vcat.laxepi} for \( \cat V
\)-fully faithful \( \cat V \)-functors, despite not being a consequence of
duality, can be shown to hold as well:

\begin{lemma}[{\cite[Proposition 2.4]{LPS23}}]
  \label{lem:vcat.vff}
  A \( \cat V \)-functor \( p \colon E \to B \) is \( \cat V \)-fully faithful
  if and only if the functor \( \VCAT(p,\cat V) \) is a lax epimorphism.
\end{lemma}

\begin{proof}
  We have an adjunction \( \Lan_p \adj \VCAT(p,\cat V) \). From previous
  remarks, \( \VCAT(p,\cat V) \) is a lax epimorphism if and only if \( \Lan_p
  \) is fully faithful.

  By the enriched Yoneda lemma, \( p \) is \( \cat V \)-fully faithful if and
  only if \( \Lan_p \) is fully faithful (see \cite[Proposition
  4.23]{Kel05}).
\end{proof}

\section{Enriched Cauchy completion}

Let \( \cat C \) be a (possibly large) \( \cat V \)-category. An object \( x
\) on \( \cat C \) is said to be \textit{tiny} (also called \textit{absolutely
presentable} in \cite{BD86} and \textit{small-projective} in \cite{Kel05}) if
the representable \( \cat V \)-functor \( \cat C(x,-) \colon \cat C \to \cat V
\) preserves colimits. 

The \textit{Cauchy completion} of a small \( \cat V \)-category \( X \) is the
full sub-\( \cat V \)-category of tiny objects of \( \VCAT[X^\op,\cat V] \),
and is denoted by \( \Cauchy X \). We observe that, in general, \( \Cauchy X \) is
not a small \( \cat V \)-category; for instance, let \( \cat V \) be the
category of complete lattices. In \cite[Section 5.5]{Kel05}, it was shown that
\( \Cauchy I \) is not small, where \( I \) is the unit \( \cat V
\)-category. Hence, we will assume for the remainder of this chapter that:
\begin{equation}
  \label{eq:cauchy.small}
  \text{\( \Cauchy X \) is a small \( \cat V \)-category for all small \( \cat V
  \)-categories \(X\).}
\end{equation}
This property holds for many base categories \( \cat V \) of our interest,
such as \( \Cat \), \( \Set \), and any small quantale. More generally, it was
shown in \cite{Joh89} that if the underlying category of \( \cat V \) is
locally presentable, then \eqref{eq:cauchy.small} holds.

The following result confirms that \( \cat V \)-equivalences preserve tiny
objects:

\begin{lemma}[{\cite[Lemma 2.1]{LPS23}}]
  \label{lem:tiny.preserved}
  Let \( F \adj G \colon \cat D \to \cat C \) be a \( \cat V \)-adjunction of
  (possibly large) \( \cat V \)-categories. If \(G\) preserves colimits, then
  \(F\) preserves tiny objects.
\end{lemma}

\begin{proof}
  If \( a \) is tiny, then \( \cat D(Fa,-) \iso \cat C(a,G(-)) \) is a
  composite of functors that preserve colimits, hence \(Fa\) is tiny.
\end{proof}

A consequence of Lemma \ref{lem:tiny.preserved} is that for any \( \cat V
\)-functor \( p \colon X \to Y \), we may define  \( \Cauchy p \colon \Cauchy
X \to \Cauchy Y \) by restricting the enriched left Kan extension \( \Lan_p
\colon \VCAT[X,\cat V] \to \VCAT[Y,\cat V] \) to the tiny objects.  Indeed, we
have a chain of \( \cat V \)-adjunctions
\begin{equation*}
  \Lan_p \adj \VCAT[p,\cat V] \adj \Ran_p
\end{equation*}
which confirms that \( \VCAT[p, \cat V] \) preserves colimits, so \( \Lan_p \)
preserves tiny objects.

By the enriched Yoneda lemma, we readily confirm that any representable
\( \cat V \)-presheaf is tiny,
\begin{align*}
  \VCAT[X^\op, \cat V](X(-,x),\colim(W,F))
    &\iso \colim(W,F)x \\
    &\iso \colim(W,F(-,x))  \\
    &\iso \colim(W,\VCAT[X^\op, \cat V](X(-,x),F)),
\end{align*}
so that the Yoneda \( \cat V \)-embedding \( \mathfrak y \colon X \to
\VCAT[X^\op, \cat V] \) restricts to a \( \cat V \)-functor \( \eta_X \colon X
\to \Cauchy X \). It also follows that for any \( \cat V \)-functor \( p
\colon X \to Y \), we have a \( \cat V \)-natural isomorphism \( \eta \circ p
\iso \Cauchy p \circ \eta \), and that \( \VCAT(\eta,\cat V) \) is an 
equivalence.

This allows us to highlight the relationship between Cauchy completion and
fully faithful morphisms/lax epimorphisms:

\begin{lemma}[{\cite[Lemma 2.2]{LPS23}}]
  \label{lem:ff.laxepi}
  If \( p \colon X \to Y \) is a \( \cat V \)-functor, then the induced
  functor 
  \begin{equation*}
    \VCAT(p,\cat V) \colon \VCAT(Y,\cat V) \to \VCAT(X,\cat V) 
  \end{equation*}
  is fully faithful (respectively, a lax epimorphism) if and only if \(
  \VCAT(\Cauchy p,\cat V) \) is fully faithful (respectively, a lax
  epimorphism).
\end{lemma}

\begin{proof}
  We observe that the following diagram commutes up to isomorphism:
  \begin{equation*}
    \begin{tikzcd}[column sep=large]
      \VCAT(\Cauchy Y,\cat V) 
        \ar[r,"{\VCAT(\eta_Y,\cat V)}"]
        \ar[d,"{\VCAT(\Cauchy p,\cat V)}",swap]
        \ar[rd,"\iso" description,phantom]
      & \VCAT(Y, \cat V)
        \ar[d,"{\VCAT(p,\cat V)}"] \\
      \VCAT(\Cauchy X,\cat V) 
        \ar[r,"{\VCAT(\eta_X,\cat V)}",swap]
      & \VCAT(X, \cat V)
    \end{tikzcd}
  \end{equation*}
  Since the rows are equivalences, the result follows.
\end{proof}

We immediately conclude that:

\begin{corollary}[{\cite[Proposition 2.3]{LPS23}}]
  \label{cor:lax.epi}
  The following are equivalent for a \( \cat V \)-functor \( p \colon X \to Y
  \):
  \begin{enumerate}[label=(\roman*),noitemsep]
    \item
      \label{enum:p.laxepi}
      \( p \) is a lax epimorphism,
    \item
      \label{enum:cp.laxepi}
      \( \Cauchy p \) is a lax epimorphism,
    \item
      \label{enum:pstar.ff}
      \( \VCAT(p,\cat V) \) is fully faithful.
    \item
      \label{enum:cpstar.ff}
      \( \VCAT(\Cauchy p,\cat V) \) is fully faithful.
  \end{enumerate}
\end{corollary}
\begin{proof}
  The equivalences \ref{enum:p.laxepi} \( \iff \) \ref{enum:pstar.ff} and
  \ref{enum:cp.laxepi} \( \iff \) \ref{enum:cpstar.ff} were obtained in
  Proposition \ref{prop:vcat.laxepi}, and the equivalence \ref{enum:pstar.ff}
  \( \iff \) \ref{enum:cpstar.ff} follows by Lemma \ref{lem:ff.laxepi}.
\end{proof}

\begin{corollary}[{\cite[Proposition 2.4]{LPS23}}]
  \label{cor:vff}
  The following are equivalent for a \( \cat V \)-functor \( p \colon X \to Y
  \):
  \begin{enumerate}[label=(\roman*),noitemsep]
    \item
      \label{enum:p.vff}
      \( p \) is \( \cat V \)-fully faithful,
    \item
      \label{enum:cp.vff}
      \( \Cauchy p \) is \( \cat V \)-fully faithful,
    \item
      \label{enum:pstar.laxepi}
      \( \VCAT(p,\cat V) \) is a lax epimorphism.
    \item
      \label{enum:cpstar.laxepi}
      \( \VCAT(\Cauchy p,\cat V) \) is a lax epimorphism.
  \end{enumerate}
\end{corollary}
\begin{proof}
  The equivalences \ref{enum:p.vff} \( \iff \) \ref{enum:pstar.laxepi} and
  \ref{enum:cp.vff} \( \iff \) \ref{enum:cpstar.laxepi} were obtained in
  Proposition \ref{lem:vcat.vff}, and we obtain the equivalence
  \ref{enum:pstar.laxepi} \( \iff \) \ref{enum:cpstar.laxepi} via Lemma
  \ref{lem:ff.laxepi}.
\end{proof}

Combining Corollaries \ref{cor:lax.epi} and \ref{cor:vff}, we obtain

\begin{theorem}[{\cite[Theorem 2.5]{LPS23}}]
  \label{thm:cauchy}
  The following are equivalent for a \( \cat V \)-functor \( p \colon X \to Y
  \):
  \begin{enumerate}[label=(\roman*),noitemsep]
    \item
      \label{enum:p.vff.laxepi}
      \( p \) is a \( \cat V \)-fully faithful lax epimorphism,
    \item
      \label{enum:cp.vff.laxepi}
      \( \Cauchy p \) is a \( \cat V \)-fully faithful lax epimorphism,
    \item
      \label{enum:pstar.vff.laxepi}
      \( \VCAT(p,\cat V) \) is a fully faithful lax epimorphism,
    \item
      \label{enum:cp.eqv}
      \( \Cauchy p \) is an equivalence,
    \item
      \label{enum:pstar.eqv}
      \( \VCAT(p,\cat V) \) is an equivalence.
  \end{enumerate}
\end{theorem}
\begin{proof}
  Corollaries \ref{cor:lax.epi} and \ref{cor:vff} guarantee that
  \ref{enum:p.vff.laxepi} \( \iff \) \ref{enum:cp.vff.laxepi} \( \iff \)
  \ref{enum:pstar.vff.laxepi}. 

  As any equivalence is a fully faithful lax epimorphism, we have
  \ref{enum:cp.eqv} \( \implies \) \ref{enum:cp.vff.laxepi}, and since
  equivalences fix tiny objects, we conclude \ref{enum:pstar.eqv} \( \implies
  \) \ref{enum:cp.eqv}.

  Finally, we note that we have an adjunction \( \Lan_p \adj \VCAT(p, \cat V)
  \), so we obtain \ref{enum:pstar.vff.laxepi} \( \implies \)
  \ref{enum:pstar.eqv} by Proposition \ref{prop:adj.ff.lax.epi.is.eqv}.
\end{proof}

\section{Descent for the bifibration of split fibrations}

\begin{lemma}[{\cite[Proposition 3.1]{LPS23}}]
  \label{lem:star.set.cat}
  A functor \( p \colon E \to B \) between small categories is fully faithful
  (respectively, a lax epimorphism) if and only if \( \CAT(p,\Cat) \) is a lax
  epimorphism (respectively, fully faithful).
\end{lemma}

\begin{proof}
  We consider the fully faithful functor \( J \colon \Set \to \Cat \), which
  has left and right adjoint functors. Hence, its direct image \( J_!  \colon
  \Set\dash \CAT \to \Cat\dash \CAT \) is a fully faithful 2-functor, and has
  left and right 2-adjoints. By \cite[Remark 2.8 and Lemma 2.10]{LS21}, we
  conclude that \( J_! \) creates fully faithful morphisms and lax
  epimorphisms.

  Hence, we conclude from Corollary \ref{cor:vff} (respectively, Corollary
  \ref{cor:lax.epi}) that \( J_!p \) is fully faithful (a lax epimorphism) if
  and only if \( \Cat \dash \CAT(J_!p,\Cat) \iso \CAT(p,\Cat) \) is a lax
  epimorphism (fully faithful).
\end{proof}

\begin{theorem}[{\cite[Theorem 3.2]{LPS23}}]
  \label{thm:desc.fib.opfib}
  For a functor \(p \colon E \to B \) between small categories, we consider
  the lax codescent factorization:
  \begin{equation*}
    \begin{tikzcd}
      E \ar[rr,"p"] \ar[rd]
        && B \\
      & \CoDesc(\Ker(p)) \ar[ur,"\mathcal K^{\Ker(p)}",swap] 
    \end{tikzcd}
  \end{equation*}
  The following are equivalent:
  \begin{enumerate}[label=(\roman*),noitemsep]
    \item
      \label{enum:f.desc}
      \( p \) is an effective \( F \)-descent morphism (respectively, \( F
      \)-descent morphism),
    \item
      \label{enum:kstar.cat.eqv}
      \( \CAT(\mathcal K^{\Ker(p)}, \Cat) \) is an equivalence (respectively,
      fully faithful),
    \item
      \label{enum:fd.desc}
      \( p \) is an effective \( F_D \)-descent morphism (respectively, \( F_D
      \)-descent morphism),
    \item
      \label{enum:kstar.set.eqv}
      \( \CAT(\mathcal K^{\Ker(p)}, \Set) \) is an equivalence (respectively,
      fully faithful),
    \item
      \label{enum:k.ffle}
      \( \mathcal K^{\Ker(p)} \) is a fully faithful lax epimorphism
      (respectively, a lax epimorphism),
    \item
      \label{enum:ck.eqv}
      \( \Cauchy \mathcal K^{\Ker(p)} \) is an equivalence (respectively, a
      lax epimorphism),
  \end{enumerate}
\end{theorem}

\begin{proof}
  By Theorem \ref{thm:cauchy} (respectively, Corollary \ref{cor:lax.epi}), we
  deduce that \ref{enum:k.ffle} \( \iff \) \ref{enum:ck.eqv} \( \iff \)
  \ref{enum:kstar.set.eqv}.

  Since both \( F \) and \( F_D \) preserve lax descent objects, we conclude
  by Lemma \ref{lem:lax.desc.preserve} that \ref{enum:f.desc} \( \iff \)
  \ref{enum:kstar.cat.eqv}  and \ref{enum:fd.desc} \( \iff \)
  \ref{enum:kstar.set.eqv}.

  Finally, by Lemma \ref{lem:star.set.cat} we obtain \ref{enum:kstar.cat.eqv}
  \( \iff \) \ref{enum:k.ffle}.
\end{proof}



\begin{spacing}{0.9}




\printbibliography[heading=bibintoc, title={References}]

\end{spacing}





\printthesisindex 

\end{document}